\documentclass[11pt]{amsart}

\usepackage[applemac]{inputenc}% Codage du fichier en ISO-Latin-1 (accents...)
\usepackage{xspace} % Espaces ins\`ecables avant les ponctuations doubles
\usepackage{mathrsfs} % pour pouvoir utiliser \mathscr{F} (peut-??tre plus joli que \mathcal)
\usepackage{amsfonts}
\usepackage{marvosym} % Quelques symboles utiles, comme l'euro (\EUR)
\usepackage{amsmath,amssymb,amsthm}% Formules math\`ematiques

\usepackage[usenames, dvipsnames, svgnames]{xcolor}

\usepackage[colorlinks=true,linkcolor=blue,urlcolor=blue,citecolor=blue,linktocpage=true,pagebackref]{hyperref} 
\usepackage{enumerate} % for using different counters
\usepackage{epsfig}
\usepackage[matrix,arrow]{xy}
\usepackage[matrix,arrow]{xy}%Utilisation des matrices (pour les diagrammes par exemple)
\usepackage{multicol}%Utilisation du multicolonage
\usepackage{array}%Utilisation avanc\`ee des tableaux
\usepackage{multirow}%Utilisation du multicolonnage dans les tableaux

\usepackage{fancyhdr} % En-t?tes personnalis\`es ; je les personnalise plus loin
\usepackage{fancybox} % Pour redimensionner ou tourner des blocs

\definecolor{puces}{cmyk}{0,0,0,.5} % D\`efinition d'une couleur "puces"

\newcommand{\titre}{}

\numberwithin{equation}{subsubsection}

\theoremstyle{plain}
\newtheorem{thm}{Theorem}[section]

\newtheorem{defi}[thm]{Definition}
\newtheorem{proposition}[thm]{Proposition}

\newtheorem{lem}[thm]{Lemma}
\newtheorem{coro}[thm]{Corollary}
\newtheorem{conj}[thm]{Conjecture}
\theoremstyle{remark}

\newtheorem{rem}[thm]{Remark}

\newtheorem{ex}[thm]{Example}

\newenvironment{demo}{\noindent{\textbf{Proof.}}}{\hfill \qedsymbol}

\newcommand{\C}{\mathbb{C}}
\newcommand{\R}{\mathbb{R}}
\newcommand{\qq}{\mathbb{Q}}
\newcommand{\FF}{\mathbb{F}}

\newcommand{\ZZ}{\mathbb{Z}}

\newcommand{\Spec}{\mathrm{Spec}}

\newcommand{\A}{\mathbb{A}}

\newcommand{\HH}{\mathrm{H}}

\newcommand{\ocal}{\mathcal{O}}
\newcommand{\oscr}{\mathscr{O}}

\newcommand{\B}{\mathrm{B}}

\newcommand{\U}{\mathrm{U}}

\newcommand{\F}{\mathbb{F}}

\newcommand{\ov}{\overline}
\newcommand{\mf}{\mathfrak}
\newcommand{\mr}{\mathrm}

\newif\iffinalrun
\iffinalrun
\else
 \fi

\iffinalrun
  \newcommand{\need}[1]{}
  \newcommand{\mar}[1]{}
\else
  \newcommand{\need}[1]{{\tiny *** #1}}
\newcommand{\mar}[1]{\marginpar{\raggedright\tiny FIXME: #1 }}\fi

\title{Hecke operators and the coherent cohomology of Shimura varieties}
\author{Najmuddin Fakhruddin}
\address{School of Mathematics, Tata Institute of Fundamental
  Research, Homi Bhabha Road, Mumbai 400005, INDIA}
\email{naf@math.tifr.res.in}
\author{Vincent Pilloni}
\address{CNRS, Ecole Normale Sup\'erieure de Lyon, Lyon, FRANCE}
\email{vincent.pilloni@ens-lyon.fr}
\date{}
\email{}

\begin{document}

\maketitle

\begin{abstract} We consider the problem of defining an action of
  Hecke operators on the coherent cohomology of certain integral
  models of Shimura varieties. We formulate a general conjecture
  describing which Hecke operators should act integrally and solve the
  conjecture in certain cases. As a consequence, we obtain $p$-adic
  estimates of Satake parameters of certain non-regular self dual
  automorphic representations of $\mathrm{GL}_n$.
\end{abstract}

\tableofcontents

\section{Introduction}

This paper studies the problem of defining an   action of Hecke operators on certain natural integral  coherent cohomology of Shimura varieties.

Let us start by describing the situation for modular curves. Let
$N \geq 4$ be an integer, and let $X \rightarrow \Spec~\ZZ[1/N]$ be
the compactified modular curve of level $\Gamma_1(N)$. Let $k \in \ZZ$
and let $\omega^k$ be the sheaf of weight $k$ modular forms. Let $p$
be a prime number, $(p,N) =1$.  When $k \geq 1$, there is a familiar
Hecke operator $T_p$ acting on the $\C$-vector space of weight $k$
modular forms. On $q$-expansions, the operator $T_p$ is given (on
forms of nebentypus a character
$\chi : \ZZ/N\ZZ^\times \rightarrow \overline{\ZZ}^\times$) by the
formula
$T_p(\sum_{n \geq 0} a_n q^n ) = \sum_{n\geq 0} a_{np} q^n +
p^{k-1}\chi(p) a_n q^{np}$.  This formula is integral and the
$q$-expansion principle implies that the action of $T_p$ actually
arises from an action on $\HH^0(X, \omega^k)$. We now give a more
geometric construction of $T_p$.  Assume first that we work over
$\ZZ[1/Np]$.  Then, viewing a modular form $f$ as a rule on triples
$(E, \alpha_N, \omega)$ where $E$ is an elliptic curve, $\alpha_N$ is
a point of order $N$, and $\omega$ is a nowhere vanishing differential
form (and satisfying a growth condition near the cusps), there is a
geometric formula defining $T_p$ \cite[Formula
1.11.0.2]{MR0447119}:
$$ T_p(f(E, \alpha_N, \omega)) = \frac{1}{p} \sum_{H \subset E[p]} f(E/H,  \alpha_N', \omega')$$ where $H \subset E[p]$ runs over all subgroups of $E[p]$ of order $p$, and if we let $\pi_H : E \rightarrow E/H$ be the  isogeny then $\pi_H(\alpha_N) = \alpha'_N$ and $\pi_H^\star \omega' = \omega$. 

The above formula does not really make sense over $\ZZ[1/N]$, but a slight modification of it will be meaningful. We  explain this modification. 

%but is a bit involved because the displayed formula does not really make sense in characteristic $p$ :
%\begin{enumerate}
%\item we have divided by $p$, 
% \item an elliptic curve in caracteristic $p$ has very few subgroups of order $p$,
% \item the isogeny  $\pi_H$ may not be Žtale and $\pi_H^\star \omega'$ could be zero.  
% \end{enumerate}
There is a correspondence $X_0(p)$ over $X$ corresponding to the  level $\Gamma_1(N) \cap \Gamma_0(p)$ (so $X_0(p)$ is a compactification of the moduli of triples $(E, \alpha_N, H)$ where $H \subset E[p]$ is a subgroup of order $p$), and there are two projections :
 $ p_1, p_2 : X_0(p) \rightarrow X$ (induced by $(E,H) \mapsto E$ and $(E,H) \mapsto E/H$). Then $T_p$ originates from a cohomological correspondence $T_p : p_2^\star \omega^k \rightarrow p_1^! \omega^k$, where $p_1^!$ is the right adjoint functor  to $(p_1)_\star$.

 The cohomological correspondence $T_p$ is obtained using the
 differential of the universal isogeny
 $\pi_H^\star : p_2^\star \omega \rightarrow p_1^\star \omega$, the
 trace of $p_1$, and a suitable normalization by a power of $p$ that
 makes everything integral (this corresponds to the coefficient
 $\frac{1}{p}$ in the displayed formula).  More precisely, we first
 define a (rational if $k \leq 0$) map over $X_0(p)$,
 $$T_p^{naive} : p_2^\star \omega^k \stackrel{(\pi_H^\star)^{\otimes
     k}}\longrightarrow p_1^\star \omega^k
 \stackrel{\mathrm{tr_{p_1}}}\longrightarrow p_1^! \omega^k$$ and then
 set $T_p = p^{-1} T_p^{naive}$ when $k\geq 1$ and
 $T_p = p^{-k} T_p^{naive}$ if $k \leq 1$.

The first basic result for elliptic modular forms is  the following
(see also \cite[\S 4.5]{MR2311664} and \cite[Proposition 3.11]{MR3725733}):
\begin{proposition} For any $k \in \ZZ$, we have a   cohomological correspondence $T_p : p_2^\star \omega^k \rightarrow p_1^! \omega^k$ over $X_0(p)$, inducing a Hecke operator 
$T_p \in \mathrm{End} ( \mathrm{R}\Gamma(X, \omega^k))$
\end{proposition}

Given the cohomological correspondence, the operator on the cohomology is simply obtained using pullback and pushforward as follows: 
$$T_p :  \HH^\star(X, \omega^k) \rightarrow \HH^\star(X_0(p), p_2^\star \omega^k) \stackrel{T_p} \longrightarrow  \HH^\star(X_0(p), p_1^! \omega^k)\rightarrow  \HH^\star(X, \omega^k).$$
The action of $T_p$ on formal $q$-expansions (of level $\Gamma_1(N)$ and Nebentypus $\chi$) is given by :
\begin{itemize}
\item $T_p(\sum_{n \geq 0} a_n q^n ) = \sum_{n\geq 0} a_{np} q^n + p^{k-1}\chi(p) a_n q^{np}$ if $k \geq 1$,
\item $T_p(\sum_{n \geq 0} a_n q^n ) = \sum_{n\geq 0} p^{1-k}a_{np} q^n + \chi(p) a_n q^{np}$ if $k \leq 1$,
\end{itemize}
therefore $T_p$ appears to be optimally integral.

  %We first explain the case of Hilbert modular varieties. Let $F$ be a totally real field in which $p$ is unramified. Let $K \subset \mathrm{GL}_2(\mathbb{A}_f)$ be a compact open subgroup. We let ${Sh}_K \rightarrow \Spec~\qq$  be the corresponding Hilbert modular variety of level $K$.  We assume that $K = K^pK_p$ with $K_p$ hyperspecial.  The Hilbert modular variety has an integral model  $\mathfrak{Sh}_K \rightarrow \Spec~\ZZ_{(p)}$ and subject to the choice of certain polyhedral decompositions, we have toroidal compactifications $\mathfrak{Sh}_{K,\Sigma} \rightarrow \Spec~\ZZ_{(p)}$.  We let $E'$ be the Galois closure of $F$ and fix a place $\lambda'$ above $p$ in $E'$. We now consider $\mathfrak{Sh}_K$ over $\Spec~\ocal_{E, \lambda'}$. 

% the ring of integers of a  finite extension of $\ZZ_p$ 

%. The first natural question, is to understand the normalization factors in a systematic way. 

\bigskip

One would like to generalize this to other Shimura varieties. To do
this, we introduce some notations and recall the fundamental results
concerning the cohomology of automorphic vector bundles on Shimura
varieties. A standard reference for this material is
\cite{harris-ann-arb}.

Let $(G,X)$ be a Shimura datum, let $K \subset G(\mathbb{A}_f)$ be a
neat compact open subgroup, and let $Sh_K $ be the associated Shimura
variety. This is a smooth scheme defined over the reflex field $E$.
Over $Sh_K$ there is a large supply of automorphic vector bundles
$\mathcal{V}_{\kappa,K}$, naturally parametrized by weights $\kappa$
of a maximal torus in $G$, dominant for the compact roots. For a
choice of polyhedral cone decomposition $\Sigma$, one can construct a
toroidal compactification $Sh^{tor}_{K,\Sigma}$ of $Sh_K$, such that
$D_{K,\Sigma} = Sh^{tor}_{K,\Sigma} \setminus Sh_{K,\Sigma}$ is a
Cartier divisor.
% {\color{red} Say that the compactification is smooth
%   for suitable choices ? maybe not so useful there}
Moreover, there is a canonical extension
$\mathcal{V}_{\kappa, K,\Sigma}$ of the vector bundle
$\mathcal{V}_{\kappa, K}$, as well as a sub-canonical extension
$\mathcal{V}_{K,\Sigma}(-D_{K,\Sigma})$.

The coherent cohomology complexes we are interested in (which generalize
the classical notion of modular forms) are:
$$\mathrm{R}\Gamma(Sh^{tor}_{K,\Sigma}, \mathcal{V}_{\kappa,
  K,\Sigma})~~ \textrm{and} ~\mathrm{R}\Gamma(Sh^{tor}_{K,\Sigma},
\mathcal{V}_{\kappa, K,\Sigma}(-D_{K,\Sigma})).$$ These cohomology
complexes are independant of $\Sigma$, they have good functorial
properties in the level $K$ and they carry an action of the Hecke
algebra
$\mathcal{H}_K = \mathcal{C}^\infty_c( K \backslash G(\mathbb{A}_f)/K,
\ZZ)$.

Over $\C$ this coherent cohomology can be computed via the
$(\mathfrak{p},K)$-cohomology of the space of automorphic forms on the
group $G$ \cite{JunSu}. Automorphic representations contributing to
these cohomology groups therefore possess a rational structure.  We
now fix a prime $p$.  In order to study $p$-adic properties of
automorphic forms contributing to the coherent cohomology, we
introduce integral structures on
$\mathrm{R}\Gamma(Sh^{tor}_{K,\Sigma}, \mathcal{V}_{\kappa,
  K,\Sigma})$ and
$\mathrm{R}\Gamma(Sh^{tor}_{K,\Sigma}, \mathcal{V}_{\kappa,
  K,\Sigma}(-D_{K,\Sigma}))$.

Let $\lambda$ be a place of $E$ above $p$. If the Shimura datum
$(G,X)$ is of abelian type, $G$ is unramified at $p$, and $K= K^pK_p$,
where $K^p \subset G(\mathbb{A}_f^p)$ and $K_p \subset G(\qq_p)$ is
hyperspecial, we have natural integral models $\mathfrak{Sh}_{K}$ for
$Sh_{K}$, as well as an integral model of $\mathcal{V}_{\kappa, K}$,
over $\Spec~\ocal_{E, \lambda}$ \cite{MR1304906}, \cite{MR1124982},
\cite{MR2669706}, \cite{MR3569319}. Moreover, there is a theory of
toroidal compactifications (at least in the Hodge case)
$\mathfrak{Sh}^{tor}_{K,\Sigma}$ of $\mathfrak{Sh}_{K}$, with integral
canonical extension $\mathcal{V}_{\kappa, K,\Sigma}$, and integral
sub-canonical extension
$\mathcal{V}_{\kappa,
  K,\Sigma}(-D_{K,\Sigma})$
\cite{MR1083353}, \cite{MR3186092}, \cite{MR2968629},
\cite{MR3948111}.

We are therefore led to consider the cohomology complexes
$\mathrm{R}\Gamma(\mathfrak{Sh}^{tor}_{K,\Sigma}, \mathcal{V}_{\kappa,
  K,\Sigma})$ as well as
$\mathrm{R}\Gamma(\mathfrak{Sh}^{tor}_{K,\Sigma}, \mathcal{V}_{\kappa,
  K,\Sigma}(-D_{K,\Sigma}))$.  Our work investigates the question of
extending the action of the Hecke algebra $\mathcal{H}_K$ to these
cohomology groups. The action of the prime to $p$ Hecke algebra
extends without much difficulty, so our main task is to investigate
the action of (a suitable sub-algebra of) the local Hecke algebra
$\mathcal{C}^\infty_c( K_p \backslash G(\qq_p)/K_p,
\qq)$ on
$\mathrm{R}\Gamma(\mathfrak{Sh}^{tor}_{K,\Sigma}, \mathcal{V}_{\kappa,
  K,\Sigma})$.

We formulate a conjecture (see Conjecture \ref{conj3}, as well as
Conjectures \ref{conj1} and \ref{conj2}) precisely describing a sub
$\ZZ$-algebra $\mathcal{H}_{p,\kappa, \iota}^{int}$ (where $\iota$ is
an isomorphism of $\C$ and $\ov{\qq}_p$) of
$\mathcal{C}^\infty_c( K_p \backslash G(\qq_p)/K_p, \qq)$ which should
act on the cohomology complexes
$\mathrm{R}\Gamma(\mathfrak{Sh}^{tor}_{K,\Sigma},
\mathcal{V}_{K,\Sigma})$ and
$\mathrm{R}\Gamma(\mathfrak{Sh}^{tor}_{K,\Sigma},
\mathcal{V}_{K,\Sigma}(-D_{K,\Sigma}))$.

The definition of  $\mathcal{H}_{p,\kappa,\iota}^{int}$ is given in Definition \ref{defi-hecke-algebra}, with the help of the  Satake basis $[V_\lambda]$ of $\mathcal{H}_p$  (where $\lambda$ runs through the set of dominant and Galois invariant chararacters of the dual group $\hat{G}$ of $G$) normalized by a certain power of $p$ determined by $\lambda$ and the weight $\kappa$.

In the modular curve case, the description of this sub-algebra
precisely reflects the normalizing factors $p^{-\inf\{1,k\}}$ in the
definition of the $T_p$-operator on the cohomology in weight $k$.

This conjecture is a translation for the coherent cohomology of
Shimura varieties of results obtained by V.~Lafforgue on the Betti
cohomology of locally symmetric spaces in \cite{MR2869300}. It is
inspired by Katz--Mazur inequality: For $X$ a proper smooth scheme
defined over $\ZZ_p$, Katz and Mazur gave $p$-adic estimates for the
eigenvalues of the geometric Frobenius acting on the cohomology
$\HH^i(X_{\overline{\qq}_p}, \qq_\ell)$ ($\ell \neq p$). The estimates
say that $p$-adic Newton polygon of the characteristic polynomial of
Frobenius lies above the Hodge polygon (determined by the filtration
on the de Rham cohomology of $X_{{\qq}_p}$).

One believes that similar estimates hold for algebraic automorphic
representations because of the conjectural correspondence between
motives and automorphic representations \cite{MR1044819},
\cite{MR3444225}. Roughly speaking, the weight $\kappa$ determines the
Hodge polygon, while the characters of $\mathcal{H}_{p}$ determine the
characteristic polynomials of Frobenii. We postulate the Hodge--Newton
inequality to determine the precise shape of the maximal sub-algebra
$\mathcal{H}_{p,\kappa, \iota}^{int} \subset \mathcal{H}_p$ which
should act integrally on the cohomology.

We then try to prove our conjecture in certain special cases.  To the
charateristic functions of  {all} double cosets
$K_p g K_p$ in the local Hecke algebra, one can associate Hecke
correspondences $p_1, p_2: Sh_{K\cap g K g^{-1}} \rightarrow Sh_K$
over $Sh_K$. These correspondences rarely admit integral models whose
geometry is understood, except when $g$ is associated with a minuscule
coweight. In this case, $K_p\cap g K_p g^{-1}$ is a parahoric subgroup
and there is a good theory of integral models whose local geometry is
described by the local model theory.  We are thus led to work with
Hecke operators associated to minuscule coweights (this is a serious
restriction).

At this point a second obstacle arises: the projections $p_i$ almost
never extend integrally to finite flat morphisms. This means that
defining the necessary trace maps in cohomology is in principle
complicated. We develop, using Grothendieck--Serre duality
\cite{Hartshorne}, a formalism of cohomological correspondences in
coherent cohomology which solves the problem under some assumptions on
the correspondences. In particular, we assume that our correspondences
are given by Cohen--Macaulay schemes. This suffices for our pusposes
since using the theory of local models (see, e.g., \cite{PRS}), one
can often prove that integral models of Shimura varieties with
parahoric level structure are Cohen--Macaulay.

In order to prove Conjecture \ref{conj3}, our strategy is to switch to
the local model, where we can make all the computations, and then
transfer back the information to the Shimura variety.

We completely solve our conjecture for Hilbert modular varieties
($G = \mathrm{Res}_{F/\qq} \mathrm{GL}_2$), when $p$ is unramified in
the totally real field $F$. To formulate the result precisely in this
case, let $E'$ be the finite extension of $E=\qq$ equal to the Galois
closure of $F$.  Weights for Hilbert modular forms are tuples of
integers $\kappa = ((k_\sigma)_{\sigma \in \mathrm{Hom}(F, E')}; k)$
where $k_\sigma$ and $k$ all have the same parity.  We let
$\iota : E' \rightarrow \overline{\qq}_p$ be an embedding. Let
$p = \prod_i \mathfrak{p}_i$ be the decomposition of $p$ in $\ocal_F$
has a product of prime ideals. %Let $\iota : E' \rightarrow \ov{\qq}_p$
%be an embedding.

\begin{thm}[Theorem \ref{main-thm-symplectic} and \S
  \ref{sectionHilbert}] The Hecke algebra
  $\mathcal{H}_{p, \kappa, \iota}^{int} = \otimes_i
  \ZZ[T_{\mathfrak{p}_i}, S_{\mathfrak{p}_i},
  S_{\mathfrak{p}_i}^{-1}]$ acts on
  $\mathrm{R}\Gamma(\mathfrak{Sh}^{tor}_{K,\Sigma},
  \mathcal{V}_{\kappa, K,\Sigma})$ and
  $\mathrm{R}\Gamma(\mathfrak{Sh}^{tor}_{K,\Sigma},
  \mathcal{V}_{\kappa, K,\Sigma}(-D_{K,\Sigma}))$.
\end{thm}

Under the assumption that the weight $\kappa$ belongs to a certain
cone (conjecturally the ample cone), the theorem was first proved in
\cite[Proposition 3.11]{MR3725733} by different methods.

We can be explicit about normalization factors. Let
$I_i = \{ \sigma \in \mathrm{Hom}(F, E'),~\iota \circ
\sigma(\mathfrak{p}_i) \subset \mathfrak{m}_{\ov{\ZZ}_p}\}$ and
let
$$T_{\mathfrak{p}_i}^{naive} = \mathrm{GL}_{2}(
\ocal_{F_{\mathfrak{p}_i}}) \mathrm{diag} ( \mathfrak{p}^{-1}_i,1)
\mathrm{GL}_{2}( \ocal_{F_{\mathfrak{p}_i}})$$
and
$$ S_{\mathfrak{p}_i}^{naive} = \mathrm{GL}_{2}(
\ocal_{F_{\mathfrak{p}_i}}) \mathrm{diag} (\mathfrak{p}^{-1}_i,
\mathfrak{p}^{-1}_i) \mathrm{GL}_{2}( \ocal_{F_{\mathfrak{p}_i}})$$ be the
naive, ``unnormalized''  version of the Hecke operators attached to the
familiar double classes (the reason we have to use these double classes and not their inverse is that we set up the theory in such a way that the  Hecke algebra acts naturally on the left and not on the right on the cohomology). Our normalized operators are:
$$T_{\mathfrak{p}_i} = p^{\sum_{\sigma \in I_i} \sup\{ \frac{k_\sigma
    +k}{2}-1, \frac{k-k_\sigma}{2}\}} T_{\mathfrak{p}_i}^{naive}$$
and
$$S_{\mathfrak{p}_i} = p^{\sum_{\sigma \in I_i} k}
S_{\mathfrak{p}_i}^{naive}.$$

For more general Shimura varieties of symplectic type we only have
partial results because there is essentially only one minuscule
coweight (see Theorem \ref{main-thm-symplectic}). However,  the situation is
better for unitary Shimura varieties.  Let $F$ be a totally real
field, and let $L$ be a totally imaginary quadratic extension of $F$. We let
$G \subset \mathrm{Res}_{L/\qq} \mathrm{GL}_n$ be a unitary group of
signature $(p_\tau, q_\tau)_{\tau : F \hookrightarrow \mathbb{C}}$. We assume that $p$ is unramified in $L$ and we let  $\mathfrak{Sh}^{tor}_{K,\Sigma}$ be a toroidal compactification of the  (smooth) integral model of the unitary Shimura variety attached to $G$.

\begin{thm}[Theorem \ref{main-thm-unitary}] Assume that all finite
  places $v \mid p$ in $F$ split in $L$. Let
  $\mathcal{H}_{p, \kappa, \iota}^{int} = \otimes_{0 \leq i \leq m}
  \ZZ[ T_{\mathfrak{p}_i,j},~0 \leq j \leq n,
  T_{\mathfrak{p}_i,0}^{-1}, T_{\mathfrak{p}_i,n}^{-1}]$ be the
  normalized integral Hecke algebra in weight $\kappa$ (where the
  $T_{\mathfrak{p}_i,j}$ are standard generators of this Hecke
  algebra).  All the operators $T_{\mathfrak{p}_i,j}$ act on
  $\mathrm{R}\Gamma(\mathfrak{Sh}^{tor}_{K,\Sigma},
  \mathcal{V}_{\kappa, K,\Sigma})$ and
  $\mathrm{R}\Gamma(\mathfrak{Sh}^{tor}_{K,\Sigma},
  \mathcal{V}_{\kappa, K,\Sigma}(-D_{K,\Sigma}))$.  In particular,
  $\mathrm{Im} \big(\HH^i(\mathfrak{Sh}^{tor}_{K,\Sigma},
  \mathcal{V}_{\kappa, K,\Sigma}) \rightarrow
  \HH^i(\mathfrak{Sh}^{tor}_{K,\Sigma}, \mathcal{V}_{\kappa,
    K,\Sigma}) \otimes_{\ZZ_p}\qq_p \big)$ is a lattice which is
  stable under the action of $\mathcal{H}_{p, \kappa, \iota}^{int}$.
\end{thm}

\begin{rem} The limitation of  this last result is that we have not been able to prove that the operators $T_{\mathfrak{p}_i,j}$ commute with each other in general. 
\end{rem}

Let $L$ be a CM or a totally real number field. One can realize
regular algebraic essentially (conjugate) self-dual cuspidal
automorphic representations of $\mathrm{GL}_n/L$ in the Betti
cohomology of Shimura varieties. The interest of coherent cohomology
is that it captures more automorphic representations.  Namely, one can
weaken the regularity condition to a condition that we call weakly
regular odd (the oddness property is automatically satisfied in the
regular case).  Weakly regular, algebraic, odd, essentially (conjugate) self dual, cuspidal automorphic representation on $\mathrm{GL}_n/L$
admit a compatible system of Galois representations, but at the moment
these compatible systems are not known to be de Rham in general (and
local-global compatibility is not known). We can nevertheless prove
the following result which is to be viewed as the Katz--Mazur
inequality.

\begin{thm}[Theorem \ref{odd}]
  Let $\pi$ be a weakly regular, algebraic, odd, essentially (conjugate) self dual,  cuspidal automorphic representation of $\mathrm{GL}_n/L$ with  infinitesimal character $\lambda = (\lambda_{i, \tau}, 1\leq i \leq n, \tau \in \mathrm{Hom} (L, \overline{\qq}))$ and $\lambda_{1,\tau} \geq \cdots \geq \lambda_{n,\tau}$.  Let $p$ be a prime unramified in $L$ and $w$ be a finite place  of $L$ dividing $p$. Assume also that $\pi_w$ is spherical, and corresponds to a semi-simple conjugacy class $\mathrm{diag}(a_1,\cdots, a_n) \in \mathrm{GL}_n(\overline{\qq})$ by the Satake isomorphism. We let $\iota: \overline{\qq} \rightarrow \overline{\qq}_p$ be an embedding and $v$ the associated $p$-adic valuation normalized by $v(p)=1$.  After permuting we assume that $v(a_1) \leq \cdots \leq v(a_n)$. Let $I_w \subset \mathrm{Hom} (L, \overline{\qq})$ be the set of embeddings $\tau$ such that $\iota \circ \tau$ induces the $w$-adic valuation on $L$. 
 Then  we have $$\sum_{i=1}^k v(a_i) \geq \sum_{\tau \in I_w}
 \sum_{\ell=1}^k - \lambda_{\ell, \tau},$$ for $1 \leq k \leq n$, with equality if $k= n$. 
 \end{thm}

 \subsection{Organisation of the paper} In Section \ref{residue} we
 develop a formalism of cohomological correspondences in coherent
 cohomology. In Section \ref{section-Satake} we give a number of
 classical results concerning the structure of the local spherical
 Hecke algebra of an unramified group. In Section
 \ref{section-Shimura}, we introduce Shimura varieties and their
 coherent cohomology and formulate Conjecture \ref{conj3} on the
 action of the integral Hecke algebra on the integral coherent
 cohomology. In Section \ref{section-symplectictype} and
 \ref{section-local-model-symp} we consider the case of Shimura
 varieties of symplectic type and their local models. In Section
 \ref{section-unitary} and \ref{sect-local-model-unitary} we consider
 the case of Shimura varieties of unitary type and their local
 models. Finally, the last section deals with applications to
 automorphic representations and Galois representations.

 \subsection{Acknowledgements} We would like to thank Arvind Nair, George Boxer and
 Sandeep Varma for useful conversations. We would like to thank
 Olivier Taibi for directing us to the paper \cite{MR2869300} and for
 many useful exchanges and suggestions. We would like to thank the TIFR and the ENS of
 Lyon for hosting us during this project. The first author would also
 like to thank Laurent Fargues for supporting his visit to Lyon in
 2018 during the GeoLang programme via his ERC grant. The second
 author has been supported by the ANR Percolator, and the
 ERC-2018-COG-818856-HiCoShiVa.

\section{Correspondences and coherent cohomology}\label{residue}
\subsection{Preliminaries on residues and duality}

We start by recalling some results of Grothendieck duality theory for
coherent cohomology. The original reference for this, which we use
below, is \cite{Hartshorne}; although this and \cite{conrad-bc}, which
is based on it, suffice for our purposes, the more abstract approaches
of \cite{verdier-base-change}, \cite{neeman-duality} and \cite{lipman}
give more general results with arguably more efficient proofs. In
particular, the latter two references show that the noetherian and
finite Krull dimension hypotheses of \cite{Hartshorne} can be
eliminated and that most results extend to the unbounded derived
category.

For a scheme $X$ we let $\mathbf{D}_{qcoh}(\oscr_X)$ be the
subcategory of the derived category $\mathbf{D}(\oscr_X)$ of
$\oscr_X$-modules whose objects have quasi-coherent cohomology
sheaves. We let $\mathbf{D}^{+}_{qcoh}(\oscr_X)$
(resp. $\mathbf{D}^{-}_{qcoh}(\oscr_X)$) be the full subcategory of
$\mathbf{D}_{qcoh}(\oscr_X)$ whose objects have $0$ cohomology sheaves
in sufficiently negative (resp. positive) degree. We let
$\mathbf{D}^{b}_{qcoh}(\oscr_X)$ be the full subcategory of
$\mathbf{D}_{qcoh}(\oscr_X)$ whose objects have $0$ cohomology sheaves
for all but finitely many degrees. We remark that if $X$ is locally
notherian $\mathbf{D}^{+}_{qcoh}(\oscr_X)$ is also the derived
category of the category of bounded below complexes of quasi-coherent
sheaves on $X$ \cite[I, Corollary 7.19]{Hartshorne}.  We let
$\mathbf{D}^{b}_{qcoh}(\oscr_X)_{fTd}$ be the full subcategory of
$\mathbf{D}^{b}_{qcoh}(\oscr_X)$ whose objects are quasi-isomorphic to
bounded complexes of flat sheaves of $\oscr_X$-modules \cite[II,
Definition 4.13]{Hartshorne}. We fix for the rest of this section a
noetherian affine scheme
$S$.

\subsubsection{Embeddable morphisms} Let $X$, $Y$ be two $S$-schemes
and $f : X \rightarrow Y$ be a morphism of $S$-schemes.  The morphism
$f$ is embeddable if there exists a smooth $S$-scheme $P$ and a finite
map $ i : X \rightarrow P \times_S Y$ such that $f$ is the composition
of $i$ and the second projection \cite[p.~189]{Hartshorne}. The
morphism $f$ is projectively embeddable if it is embeddable and $P$
can be taken to be a projective space over $S$
\cite[p.~206]{Hartshorne}.

\subsubsection{The functor $f^!$}
For $f : X \rightarrow Y$ a morphism of $S$-schemes, there is a
functor
$Rf_\star : \mathbf{D}_{qcoh}(\oscr_X) \rightarrow
\mathbf{D}_{qcoh}(\oscr_Y)$. By \cite[III, Theorem 8.7]{Hartshorne} if $f$ is
embeddable, we can define a functor
$f^! : \mathbf{D}^+_{qcoh}(\oscr_Y) \rightarrow
\mathbf{D}^+_{qcoh}(\oscr_X)$. If $f$ is projectively embeddable, by
\cite[III, Theorem 10.5]{Hartshorne} there is a natural transformation (trace
map) $Rf_\star f^! \Rightarrow Id$ of endofunctors of
$\mathbf{D}^+_{qcoh}(\oscr_Y)$.  Moreover, by \cite[III, Theorem 11.1]{Hartshorne} this natural transformation induces a duality isomorphism:

\begin{equation} \label{eq:duality}
  \mathrm{Hom}_{\mathbf{D}_{qcoh}(\oscr_X)} ( \mathscr{F}, f^! \mathscr{G}) 
\stackrel{\sim}\rightarrow \mathrm{Hom}_{\mathbf{D}_{qcoh}(\oscr_Y)}(
R f_\star \mathscr{F}, \mathscr{G})
\end{equation}
for $\mathscr{F} \in
\mathbf{D}^{-}_{qcoh}(\oscr_X)$ and $\mathscr{G} \in \mathbf{D}^+_{qcoh}(\oscr_Y)$.

The functor $f^!$ for embeddable morphism enjoys many good
properties. Let us record one that will be crucially used.
\begin{proposition}\label{prop-finite-tor-dim} \cite[III, Proposition
  8.8]{Hartshorne} If $\mathscr{F} \in \mathbf{D}^+_{qcoh}(\oscr_Y)$
  and $\mathscr{G} \in \mathbf{D}^b_{qcoh}(\oscr_Y)_{fTd}$, there is a
  functorial isomorphism
  $f^!  \mathscr{F} \otimes^L Lf^\star \mathscr{G} = f^!( \mathscr{F}
  \otimes^L \mathscr{G})$.
\end{proposition}

\subsection{Cohen--Macaulay, Gorenstein and lci morphisms}
\subsubsection{Cohen--Macaulay morphisms} \label{CM-maps}

Recall that a morphism $h:X \to S$ is called Cohen--Macaulay
(abbreviated to \emph{CM} for the rest of this section) if $h$ is
flat, locally of finite type and has CM fibres.

\begin{lem}\label{lem-CM} For a CM morphism of pure relative dimension $n$,
$h^!\oscr_S = \omega_{X/S}[n]$, where $\omega_{X/S}$ is a coherent
sheaf which is flat over $S$.
Assume moreover that $S$ is CM. Then $\omega_{X/S}$ is a $CM$ sheaf. 
\end{lem}

\begin{proof} For the first point, see \cite[ Theorem
  3.5.1]{conrad-bc}. For the second point, first observe that if $S$
  is Gorenstein (e.g., $S$ is the spectrum of a field), so
  $\oscr_S[0]$ is a dualizing complex for $S$, then
  $\omega_{X/S}[n] = h^! \oscr_S$ is a dualizing complex for $X$
  \cite[V, \S \S 2,10]{Hartshorne}. Using local duality \cite[Theorem
  1.2.8]{bruns1998cohen}, it follows from \cite[Corollary
  3.5.11]{bruns1998cohen} that $\omega_{X/S}$ is a CM sheaf.  In
  general, for a CM map $h$ as above, the formation of $\omega_{X/S}$
  is compatible with base change \cite[Theorem 3.6.1]{conrad-bc}, so
  it follows from \cite[Proposition 6.3.1]{ega-4-ii} that
  $\omega_{X/S}$ is a CM sheaf.

\end{proof}

\subsubsection{Gorenstein morphisms} A local ring $A$ is Gorenstein if it
satisfies the equivalent properties of \cite[V, Theorem
9.1]{Hartshorne} (for instance, if $A$ is a dualizing complex for
itself). Note that local complete intersection rings, in particular
regular rings, are Gorenstein. Furthermore, Gorenstein rings are CM.
A locally noetherian scheme is Gorenstein if all its local rings are
Gorenstein. Finally, a morphism $h:X \to S$ is Gorenstein if it is
flat and all its fibres are Gorenstein.

\begin{lem} For a proper Gorenstein morphism $h:X \to S$ of
  pure relative dimension $n$, $h^!\oscr_S = \omega_{X/S}[n]$, where
  $\omega_{X/S}$ is an invertible sheaf on $X$.
  \end{lem}

\begin{proof} A Gorenstein morphism is CM so by Lemma \ref{lem-CM},
  $h^!\oscr_S = \omega_{X/S}[n]$, where $\omega_{X/S}$ is a coherent
  sheaf on $X$. Since for a CM map $h$ the formation of $\omega_{X/S}$
  is compatible with base change \cite[Theorem 3.6.1]{conrad-bc}, for
  any point $s \in S$ $\omega_{X/S} \vert_{X_s} = \omega_{X_s/s}$. By
  \cite[V, Proposition 2.4, Theorem 8.3]{Hartshorne} (and the remark
  after Theorem 8.3), $\omega_{X_s/s}$ is a dualizing sheaf for
  $X_s$. By \cite[V, Theorem 3.1]{Hartshorne}, the Gorenstein
  hypothesis on $X_s$ implies that $\omega_{X_s/s}$ is an invertible
  sheaf. Since $\omega_{X/S}$ is flat over $S$ by Lemma \ref{lem-CM},
  it is flat over $X$ by the fibrewise flatness criterion
  \cite[Th\'eor\`eome 11.3.10]{MR0217086}, so it is an invertible
  sheaf over $X$.
\end{proof}

\subsubsection{Local complete intersection morphisms} A morphism
$h : X \rightarrow S$ is called a local complete intersection
(henceforth \emph{lci}) morphism if locally on $X$ we have a
factorization $ h : X \stackrel{i}\rightarrow Z \rightarrow S$ where
$i$ is a regular immersion \cite[D\'efinition 16.9.2]{MR0217086}, and
$Z$ is a smooth $S$-scheme. If $h$ is lci, then the cotangent complex
of $h$, denoted by $\mathbb{L}_{X/S}$, is a perfect complex
concentrated in degree $-1$ and $0$ \cite[Proposition
3.2.9]{MR0491680}. Its determinant in the sense of \cite{MR0437541} is
denoted by $\omega_{X/S}$.

\begin{lem} If $h : X \rightarrow S$ is an embeddable morphism and a local complete intersection of pure relative dimension $n$, then $h^! \oscr_{X} = \omega_{X/S}[n]$ where $\omega_{X/S}$ is the determinant of the cotangent complex $\mathbb{L}_{X/S}$.
\end{lem}
\begin{demo} This follows from the very definition of $h^!$ given in
  \cite[III, Theorem 8.7]{Hartshorne}. 
\end{demo} 

\subsection{An application: construction of a trace map} 
% Let $X, Y$ be two embeddable $S$-schemes and let $f : X \rightarrow Y$
% be an embeddable morphism.  

\begin{proposition} Let $X, Y$ be two embeddable $S$-schemes and let
  $f : X \rightarrow Y$ be an embeddable morphism.  Assume that
  $X \rightarrow S$ is $CM$ and that $Y \rightarrow S$ is
  Gorenstein. Assume that $X$ and $Y$ have the same pure relative
  dimension over $S$. Then $f^! \oscr_Y : = \omega_{X/Y}$ is a
  coherent sheaf. If moreover $S$ is assumed to be CM, then
  $\omega_{X/S}$ is a CM sheaf. If $X$ and $Y$ are smooth over $S$,
  $\omega_{X/S} = \det \Omega^{1}_{X/S} \otimes_{f^{-1}\oscr_Y}
  (f^{-1}\det \Omega^{1}_{Y/S})^{-1}$.
\end{proposition}

\begin{proof} 
  We have $h^! \oscr_{S} = \omega_{X/S}[n]$. On the other hand,
\begin{eqnarray*}
h^! \oscr_{S}& = &f^!(g^! \oscr_S)\\
& =& f^! (\omega_{Y/S}[n]) \\
& = & f^!( \oscr_Y \otimes \omega_{Y/S}[n]) \\
&= &f^! ( \oscr_Y) \otimes f^\star \omega_{Y/S}[n].
\end{eqnarray*}
\end{proof}

We observe that by adjunction we have a universal trace map $R f_\star \omega_{X/Y} \rightarrow \oscr_{Y}$. In particular for any section $s \in \HH^0( X, \omega_{X/Y})$ (or equivalently morphism $\oscr_{X} \rightarrow \omega_{X/Y}$), we get a corresponding trace $\mathrm{Tr}_s : R f_\star \oscr_{X} \rightarrow \oscr_Y$. 

\begin{proposition}\label{prop-trace} Let $X, Y$ be two embeddable $S$-schemes and let $f : X \rightarrow Y$
  be an embeddable morphism.  Assume that:
\begin{enumerate}
\item   $X \rightarrow S$ is $CM$, 
\item $Y \rightarrow S$ is Gorenstein, 
\item $S$ is $CM$, 
\item   $X$ and $Y$ have the  same pure relative dimension over $S$,
\item there are open sets $V \subset X$, $U \subset Y$ such that $f(V) \subset U$,
  $U$ and $V$ smooth over $S$ and $X \backslash V$ is of codimension $2$ in $X$. 
 \end{enumerate}
  Then there is a canonical morphism $\Theta : \oscr_{X} \rightarrow \omega_{X/Y}$ called the fundamental class. 
 Moreover if $W \subset Y$ is an open subscheme and $X\times_Y W \rightarrow W$ is finite flat, then the trace $\mathrm{Tr}_\Theta$ restricted to $W$ is the usual trace map for the finite flat morphism $X\times_Y W \rightarrow W$.
 \end{proposition}
 
 \begin{proof} It is enough to specify the fundamental class over $V$
   because it will extend to all of $X$ by Lemma \ref{lem-CM}.  Then
   over $V$, we have a map
   $\mathrm{d} f : f^{\star}\Omega^{1}_{U/S} \rightarrow
   \Omega^{1}_{V/S}$ and we define the fundamental class as the
   determinant
   $\det (\mathrm{d} f) \in \HH^0(V, \det \Omega^{1}_{V/S} \otimes
   (f^\star\det \Omega^{1}_{U/S})^{-1})$.  To prove the second claim,
   we can assume that $X,Y$ are smooth over $S$ and the map
   $X \rightarrow Y$ is finite flat. In this situation,
   $X\rightarrow Y$ is lci.  We claim that the cotangent complex
   $\mathbb{L}_{X/Y}$ is represented by the complex in degree $-1$ and
   $0$ :
   $[f^\star\Omega^1_{Y/S} \stackrel{\mathrm{d} f}\rightarrow
   \Omega^1_{X/S}]$, and that the determinant
   $\mathrm{det} (\mathrm{d}f ) \in \HH^0(X, \omega_{X/Y}) =
   \mathrm{Hom}(\oscr_{X}, f^! \oscr_{Y})$ is the trace map $tr_f$.
   We have a closed embedding $ i : X \hookrightarrow X\times_S Y$ of
   $X$ into the smooth $Y$-scheme $X\times_S Y$.  We have an exact
   sequence:
$$0 \rightarrow \mathcal{I}_Y  \rightarrow \oscr_{Y \times_S Y}
\rightarrow \oscr_Y \rightarrow 0$$ which gives after tensoring with
$\oscr_{X}$ above $\oscr_{Y}$
$$0 \rightarrow \mathcal{I}_X  \rightarrow \oscr_{X \times_S Y}
\rightarrow \oscr_X \rightarrow 0$$ where $\mathcal{I}_{X}$ is the
ideal sheaf of the immersion $i$. It follows that
$\mathcal{I}_X/\mathcal{I}_X^2 = \mathcal{I}_Y/\mathcal{I}_Y^2
\otimes_{\oscr_{Y}} \oscr_X = \Omega^1_{Y/S} \otimes_{\oscr_{Y}}
\oscr_X$.

On the other hand, $i^\star \Omega^{1}_{X \times_S Y/Y} =
\Omega^{1}_{X/S}$.  The cotangent complex is represented by $[
\mathcal{I}_{X}/\mathcal{I}_{X}^2 \rightarrow i^\star \Omega^{1}_{X
  \times_S Y/Y} ]$ which is the same as $[f^\star\Omega^1_{Y/S}
\rightarrow \Omega^1_{X/S}]$.

The morphism
$ f_\star \det \mathbb{L}_{X/Y} = \underline{\mathrm{Hom}} ( f_\star
\oscr_{X}, \oscr_Y) \rightarrow \oscr_Y$ is the residue map which
associates to $\omega \in f_\star \Omega^1_{X/S}$ and to
$(t_1, \cdots, t_n)$ local generators of the ideal $\mathcal{I}_X$
over $Y$ the function $\mathrm{Res} [\omega, t_1,...,t_n]$. It follows
from \cite[Property (R6), p.~198]{Hartshorne} that the determinant
of $[f^\star\Omega^1_{Y/S} \rightarrow \Omega^1_{X/S}]$ maps to the
usual trace map.
\end{proof}

\subsubsection{Fundamental class and divisors} We need a slight variant of the last proposition, where we also have some divisors.
\begin{proposition} Let $X, Y$ be two embeddable $S$-schemes and let
  $f : X \rightarrow Y$ be an embeddable morphism.  Let
  $D_X \hookrightarrow X$ and $D_Y \hookrightarrow Y$ be relative
  (with respect to $S$) effective Cartier divisors.

  Assume that:
\begin{enumerate}
\item   $X \rightarrow S$ is $CM$, 
\item $Y \rightarrow S$ is Gorenstein, 
\item $S$ is $CM$, 
\item   $X$ and $Y$ have the  same pure relative dimension over $S$,
\item there are open sets $V \subset X$, $U \subset Y$ such that $f(V) \subset U$,
  $U$ and $V$ are smooth over $S$ and $X \backslash V$ is of codimension $2$ in $X$,
\item the morphism $f$ restricts to a surjective map from $D_X$ to $D_Y$,
\item $D_X \cap V$ and $D_Y \cap U$ are relative normal crossings divisors.
 \end{enumerate}
 The fundamental class $\Theta : \oscr_X \rightarrow f^! \oscr_Y$
 constructed in Proposition \ref{prop-trace} induces a map:
 $$\oscr_X(-D_X) \rightarrow f^!  \oscr_Y(-D_Y).$$
   \end{proposition}

\begin{demo} We may assume that $X$ and $Y$ are smooth, $D_X$ and
  $D_Y$ are relative normal crossing divisors. In that case, we have a
  well defined differential map $ \mathrm{df} : f^\star
  \Omega^1_{Y/S}(\log D_Y) \rightarrow \Omega^1_{X/S}( \log
  D_X)$. Taking the determinant yields $\det \mathrm{df} : f^\star
  \det \Omega^1_{Y/S}(D_Y) \rightarrow \det \Omega^1_{X/S}(D_X)$ or
  equivalently $\det \mathrm{df} : \oscr_{X}(-D_X) \rightarrow
  f^!\oscr_{Y}(-D_Y)$.  
\end{demo}

\subsection{Cohomological correspondences} 
Let $X$, $Y$ be two $S$-schemes.

\begin{defi} A correspondence $C$ over $X$ and $Y$ is a diagram of
  $S$-morphisms :
\begin{eqnarray*}
\xymatrix{ & C \ar[rd]^{p_1}  \ar[ld]_{p_2} & \\
X & & Y }
\end{eqnarray*}
where $X$, $Y$, $C$ have the same pure relative dimension over $S$ and
the morphisms $p_1$ and $p_2$ are projectively embeddable.
\end{defi}

\begin{rem}  In practice, the maps $p_1$, $p_2$ will often be surjective and
  generically finite.
\end{rem}

Let $\mathscr{F} \in \mathbf{D}^{-}_{qcoh}(\oscr_X)$ and
$\mathscr{G} \in \mathbf{D}^+_{qcoh}(\oscr_Y)$.

\begin{defi} A cohomological correspondence from $\mathscr{F}$ to
  $\mathscr{G}$ is the data of a correspondence $C$ over $X$ and $Y$
  and a map $ T : R(p_1)_\star Lp_2^\star \mathscr{F}
  \rightarrow \mathscr{G}$.
\end{defi}

The map $T$ can be seen, by \eqref{eq:duality}, as a map
$Lp_2^\star \mathscr{F} \rightarrow p_1^! \mathscr{G}$.
Note that if $\mathscr{F}$ and $\mathscr{G}$ are coherent sheaves (in
degree $0$), then $Lp_2^\star \mathscr{F}$ is concentrated in degrees
$\leq 0$ and $p_1^!\mathscr{G}$ is concentrated in degrees $\geq 0$
(as follows from the construction in \cite[III, Theorem 8.7]{Hartshorne}), so
any map $L p_2^\star \mathscr{F} \rightarrow p_1^! \mathscr{G}$ factors
uniquely through the natural map 
$L p_2^*\mathscr{F} \to p_2^* \mathscr{F}$.
It gives rise to a map, still denoted by $T$, on cohomology:
$$  \mathrm{R} \Gamma(X, \mathscr{F}) \stackrel{Lp_2^\star}\rightarrow \mathrm{R}\Gamma(C, p_2^\star \mathscr{F}) = \mathrm{R}\Gamma(Y, \mathrm{R}(p_1)_\star p_2^\star \mathscr{F}) \stackrel{T}\rightarrow \mathrm{R}\Gamma( Y, \mathscr{G}).$$

\begin{ex} \label{ex:corres} Let $C$ be a correspondence over $X$ and
  $Y$. We assume that the map $p_1 : C \rightarrow Y$ satisfies the
  assumptions of Proposition \ref{prop-trace}, so there is a
  fundamental class $\oscr_C \rightarrow p_1^!\oscr_Y$.  Let
  $\mathscr{F}$ and $\mathscr{G}$ be locally free sheaves of finite
  rank over $X$ and $Y$. We assume that there is a map
  $p_2^\star \mathscr{F} \rightarrow p_1^\star \mathscr{G}$.  Twisting
  the fundamental class by $p_1^\star \mathscr{G}$, we get a morphism
  $ p_1^\star \mathscr{G} \rightarrow p_1^! \oscr_{Y} \otimes
  p_1^\star \mathscr{G}$ and using the isomorphism of Proposition
  \ref{prop-finite-tor-dim}, we get a map
  $ p_1^\star \mathscr{G} \rightarrow p_1^!  \mathscr{G}$.  Composing
  everything we obtain
  $T : p_2^\star \mathscr{F} \rightarrow p_1^! \mathscr{G}.$
\end{ex}

\section{The Satake isomorphism for unramified reductive groups}\label{section-Satake}

\subsection{The dual group}\label{sec-dual-group}  
Let $K$ be a nonarchimedean local field, $\ocal_K$ its ring of
integers, and $\pi \in \ocal_K$ a uniformizing element. Let $q$ be the
cardinality of the residue field $\ocal_K/(\pi)$.  Let $G$ be a
reductive group over $\Spec~\ocal_K$. We assume that $G$ is
quasi-split and splits over an unramified extension of $K$, and fix a
Borel subgroup $B$ and a maximal torus $T \subset B$. We denote by
$X_\star (T)$ and $X^\star(T)$ the groups of cocharacters and
characters of $T$. These groups carry an action of
$\Gamma = \hat{\ZZ}$, the unramified Galois group. We denote by
$T_d \subset T$ the maximal split torus inside $T$. Its cocharacter
group is $X_\star(T)^\Gamma$ and we have
$X^{\star}(T_d)_{\R} = (X^{\star}(T)_{\R})_{\Gamma}$.

We now introduce the Langlands dual of $G$, following \cite[Chapter
I]{MR546608}.  Let $\hat{G}$ be the dual group of $G$, defined over
$\ov{\qq}$, and let $\hat{T}$ be a maximal torus in $\hat{G}$ (so that
$X_\star (\hat{T}) = X^\star(T)$). We denote by $\Phi^+$ the set of
positive roots for $B$ in $X^\star(T)$ and by
$\hat{\Phi}^+ \subset X_{\star}(T) = X^{\star}(\hat{T})$ the
corresponding set of positive coroots. This is the set of positive
roots for a Borel subgroup $\hat{B}$ of $\hat{G}$ which contains
$\hat{T}$.

We let $W$ be the Weyl group of $G$ and $\hat{G}$ and we denote by
$w_0$ the longest element in $W$ (which maps the Borel $B$ to the
opposite Borel). We let $W_d$ be the subgroup of $W$ which stabilizes
$T_d$; this is also the fixed point set of the action of $\Gamma$ on
$W$ \cite[\S 6.1]{MR546608}.

The set of dominant weights for $G$ is denoted by $P^+$. This is the
cone in $X^\star(T)$ of characters $\lambda$ such that
$\langle \lambda, \alpha \rangle \geq 0$ for all
$\alpha \in \hat{ \Phi}^+$.  We let $P^+_\R \subset X^\star(T)_{\R}$
be the positive Weyl chamber which is the $\mathbb{R}_{\geq 0}$-linear
span of $P^+$. The cone $P^+_\R$ is a fundamental domain for the
action of $W$ on $X^\star(T) \otimes \R$.  We have similar notations
for $\hat{G}$. We denote by $\rho$ the half sum of elements of
$\Phi^+$.

We choose a pinning to define an action of $\Gamma$ on $\hat{G}$,
preserving $\hat{B}$ and $\hat{T}$, and we then let
${}^LG = \hat{G} \rtimes \Gamma$.

\subsection{Hecke algebras}\label{section-Hecke-algebras}

Whenever we have a unimodular locally compact group $\mathcal{G}$ and
a compact open subgroup $\mathcal{K}$, we denote by
$\mathcal{H}(\mathcal{G}, \mathcal{K})$ the algebra of compactly
supported left and right $\mathcal{K}$-invariant functions from
$\mathcal{G}$ to $\ZZ$. The product in
$\mathcal{H}(\mathcal{G}, \mathcal{K})$ is the convolution product
(the Haar measure on $\mathcal{G}$ is normalized by
$vol(\mathcal{K}) = 1$, so $\mathbf{1}_{\mathcal{K}}$ is the unit in
$\mathcal{H}(\mathcal{G}, \mathcal{K})$).

We have an isomorphism of algebras
$\mathcal{H}(T_d(K), T_d(\ocal_K))= \ZZ[X_\star(T_d)]$. Moreover, the
restriction map
$\mathcal{H}(T(K), T(\ocal_K)) \rightarrow \mathcal{H}(T_d(K),
T_d(\ocal_K))$ is an isomorphism of algebras because $T(K)/T(\ocal_K)
= T_d(K)/T_d(\ocal_K)$  by \cite[\S 9.5]{MR546608}.

Since $G$ is quasi-split, the centralizer of $T_d$ in $G$ is $T$, so
$W_d$ is equal to $N(T_d)/T$. We have a Levi decomposition $B= T N$.
Let $\delta$ be the modulus character for $B$.

We now study $ \mathcal{H} :=\mathcal{H}(G(K), G(\ocal_K))$, the
spherical Hecke algebra for $G$.  The characteristic functions
$T_\lambda = \mathbf{1}_{G(\ocal_K) \lambda(\pi) G(\ocal_K)}$ for
$\lambda \in (\hat{P}^+)^\Gamma$ form a $\ZZ$-basis of $\mathcal{H}$
by the Cartan decomposition.

We define the Satake transform: 
\begin{eqnarray*}
  \mathcal{S} :  \mathcal{H} \otimes \ZZ[q^{\frac{1}{2}}, q^{-\frac{1}{2}}] &\rightarrow&  \mathcal{H}(T(K), T(\ocal_K)) \otimes \ZZ[q^{\frac{1}{2}}, q^{-\frac{1}{2}}] \\
  f & \mapsto & \delta(t)^{\frac{1}{2}} \int_{N(K)} f(tn) dn 
\end{eqnarray*}
where the measure on $N(K)$ is normalized so that
$N(K) \cap G(\ocal_K)$ has measure $1$.  It induces an isomorphism
\cite[ Theorem 4.1]{cartier}:
 $$\mathcal{S} :  \mathcal{H} \otimes \ZZ[q^{\frac{1}{2}}, q^{-\frac{1}{2}}] \simeq  \ZZ[q^{\frac{1}{2}}, q^{-\frac{1}{2}}][X_\star(T_d)]^{W_d}.$$
 
 For any $\lambda \in (\hat{P}^+)^\Gamma$ let $V_\lambda$ be the
 irreducible representation of $\hat{G}$ with highest weight
 $\lambda$. Since $\Gamma$ preserves a pinning, the action of
 $\hat{G}$ on $V_{\lambda}$ extends to an action of ${}^LG$. This
 extension is determined uniquely by the condition that $\Gamma$
 preserves a highest weight vector and any other extension differs
 from it by tensoring with a character of $\Gamma$.  We continue to
 denote this extended representation by $V_{\lambda}$.  If we consider (as we may)
 $\hat{G}$ and ${}^LG$ as being defined over $\qq$, then this
 extension of $V_{\lambda}$ is also defined over $\qq$.

 Let $[V_\lambda]$ be the the trace of the representation $V_\lambda$
 on $\hat{T} \rtimes \sigma$, where $\sigma$ is the generator of
 $\Gamma$ given by Frobenius. For $\lambda \in (\hat{P}^+)^{\Gamma}$,
 let $[\lambda]$ denote the formal sum of the elements in the
 $W_d$-orbit of $\lambda$, viewed as an element of
 $\ZZ[X_\star(T_d)]^{W_d}$.
 
 \begin{lem} \label{lem:trace}

   The trace $[V_\lambda]$ belongs to
   $\ZZ[X_\star(T_d)]^{W_d}$. Moreover,
 
   $[V_\lambda] = [\lambda] + \sum_{\mu < \lambda} a_\mu[\mu]$ with
   $\mu \in (\hat{P}^+)^\Gamma$ and
   $a_\mu \in \ZZ.$\footnote{It can be shown that $a_{\mu} \geq 0$ but
     we do not need this.}
 \end{lem}
 
 \begin{proof}
   We have a decomposition
   $V_\lambda = \oplus_{\mu \in X_\star(T)} V_\lambda^{\mu}$, where
   $V_{\lambda}^\mu$ is the weight $\mu$ eigenspace of $V_{\lambda}$
   for the action of $\hat{T}$. The action of $\sigma$ permutes the
   spaces $V_\lambda^{\mu}$, so if $\mu$ is not fixed by $\sigma$ then
   the trace of any element $t\rtimes \sigma$ restricted to
   $\sum_{\mu' \in O_{\Gamma}(\mu)} V_{\lambda}^{\mu'}$ is zero. Here
   $O_{\Gamma}(\mu)$ denotes the $\Gamma$-orbit of $\mu$. It follows
   that $[V_\lambda] \in \ov{\qq}[X_\star(T_d)]$. Since the weight
   spaces are permuted by $W$ and $\sigma$ acts trivially on $W_d$,
   $[V_\lambda] \in \ov{\qq}[X_\star(T_d)]^{W_d}$.

   As noted earlier, we may consider $\hat{G}$, ${}^LG$ and
   $V_{\lambda}$ as being defined over $\qq$.  For each
   $\mu \in X_\star(T_d)$, the eigenspace $V_\lambda^{\mu}$ is defined
   over $\qq$ and $\sigma$ acts on this space as a finite order
   operator, so its trace must be in $\ZZ$.

   Finally, since the highest weight space is one dimensional, we have
   $[V_\lambda] = [\lambda] + \sum_{\mu < \lambda} a_\mu [\mu]$
   $\mu \in (\hat{P}^+)^{\Gamma}$ with $a_\mu \in \ZZ$.
\end{proof}

\begin{coro}\label{cor:basis}%
  The elements $[V_\lambda]_{\lambda \in (\hat{P}^+)^\Gamma}$ provide
  a $\ZZ$-basis of $\ZZ[X_\star(T_d)]^{W_d}$.
\end{coro}
\begin{proof}
  This follows immediately from Lemma \ref{lem:trace} and the fact
  that the set $\{[\lambda]\}_{\lambda \in (\hat{P}^+)^{\Gamma}}$ is a
  $\ZZ$-basis of $\ZZ[X_\star(T_d)]^{W_d}$.
\end{proof}

 We may now relate the two bases
 $\{[V_\lambda]\}_{\lambda \in (\hat{P}^+)^\Gamma}$ and
 $\{T_\lambda\}_{\lambda \in (\hat{P}^+)^\Gamma}$ of the unramified
 Hecke algebra.
 
 \begin{proposition}
 We have the following properties:
 \begin{enumerate} 
\item $\mathcal{S}(T_\lambda) = q^{\langle \lambda, \rho \rangle} [V_\lambda] + \sum_{ \mu < \lambda} b_{\lambda}(\mu) (q^{\langle \mu, \rho \rangle} [V_\mu] )$ for integers $b_{\lambda}(\mu)$,
\item $q^{\langle \lambda, \rho \rangle} [V_\lambda] = \mathcal{S}(T_\lambda) + \sum_{\mu < \lambda} d_{\lambda}(\mu) \mathcal{S}(T_\mu)$ for integers $d_{\lambda}(\mu)$. 
\end{enumerate}
\end{proposition}

\begin{proof}
  The formulae (1) and (2) are clearly equivalent.  If $G$ is split,
  these formulae are (3.9) and (3.12) of \cite{gross-satake}, though a
  a complete proof is not given there so we provide the missing
  references:

  By (6.8) of \cite{satake}, for $\lambda$, $\mu$ dominant coweights
  in $X_{\star}(T_d)$,
  $\mathcal{S}(T_{\lambda})(\mu(\pi)) =
  \delta(\mu(\pi))a_{\lambda}(\mu)$, where $a_{\lambda}(\mu)$ is an
  integer. By \cite[Remark 2, p.~30]{satake},
  $a_{\lambda}(\lambda) = 1$ iff
  \begin{equation}
    (G(\ocal_K) \lambda(\pi) N(K)) \cap G(\ocal_K) \lambda(\pi)
    G(\ocal_K) = G(\ocal_K) \lambda(\pi).
  \end{equation}
  This holds by part (ii) of the proposition in (4.4.4) of
  \cite{bruhat-tits}. Part (i) of the same proposition implies that
  $a_{\lambda}(\mu) = 0$ unless $\mu \leq \lambda$. (We note that
  these results in the unramified quasi-split case follow from the
  case of split groups by base changing to an unramified extension
  over which the group splits, applying the results there, and then
  taking Galois invariants. Also, all cases that we actually use later
  are already proved in full in \cite{satake}.)
  
  Since $\delta(\mu(\pi)) = q^{\langle \mu, \rho \rangle}$ (as $B$ is
  a Borel subgroup), (1) now follows from Lemma \ref{lem:trace}.
\end{proof}

\subsection{Conjugacy classes}\label{sect-conj-classes} 

We denote by $\hat{T_d}$ the torus with cocharacter group $X^\star(T_d)$. Note that there is a map $\hat{T} \rightarrow \hat{T_d}$. It follows from the Satake isomorphism that there is a bijection: 
$$ \mathrm{Hom} ( \mathcal{H}, \overline{\qq})  = \hat{T_d}(\ov{\qq})/W_d$$
We have a surjective map $N_{\hat{G}}(\hat{T}) \rightarrow W$. Let us
denote by $\widetilde{W_d}$ the inverse image of $W_d$. By Lemma 6.4
and 6.5 of \cite{MR546608}, we have natural bijections
\begin{equation} \label{eq:that}
  \hat{T}(\ov{\qq}) \times \sigma / Int( \widetilde{W_d})
  \rightarrow  \hat{G}(\ov{\qq})^{ss}/\sigma-\mathrm{conj}
\end{equation}
and
\begin{equation} \label{eq:td}
  \hat{T}(\ov{\qq}) \times \sigma / Int( \widetilde{W_d}) \rightarrow
  \hat{T}_d(\ov{\qq})/W_d
\end{equation}
so we deduce that there is a bijection
$$ \mathrm{Hom} ( \mathcal{H}, \overline{\qq}) \stackrel{\sim}\rightarrow \hat{G}(\ov{\qq})^{ss}/\sigma-\mathrm{conj} $$
which associates to any $\chi : \mathcal{H} \rightarrow  \overline{\qq}$ a  semi-simple $\sigma$-conjugacy class $c \in \hat{G}(\ov{\qq})^{ss}$ characterized by the property that 

$$ \chi ([V_\lambda]) = \mathrm{Tr}(c \rtimes \sigma \vert V_\lambda)$$ for all $\lambda \in (\hat{P}^+)^\Gamma$.

\subsection{The Newton map}\label{sect-Newtonmap} Let us now fix a valuation $v$ on $\ov{\qq}$ extending the $p$-adic valuation of $\qq$ for some prime number $p$, and normalized by $v(p)=1$.  The valuation $v$ defines a homomorphism $\mathbb{G}_m(\ov{\qq}) \rightarrow \R$ and induces the Newton polygon map: 
$$ \mathrm{Newt}_v :  \hat{G}(\ov{\qq})^{ss}/\sigma-\mathrm{conj} =
\hat{T_d}(\ov{\qq})/W_d \stackrel{v}{\rightarrow} X^\star(T_d)_{\R}/W_d \stackrel{\phi}{\to} (P^+_{\R})^{\Gamma},$$
where the last map $\phi$ is defined as follows:
We have a canonical identification of $(X^{\star}(T)_{\R})^{\Gamma}$
with $(X^{\star}(T)_{\R})_{\Gamma} = X^{\star}(T_d)_{\R}$ induced by
the inclusion of $(X^{\star}(T)_{\R})^{\Gamma}$ in
$X^{\star}(T)_{\R}$.  Given $y \in X^{\star}(T_d)_{\R}/W_d$, let
$x \in (X^{\star}(T)_{\R})^{\Gamma}$ be a lift of $y$, let $w \in W$
be such that $w(x) \in P^+_{\R}$ and then define $\phi(y)$ to be the
$\Gamma$-average of $w(x)$, which is an element of
$(P^+_{\R})^{\Gamma}$. Namely, $$\phi(y) = \frac{1}{ \vert \Gamma/ \mathrm{Stab}_{\Gamma} w(x) \vert} \sum_{\gamma  \in  \Gamma/ \mathrm{Stab}_{\Gamma} w(x)} \gamma.w(x).$$ This is independent of the choice of the lift $x$ of $y$,
since any other lift is of the form $w'(x)$ with $w' \in W_d$.

\begin{lem} \label{lem:w_0}
  Let $c \in \hat{G}(\ov{\qq})^{ss}/\mathrm{conj}$ and $\lambda \in
  (\hat{P}^+)^{\Gamma}$. Then  the minimal valuation of an eigenvalue of $c
  \rtimes \sigma$ acting on $V_{\lambda}$ is equal to $\langle
  w_0(\lambda), \mathrm{Newt}_v(c) \rangle$.
\end{lem}

\begin{proof}
  Since the set of eigenvalues is invariant under conjugation, we may
  assume by \eqref{eq:that} that $c \in \hat{T}(\ov{\qq})$. Let
  $\hat{T}^{\Gamma}$ be the largest subtorus of $\hat{T}$ on which
  $\Gamma$ acts trivially. The map $\hat{T}^{\Gamma} \to \hat{T}_d$
  induced by the inclusion of $\hat{T}^{\Gamma}$ in $\hat{T}$ is an
  isogeny, so the map
  $\hat{T}^{\Gamma}(\ov{\qq}) \to \hat{T}_d(\ov{\qq})$ is
  surjective. Thus, by \eqref{eq:td} we may assume that
  $c \in \hat{T}^{\Gamma}(\ov{\qq})$.

  Let $\mu$ be a weight of $\hat{T}$ occurring in $V_{\lambda}$. Then
  $\sigma$, hence also $c \rtimes \sigma$, preserves the subspace
  $\sum_{\mu' \in O_{\Gamma}(\mu)} V_{\lambda}^{\mu'}$ of
  $V_{\lambda}$. Since $c$ is invariant under $\Gamma$, $c$ acts on
  this space by the scalar $\mu(c)$. On the other hand $\sigma$ acts
  on this space by a finite order automorphism. It follows that all
  the eigenvalues of $c \rtimes \sigma$ on this space have valuation
  equal to $v(\mu(c))$.

  The choice of the element $c$ gives an obvious lift
  $x \in (X^{\star}(T)_{\R})^{\Gamma}$ of the image of $c$ in
  $X^\star(T_d)_{\R}/W_d$. Let $w \in W$ be such that
  $w(x) \in {P}^+_{\R}$. Then $v(\nu(w(c)) \geq 0$ for all positive
  coroots $\nu \in \hat{\Phi}^+$. The eigenvalue of smallest valuation
  of $w(c)$ on $V_{\lambda}$ is then clearly the one corresponding to
  the lowest weight $w_0(\lambda)$, so it has valuation
  $\langle w_0(\lambda), w(x) \rangle$. Since $\lambda$, hence
  $w_0(\lambda)$, and also the pairing $\langle \ , \rangle$, are
  invariant under $\Gamma$, we have
  \[
    \langle w_0(\lambda), w(x) \rangle = \langle w_0(\lambda),
    \mathrm{Newt}_v(c) \rangle \ .
   \]
   On the other hand, the set of eigenvalues of $w(c)$ and $c$ on
   $V_{\lambda}$ is the same. The lemma follows since we have seen
   that the set of valuations of the eigenvalues of $c$ and
   $c \rtimes \sigma$ on $V_{\lambda}$ is the same.
\end{proof}

Recall that for two elements $\nu_1 , \nu_2 \in {P}_\R^+$, we write
$\nu_1 \leq \nu_2$ if $\nu_2-\nu_1$ is a linear combination with
coefficients in $\R_{\geq 0}$ of elements of $\Phi^+$. We use the same
notation for the restriction of this ordering to
$(P_{\R}^+)^{\Gamma}$.

\begin{rem}  Assume that  $G$ is $\mathrm{GL}_n$, $T$ is the diagonal torus and $B$ is the upper triangular subgroup. We identify $X^\star(T)$ with $\ZZ^n$ and $P^+_{\R} = \{ (\lambda_n, \cdots, \lambda_1) \in \R^n,~\lambda_{n} \geq  \cdots \geq \lambda_1\}$. To any $\lambdaÊ\in P^+_{\R}$ we can associate a convex polygon in $\R^2$ whose vertices are the $(i, \lambda_1 + \cdots + \lambda_i)$ for $0 \leq i \leq n$.  If  $\mu, \nu \in P^+_{\R}$ , then  $\mu \leq \nu$  means that the polygon of $\mu$ is above the polygon of $\nu$ with the same ending point. 
\end{rem}

\begin{lem}\label{lem-lafforgue}
  Let $\nu \in ({P}_\R^+)^{\Gamma}$ and
  $c \in \hat{G}(\ov{\qq})^{ss}/\mathrm{conj}$. Then
  \[
    \mathrm{Newt}_v(c) \leq \nu
  \]
  if and only if for all $\lambda \in (\hat{P}^+)^{\Gamma}$
  \[
    v( \mathrm{Tr}(c \rtimes \sigma \vert V_\lambda) ) \geq \langle w_0(\lambda), \nu
\rangle . \]
\end{lem}

Note that the second inequality in the lemma does not depend on the
choice of extension of $V_{\lambda}$ that we have made above, since
any other choice differs from it by tensoring with a character of
$\Gamma$.
\begin{proof}
  When $G$ is split, this is Lemme 1.3 of \cite{MR2869300}; we show
  that Lafforgue's proof extends to our setting without much
  difficulty. By \eqref{eq:that} we may and do assume that
  $c \in \hat{T}(\ov{\qq})$.
  
  We first assume that for all $\lambda \in (\hat{P}^+)^{\Gamma}$,
  $v( \mathrm{Tr}(c \rtimes \sigma \vert V_\lambda) ) \geq \langle
  w_0(\lambda), \nu \rangle$.  For $\lambda \in (\hat{P}^+)^{\Gamma}$,
  let $r_{\lambda} = \dim(V_{\lambda})$ and consider the
  representations $\wedge^i V_{\lambda}$ for
  $1 \leq i \leq r_{\lambda}$. The action of $\sigma$ on $V_{\lambda}$
  maps $V_{\lambda}^\mu$ to $V_{\lambda}^{\sigma(\mu)}$ and so the
  analogous statement holds for the weight spaces of
  $\wedge^i V_{\lambda}$.  As a representation of $\hat{G}$, this
  breaks up as a sum of representations $V_{\mu}$ with
  $\mu \in \hat{P}^+$ and $\mu \leq i\lambda$. If the line generated
  by the highest weight vector of such a representation is not
  preserved by $\sigma$, which always holds if
  $\mu \notin (\hat{P}^+)^{\Gamma}$, then this summand of
  $\wedge^i V_{\lambda}$ is not preserved by $\sigma$ and the trace of
  $c$ on the sum of all such summands (which is preserved by $c$) is
  $0$. If $\mu \in (\hat{P}^+)^{\Gamma}$ and the line spanned by the
  highest weight vector is preserved by $\sigma$, then this
  $\hat{G}$-irreducible summand isomorphic to $V_{\mu}$ is preserved
  by ${}^LG$. It is not clear  whether this
  representation is always isomorphic to our chosen extension but, as
  noted above, the condition in the lemma is independent of the
  choice.

  Since the lowest weight occuring in all the representations above is
  $\geq w_0(i\lambda)$ and only the invariant $\mu$ contribute to the
  trace of $c$, we deduce from the assumption at the beginning that
  \begin{equation} \label{eq:tr} v( \mathrm{Tr}(c \rtimes \sigma \vert \wedge^i
    V_{\lambda}) ) \geq i \langle w_0(\lambda), \nu \rangle
  \end{equation}
for all $i \geq 0$.

Let $\alpha_1,\alpha_2,\dots,\alpha_{r_{\lambda}}$ be the eigenvalues
of $c$ acting on $V_{\lambda}$ ordered so that
$v(\alpha_1)\leq\dots\leq v(\alpha_{r_{\lambda}})$. Let
$i \in \{1,\dots,r_{\lambda}\}$ be such that
$v(\alpha_1)=\dots=v(\alpha_i) < v(\alpha_{i+1})$. Then
$ v( \mathrm{Tr}(c \rtimes \sigma \vert \wedge^i V_{\lambda}) ) = iv(\alpha_1)$ so we
deduce from \eqref{eq:tr} that
$v(\alpha_1) \geq \langle w_0(\lambda), \nu \rangle$. On the other
hand by Lemma \ref{lem:w_0},
$v(\alpha_1) = \langle w_0(\lambda), \mathrm{Newt}_v(c) \rangle$ since
$w_0(\lambda)$ is the lowest weight of $V_{\lambda}$. We thus have
\begin{equation} \label{eq:tr2}
  \langle w_0(\lambda), \mathrm{Newt}_v(c) \rangle \geq \langle
  w_0(\lambda), \nu \rangle
\end{equation}
for all $\lambda \in (P^+)^{\Gamma}$. Since
$\lambda \mapsto -w_0(\lambda)$ is a bijection of
$(\hat{P}^+)^{\Gamma}$ into itself and the cone generated by the
elements of $(\hat{P}^+)^{\Gamma}$ is dual to the $\Gamma$-invariants
of the cone generated by the positive coroots of $\hat{G}$, we deduce
that $ \mathrm{Newt}_v(c) \leq \nu$.

The converse, not explicitly stated in \cite{MR2869300}, is simpler:
If $\mathrm{Newt}_v(c) \leq \nu$, then $\nu - \mathrm{Newt}_v(c)$ is a
$\Gamma$-invariant element of the cone generated by the positive
coroots of $\hat{G}$. The result then follows from the formula
$ v( \mathrm{Tr}(c\rtimes \sigma \vert V_\lambda) )=v(\alpha_1) =
\langle w_0(\lambda), \mathrm{Newt}_v(c) \rangle$ used above and the
fact that $\langle w_0(\lambda), \alpha \rangle \leq 0$ for any
$\alpha \in \Phi^+$.
\end{proof}

\section{Shimura varieties}\label{section-Shimura}

\subsection{Shimura varieties in characteristic zero} Let $(G,X)$ be a
Shimura datum \cite[\S 1]{MR546620}. This means that $G$ is a
reductive group over $\qq$, and $X$ is a $G(\R)$-conjugacy class of
morphisms
$$ h : \mathrm{Res}_{\C/\R} \mathbb{G}_m \rightarrow G_{\R}$$ satisfying the following conditions :
\begin{enumerate}
\item Let $\mathfrak{g}$ be the Lie algebra of $G$. Then for any
  $h \in X$, the adjoint action determines on $\mathfrak{g}$ a Hodge
  structure of type $(-1,1)$, $(0,0)$, $(1,-1)$.
\item For all $h \in X$, $\mathrm{Ad}h(i)$ is a Cartan involution on $G^{ad}(\R)$. 
\item $G^{ad}$ has no factor $H$ defined over $\qq$ such that $H(\R)$ is compact. 
\end{enumerate}

\begin{ex} The most fundamental example is the Siegel Shimura datum
  $(\mathrm{GSp}_{2g}, \mathcal{H}^{±}_g )$. Let
  $g \in \mathbb{Z}_{\geq 1}$, let $V = \mathbb{Q}^{2g}$, and let
  $\psi$ be the symplectic form with $\psi(e_i, e_{2g+1-i}) = 1$ if
  $1 \leq i \leq g$ and $\psi(e_i, e_j) = 0$ if $i+j \neq 2g+1$. Let
  $\mathrm{GSp}_{2g}$ be the group of symplectic similitudes of
  $(V,\psi)$. We take
  $\mathcal{H}_g^{±} = \{ M \in \mathrm{M}_{g\times
    g}(\mathbb{C}),~M=~^tM,~\mathrm{Im}(M)~\textrm{is definite} \}$ to
  be the Siegel space. Observe that
  $\mathcal{H}_g^{±} = \mathcal{H}_g^+ \cup \mathcal{H}_g^-$, where
  $\mathcal{H}_g^+$ (resp.~$\mathcal{H}_g^-$) is defined by the
  condition $\mathrm{Im}(M)$ is positive (resp.~negative) definite.
  Then $\mathcal{H}_g^{±}$ is the
  $\mathrm{GSp}_{2g}(\mathbb{R})$-orbit of the morphism
  $$h_0 : \mathrm{Res}_{\C/\R} \mathbb{G}_m \rightarrow
  G_{\R},~~~~~a+ib \in \C^\times \mapsto
  \begin{pmatrix} % or pmatrix or bmatrix or Bmatrix or ...
    a 1_g & -b S_g \\
    b S_g& a 1_g \\
   \end{pmatrix}$$ where $1_g$ is the identity $g \times g$ matrix and $S_g$ is the $g\times g$ matrix with $1$'s on the anti-diagonal and $0$'s otherwise. 
\end{ex} 

Since  $\mathrm{Res}_{\C/\R}(\C) = \C^\times \times \C^\times$,  the choice of $h \in X$ determines, by projection to the first factor,  a cocharacter $\mu : \C^\times \rightarrow G_{\C}$.  We denote by $P_{\mu} \subset G_{\C}$ the parabolic defined by $P_{\mu} = \{g \in G_\C,~~\lim_{t \rightarrow + \infty} \mathrm{Ad}(\mu(t)) g~Ê\textrm{exists}\}$ and by $M_\mu = \mathrm{Cent}_{G_{\C}}(\mu)$ its Levi factor. 

Let $\mathrm{FL}_{G,X}$ be the Flag variety parametrizing parabolic
subgroups in $G$ of type $P_{\mu}$. The Borel embedding is the map
$X \hookrightarrow \mathrm{FL}_{G,X}$ given by $h \mapsto P_\mu$.
Using this map, we endow $X$ with a complex structure
\cite[Proposition 1.1.14]{MR546620}.

Let $K \subset G(\mathbb{A}_f)$ be a compact open subgroup. We let
$Sh_K(\mathbb{C}) = G(\qq) \backslash( X \times G(\mathbb{A}_f)/K)$ be
the associated Shimura variety.  This is a complex manifold as soon as
$K$ is neat and in fact it has a structure of algebraic variety over
$\C$.

The conjugacy classes of $\mu$, $P_{\mu}$ and $M_{\mu}$ are defined
over a number field $E = E(G,X)$ called the reflex field. We deduce
that $\mathrm{FL}_{G,X}$ is defined over $E$. Moreover, it has been
proven by Shimura, Deligne, Borovoi, Milne (see, e.g.,
\cite{MR717596}) that $Sh_K(\mathbb{C})$ has a canonical model
$Sh_K \rightarrow \Spec~E$.

\begin{ex} In the Siegel case
  $(\mathrm{GSp}_{2g}, \mathcal{H}^{±}_g )$, $Sh_K$ has a moduli
  interpretation: it is the moduli space of abelian varieties of
  dimension $g$, with a polarization and a level structure (prescribed
  by $K$).
\end{ex}

\subsection{Compactifications} 

For any choice $\Sigma$ of rational polyhedral cone decomposition, one
can construct a toroidal compactification $Sh_{K, \Sigma}^{tor}$ of
$Sh_K$. In general, this is a proper algebraic space over $E$ and we
have the property that
$D_{K, \Sigma} = Sh_{K, \Sigma}^{tor} \setminus Sh_K$ is a Cartier
divisor (see \cite{AMRT} and \cite{PINK}). Furthermore, for a suitable
choice of $\Sigma$, $Sh_{K, \Sigma}^{tor}$ is smooth and projective.

\subsection{Automorphic vector bundles}

Let $Z_s(G)$ be the greatest sub-torus of the center of $G$ which has no split subtorus over $\qq$ but which splits over $\R$.  Let $G^c = G/Z_s(G)$. 

\begin{rem} If $F$ is a totally real field and $G = \mathrm{Res}_{F/\qq} \mathrm{GL}_n$, with center $Z =  \mathrm{Res}_{F/\qq} \mathrm{GL}_1$,   then $Z_s(G)$ is  the kernel of the norm map $ Z \rightarrow \mathrm{GL}_1$. 
\end{rem}

One can define (\cite[III, \S 3]{MR1044823}) an analytic space
$P_K(\mathbb{C}) \rightarrow S_K(\mathbb{C})$, called the principal
$G^c$-bundle by:
$$P_K(\mathbb{C}) = G(\mathbb{Q}) \backslash X \times G^c(\mathbb{C})
\times G(\mathbb{A}_f)/ K.$$ There is a natural,
$G(\mathbb{C})$-equivariant map $P_K(\mathbb{C}) \rightarrow \mathrm{FL}_{G,X}$
defined by sending
$(x, g, g') \in G(\mathbb{Q}) \backslash X \times G^c(\mathbb{C})
\times G(\mathbb{A}_f)/ K$ to $g^{-1}x$. We therefore have a diagram
of analytic spaces:
\begin{eqnarray*}
\xymatrix{ & P_K(\mathbb{C}) \ar[rd] \ar[ld] & \\
Sh_K(\mathbb{C}) & & \mathrm{FL}_{G,X}}
 \end{eqnarray*}
By \cite[III, Theorem 4.3 a)]{MR1044823}, $P_K(\mathbb{C})$ is the analytification of an algebraic variety $P_K$ defined over $E$, and there is a diagram of schemes: 
\begin{eqnarray*}
\xymatrix{ & P_K \ar[rd]^{\alpha} \ar[ld]_\beta & \\
Sh_K & & \mathrm{FL}_{G,X}}
 \end{eqnarray*}
where $\alpha$ is  $G$-equivariant and $\beta$ is a $G^c$-torsor. 

\begin{ex} In the Siegel case, let $A \rightarrow  Sh_K$ be the universal abelian scheme (only well defined up to quasi-isogenies,  a representative of the isogeny class can be fixed by the choice of an integral $PEL$-datum).  Then $P_K$ is the $\mathrm{GSp}_{2g}$-torsor of trivializations of $\mathcal{H}_{1,dR}(A/Sh_K)$ and the map $\alpha$ is given by the Hodge filtration on $\mathcal{H}_{1,dR}(A/Sh_K)$.
\end{ex}

 Let $E'$ be a finite extension of $E$ such that the conjugacy class of $\mu$, $M_\mu$ and $P_\mu$ have  representatives defined over $E'$, and all finite dimensional  algebraic representations of $M_{\mu}$ are defined over $E'$. We now consider  all the schemes $P_K$, $Sh_K$, $\mathrm{FL}_{G,X} = G/P_\mu$ over $E'$. 
Denote by $VB_{G}(\mathrm{FL}_{G,X})$ the category of $G$-equivariant vector bundles on $\mathrm{FL}_{G,X}$ and by  $\mathrm{Rep}_{E'}(M_{\mu})$ the category of finite dimensional algebraic representations of $M_{\mu}$ on $E'$-vector spaces.  There is a functor 
 \begin{eqnarray*} \label{functorRepVB}
  \mathrm{Rep}_{E'}(M_{\mu}) &\rightarrow& VB_{G}(\mathrm{FL}_{G,X}) \\
  V & \mapsto& \mathcal{V}  \end{eqnarray*}
  which is defined by $\mathcal{V} = G \times V/\sim$ where $\sim$ is the equivalence relation $(gx,v) \sim (g, xv)$ for all $(g,v,x) \in G \times V \times P_\mu$, and we let $P_\mu$ act on $V$ through its projection $P_\mu \rightarrow M_\mu$.  
  
Let $VB(Sh_K)$ be the category of vector bundles over $Sh_K$ and $\mathrm{Rep}_{E'}(M_{\mu}/Z_s(G))$ the category of finite dimensional algebraic representations of $M_{\mu}/Z_s(G)$ on $E'$-vector spaces.  We deduce that there  is a functor 
 \begin{eqnarray*} 
  \mathrm{Rep}_{E'}(M_{\mu}/Z_s(G)) &\rightarrow& VB(Sh_{K}) \\
  V & \mapsto& \mathcal{V}_K  \end{eqnarray*}
which is  defined as follows: a representation $V$ of $M_{\mu}/Z_s(G)$ defines a representation of $P_{\mu}$ (by letting the unipotent radical act trivially) and therefore a $G$-equivariant vector bundle $\mathcal{V}$ over $\mathrm{FL}_{G,X}$. We can pull back  this vector bundle by the map $\alpha$ to $P_K$  and descend it to $Sh_K$ using $\beta$. 

\subsection{Automorphic vector bundles and compactifications}\label{sect-automorphicvb} 

By \cite[Theorem 4.2]{MR997249}, for a choice $\Sigma$ of rational
polyhedral cone decomposition, there is a $G^c$-torsor
$P_{K,\Sigma} \rightarrow Sh_{K,\Sigma}^{tor}$ extending $P_K$, and a
diagram:
\begin{eqnarray*}
\xymatrix{ & P_{K,\Sigma} \ar[rd]^{\alpha} \ar[ld]_\beta & \\
Sh^{tor}_{K,\Sigma} & & \mathrm{FL}_{G,X}}
 \end{eqnarray*}
where $\alpha$ is  $G$-equivariant and $\beta$ is a $G^c$-torsor.
Therefore, the functor $V \mapsto \mathcal{V}_K$ extends to a functor 
 \begin{eqnarray*} 
 \ \mathrm{Rep}_{E'}(M_{\mu}/Z_s(G)) &\rightarrow& VB(Sh^{tor}_{K, \Sigma}) \\
  V & \mapsto& \mathcal{V}_{K, \Sigma}.
  \end{eqnarray*}
and $\mathcal{V}_{K, \Sigma}$ is called the canonical extension of $\mathcal{V}_{K}$. Moreover, $\mathcal{V}_{K, \Sigma}(-D_{K,\Sigma})$ is called the sub-canonical extension of $\mathcal{V}_{K}$. 
\subsection{Cohomology  of vector bundles and Hecke operators}\label{section-coho-vec-bund-Hecke}

For any $\Sigma$, $K$ and $V \in \mathrm{Rep}(M_\mu/Z_s(G))$, we
identify the vector bundle $\mathcal{V}_{K,\Sigma}$ with its
associated locally free sheaf of sections, and we consider the
cohomology groups:
$$\HH^\star (Sh^{tor}_{K, \Sigma} , \mathcal{V}_{K,
  \Sigma})~~~\textrm{and}~~~\HH^\star (Sh^{tor}_{K, \Sigma} ,
\mathcal{V}_{K, \Sigma}(-D_{K, \Sigma})).$$ By \cite[Proposition
2.4]{MR1064864} these groups are independent of the choice of
$\Sigma$, and we therefore simplify notations and let
$\HH^i(K, V) = \HH^i (Sh^{tor}_{K, \Sigma} , \mathcal{V}_{K, \Sigma})$
and
$\HH^i_{cusp}(K, V) = \HH^i (Sh^{tor}_{K, \Sigma} , \mathcal{V}_{K,
  \Sigma}(-D_{K,\Sigma}))$ .

Let $\mathcal{H}_K$ be the Hecke algebra of functions
$G(\mathbb{A}_f ) \rightarrow \ZZ$, which are compactly supported and
bi $K$-invariant.  By \cite[Proposition 2.6]{MR1064864}, the Hecke
algebra $\mathcal{H}_K$ acts on the cohomology groups $\HH^i(K, V)$
and $\HH^i_{cusp}(K, V)$.

Observe that the Hecke algebra is generated by the characteristic
functions $T_g = \mathbf{1}_{K g K}$ for $g \in G(\mathbb{A}_f)$. We
spell out the action of $T_g$ by writing the corresponding
cohomological correspondence.  For any $g \in G(\mathbb{A}_f)$, we
have a correspondence (for suitable choices of polyhedral cone
decomposition $\Sigma$, $\Sigma'$ and $\Sigma''$):
\begin{eqnarray*}
\xymatrix { &  Sh^{tor}_{gKg^{-1} \cap K, \Sigma''}\ar[ld]_{p_2} \ar[rd]^{p_1}&  \\
Sh^{tor}_{ K, \Sigma'} & & Sh^{tor}_{ K, \Sigma} }
\end{eqnarray*}
where $p_1$ is simply the forgetful map (induced by the inclusion
$gKg^{-1} \cap K \subset K$), and $p_2$ is the composite of the action
map
$g : Sh^{tor}_{gKg^{-1} \cap K, \Sigma''} \rightarrow Sh^{tor}_{K \cap
  g^{-1}Kg, \Sigma'''}$ and the forgetful map
$ Sh^{tor}_{K \cap g^{-1}Kg, \Sigma'''} \rightarrow Sh^{tor}_{ K,
  \Sigma'}$ (induced by the inclusion $g^{-1}Kg \cap K \subset
K$). There is a corresponding cohomological correspondence
$T_g : p_2^\star \mathcal{V}_{K,\Sigma'} \rightarrow p_1^!
\mathcal{V}_{K,\Sigma}$ which is simply obtained by composing the
natural isomorphism
$p_2^\star \mathcal{V}_{K,\Sigma'} \rightarrow p_1^\star
\mathcal{V}_{K,\Sigma}$ (see \cite[\S 2.5]{MR1064864}) and the map
$p_1^\star \mathcal{V}_{K,\Sigma} \rightarrow p_1^!
\mathcal{V}_{K,\Sigma}$ which is deduced from the fundamental class
$p_1^\star \oscr_{Sh^{tor}_{ K, \Sigma}} \rightarrow p_1^!
\oscr_{Sh^{tor}_{ K, \Sigma}}$ (see Proposition \ref{prop-trace}).

\subsection{The infinitesimal character}\label{sec-grouptheoretic}

Let again $(G,X)$ be a Shimura datum. The reductive group $G$ is
defined over $\qq$. We have chosen an extension $E'$ of the reflex
field $E$ over which all representations of $M_\mu$ are defined. This
actually forces $G$ to split over $E'$. Let $S$ be a (split) maximal
torus in $G_{E'}$ and let $X^{\star}(S)$ be its character group. We
assume that $S \subset (M_\mu)_{E'}$.  The roots for $G$ that lie in
the Lie algebra of $M_\mu$ are by definition the compact roots. The
other roots are called non-compact.  We make a choice of positive
roots for $M_\mu$. We also make a choice of positive roots for $G$ by
declaring that the non-compact positive roots are those corresponding
to $\mathfrak{g}/\mathfrak{p}_\mu$.  We denote by $\rho$ the half sum
of all the positive roots. Let $\kappa$ be a highest weight for
$M_\mu$.

\begin{defi}  We define  the infinisitemal character of $\kappa$, denoted $\infty(\kappa)$, by the formula 
$$\infty(\kappa) =  - \kappa - \rho \in X^\star(S)_{\qq}.$$ 
\end{defi}

\begin{rem} This is a representative of the infinitesimal character of
  automorphic representations contributing to the cohomology
  $\HH^i(K,V_\kappa)$ or $\HH^i_{cusp}(K,V_\kappa)$ \cite[Proposition
  4.3.2]{MR1064864}.
\end{rem}

We now choose a prime $p$ such that the group $G_{\qq_p}$ is
unramified at $p$. As in Section \ref{sec-dual-group}, we denote by
$T \subset G_{\qq_p}$ a maximal torus, split over an unramified
extension of $\qq_p$ and contained in a Borel subgroup
$B \subset G_{\qq_p}$. We let $X^\star(T)$ be the character group of
$T$ and let $P^+ \subset X^\star(T)$ be the cone of dominant weights.
We fix an embedding $\iota : E' \hookrightarrow \overline{\qq}_p$.

The tori $S \times_{\Spec~E'} \Spec~\overline{\qq}_p$ and
$T \times_{\Spec~\qq_p} \Spec~\overline{\qq}_p$ are conjugated by some
element $g$ of $G(\overline{\qq}_p)$.  Conjugation by $g$ defines an
isomorphism $X^\star(S) \rightarrow X^\star(T)$; this isomorphism
depends on $g$, but the composite map
$X^\star(S) \rightarrow X^\star(T) \rightarrow X^\star(T)/W \simeq
P^+$ is independent of the choice of $g$.

We therefore get a canonical element $\infty(\kappa, \iota) \in P_{\R}^+$
(which only depends on $\iota$), which is by definition the image of
$\infty(\kappa)$ via the map $X^\star(S)_{\R} \rightarrow P_{\R}^+$.

 \subsection{Newton and Hodge polygons}
 
 We assume that $K = K_p K^p$ is neat and $K_p$ is hyperspecial. Let
 $V_\kappa$ be the irreducible representation of $M_\mu$ defined over
 $E'$ with highest weight $\kappa$.

 Let $\mathcal{H}_p = \mathcal{H}(G(\qq_p), K_p)$ be the Hecke algebra
 at $p$. It acts on the groups $\HH^i(K, V_\kappa)$ and
 $\HH^i_{cusp}(K, V_\kappa)$. We put
 $ \HH^i(K, V_\kappa)_{\overline{\qq}_p} = \HH^i(K,
 V_\kappa)\otimes_{E', \iota} \overline{\qq}_p$ and
 $\HH^i_{cusp}(K, V_\kappa)_{\overline{\qq}_p}= \HH^i_{cusp}(K,
 V_\kappa)\otimes_{E', \iota} \overline{\qq}_p$.

Using the results of Sections \ref{sect-conj-classes} and  \ref{sect-Newtonmap}, we have a Newton map: 
$$ \mathrm{Newt}_\iota :  \mathrm{Hom}(\mathcal{H}_p, \overline{\qq}_p) = \hat{G}(\overline{\qq}_p)/\sigma-\mathrm{conj} \rightarrow (P^+_{\mathbb{R}})^{\Gamma}.$$

\begin{conj}\label{conj1} Let $\chi : \mathcal{H}_p \rightarrow \overline{\qq}_p$  be a character    occuring in $\HH^i(K, V_\kappa)_{\overline{\qq}_p}$ or $\HH^i_{cusp}(K, V_\kappa)_{\overline{\qq}_p}$. Then we have  the inequality in $(P^+_{\R})^\Gamma$: 
$$ \mathrm{Newt}_{\iota}( \chi) \leq  \frac{1}{\vert
  \Gamma/\mathrm{Stab}_\Gamma(\infty(\kappa, \iota))\vert}
\sum_{\gamma \in \Gamma/\mathrm{Stab}_\Gamma(\infty(\kappa, \iota))}
-w_0(\gamma \cdot\infty(\kappa, \iota)) ,$$
where $w_0$ is the longest element of the Weyl group.
\end{conj} 

We now explain that this conjecture is compatible with existing
conjectures on the existence and properties of Galois representations
attached to automorphic representations.  Our convention is that the Artin reciprocity law is normalized by sending uniformizing element to geometric Frobenius element and that the Hodge--Tate  weight of the cyclotomic character is $-1$.

This inequality is to be
viewed as an inequality between a Newton and a Hodge polygon.
According to the work of \cite{JunSu} (see also \cite{MR1064864}), all
the cohomology $\HH^i(K, V_\kappa)$ and $\HH^i_{cusp}(K, V_\kappa)$
can be represented by automorphic forms.  Let $\pi$ be an automorphic
representation contributing to these cohomology groups. We fix an
embedding $E \rightarrow \C$ and an isomorphism
$\iota : \C \rightarrow \overline{\qq}_p$ extending our embedding
$\iota : E \hookrightarrow \overline{\qq}_p$. The automorphic
representation $\pi$ is C-algebraic with infinitesimal character
$\infty(\kappa) = - \kappa-\rho$.

Let us assume that $\pi$ is also $L$-algebraic (for simplicity). Then,
according to \cite[Conjecture 3.2.2]{MR3444225}, there should be a
geometric Galois representation
$\rho_{\pi, \iota} : G_{\qq} \rightarrow ~^LG(\overline{\qq}_p)$
satisfying a list of conditions. In particular, it should be
crystalline at $p$ (because $\pi$ has spherical vectors at $p$) with
Hodge--Tate weights $-\infty(\kappa, \iota)$. (Note that our sign
convention concerning the Hodge--Tate weight of the cyclotomic
character is opposite to \cite{MR3444225}). Let
$\chi_{\pi_p} : \mathcal{H}_p \rightarrow \C$ be the character
describing the action of $\mathcal{H}_p$ on $\pi_p$.  The conjugacy
class of the crystalline Frobenius should be given by
$\iota \circ \chi_{\pi_p}$ via the Satake isomorphism.  Then the
inequality in $(P_{\R}^+)^\Gamma$:
$$ \mathrm{Newt}_{\iota}( \iota \circ \chi_{\pi_p}) \leq  \frac{1}{\vert \Gamma/\mathrm{Stab}_\Gamma(\infty(\kappa, \iota))\vert} \sum_{\gamma \in \Gamma/\mathrm{Stab}_\Gamma(\infty(\kappa, \iota))} -w_0(\gamma\cdot\infty(\kappa, \iota))$$
is the Katz--Mazur inequality (see also \cite[Lemma
4.2]{MR2869300}, and Theorem \ref{odd}).

\medskip

Motivated by this conjecture, we can introduce a modified ``integral
structure'' on the Hecke algebra $\mathcal{H}_p$:

\begin{defi}\label{defi-hecke-algebra} Let
  $\mathcal{H}^{int}_{p,\kappa, \iota}$ be the $\ZZ$-subalgebra of
  $\mathcal{H}_p \otimes {\ZZ}[q^{\frac{1}{2}}, q^{-\frac{1}{2}}]$
  generated by the elements
$$[V_\lambda] p^{\langle \lambda, \infty(\kappa, \iota) \rangle}$$
\end{defi} 

%{\color{blue} The two $-w_0$ cancel.}
\begin{lem}\label{lem-laff2} Conjecture \ref{conj1} holds if and only if, for any character $\chi : \mathcal{H}_p \rightarrow \overline{\qq}_p$    occuring in $\HH^i(K, V_\kappa)_{\overline{\qq}_p}$ or $\HH^i_{cusp}(K, V_\kappa)_{\overline{\qq}_p}$, we have $\chi(\mathcal{H}^{int}_{p,\kappa, \iota}) \subset \overline{\ZZ}_p$. 
\end{lem}
\begin{demo} This follows from Lemma \ref{lem-lafforgue}.
\end{demo} 

\begin{lem} The algebra $\mathcal{H}^{int}_{p,\kappa, \iota}$ is a $\ZZ$-algebra of finite type.
\end{lem}

\begin{demo}  Let $\{\lambda_1, \cdots, \lambda_n \}$ be dominant weights of $\hat{G}$ which generate the cone $(\hat{P}^+)^\Gamma$  as a monoid and is closed under $\leq$, in the sense that if  $\lambda \in  (\hat{P}^+)^\Gamma$ satisfies $\lambda \leq \lambda_i$  for some $1 \leq i \leq n$, then $\lambda \in \{\lambda_1, \cdots, \lambda_n \}$.  We claim that the map  $\ZZ[ T_1, \cdots, T_n] \rightarrow \mathcal{H}^{int}_{p,\kappa}$ where $T_i$ goes to $[V_{\lambda_i}] q^{\langle \lambda_i, \infty(\kappa, \iota) \rangle}$ is surjective. 
 Let $\lambda \in (\hat{P}^+)^\Gamma$, and let us prove that $[V_\lambda] q^{\langle \lambda, \infty(\kappa, \iota) \rangle}$ is in the image. We argue by induction and assume that this holds for all $\lambda' \in (\hat{P}^+)^\Gamma$ with $\lambda' < \lambda$. We can write $\lambda = \sum_{i=1}^n k_i \lambda_{i}$ with $k_i \in \mathbb{Z}_{\geq 0}$ and we have that 
 \begin{eqnarray*}
 \prod_{i=1}^nT_i^{k_i} & =&  [V_\lambda] q^{\langle \lambda, \infty(\kappa, \iota) \rangle} + \sum_{\mu \in \hat{P}^+, \mu < \lambda} c_\mu [V_\mu] q^{\langle \lambda, \infty(\kappa, \iota) \rangle}~~~\textrm{with $c_\mu \in \mathbb{Z}_{\geq 0}$}  \\
 &=& [V_\lambda] q^{\langle \lambda, \infty(\kappa, \iota) \rangle} + \sum_{\mu \in \hat{P}^+, ~\mu < \lambda} c_\mu [V_\mu]  q^{\langle \mu, \infty(\kappa, \iota) \rangle}q^{\langle \lambda-\mu, \infty(\kappa, \iota) \rangle}
 \end{eqnarray*}
 and we can conclude since
 $\langle \lambda-\mu, \infty(\kappa, \iota) \rangle \in \ZZ_{\geq 0}$
 because $ \lambda-\mu$ is a finite sum of positive roots with
 non-zero integral coefficients for $\hat{G}$.
 \end{demo}
 
 \bigskip

 \begin{proposition}\label{prop-imply-conj} Conjecture \ref{conj1}
   holds if and only if both $\HH^i(K, V_\kappa)_{ \overline{\qq}_p}$
   and $\HH^i_{cusp}(K, V_\kappa)_{ \overline{\qq}_p}$ contain
   $\overline{\ZZ}_p$-lattices which are stable under
   $\mathcal{H}^{int}_{p,\kappa, \iota}$.
\end{proposition}
\begin{demo} We deduce from Lemma \ref{lem-laff2}, that we can find a
  basis for $\HH^i(K, V_\kappa)_{ \overline{\qq}_p}$, such that the
  elements $\mathcal{H}^{int}_{p,\kappa}$ act via upper triangular
  matrices with integral diagonal coefficients. After conjugating this
  basis by a diagonal matrix
  $\mathrm{diag}( p^{k_1}, \cdots, p^{k_n})$ with
  $k_1 \gg k_2 \cdots \gg k_n$, and using that
  $\mathcal{H}^{int}_{p,\kappa, \iota}$ if a finite type algebra, we
  can suppose that it acts via integral matrices.
\end{demo}

\subsection{Shimura varieties of abelian type}

The theory of integral models of Shimura varieties produces (in many
cases) integral structures on the coherent cohomology of automorphic
vector bundles. In view of Proposition \ref{prop-imply-conj}, it is
natural to ask (see Conjectures \ref{conj2} and \ref{conj3} below)
whether these integral structures are stable under the action of our
integral Hecke algebras.

A Shimura datum $(G,X)$ is of Hodge type if there is an embedding
$(G,X) \hookrightarrow (\mathrm{GSp}_{2g}, \mathcal{H}^{±}_g)$.  A Shimura
datum $(G,X)$ is of abelian type if there is a Shimura datum of Hodge
type $(G_1,X_1)$ and a central isogeny $G_1^{der} \rightarrow G^{der}$
which induces an isomorphism 
$(G_1^{ad}, X^{ad}_1) \simeq (G^{ad}, X^{ad})$, where $X^{ad}_1$ is the
$\mathrm{G}_1^{ad}(\mathbb{R})$-conjugacy class that contains $X_1$
(and similarly for $X^{ad}$).

\begin{ex} Here is an important example of abelian Shimura datum, that
  we call a Shimura datum of symplectic type. Let $F$ be a totally
  real field. We let $G = \mathrm{Res}_{F/\qq} \mathrm{GSp}_{2g}$ and
  $X = (\mathcal{H}^{±}_g)^{[F:\qq]}$.  This datum is related to the
  following Hodge type datum (which is actually of PEL type): take
  $G_1 \subset G$ to be the subgroup of elements whose similitude
  factor is in $\mathrm{GL}_1$ and
  $X_1= (\mathcal{H}^{+}_g)^{[F:\qq]} \cup
  (\mathcal{H}^{-}_g)^{[F:\qq]}$.
\end{ex} 

We fix $(G,X)$ a Shimura datum of abelian type and
$K \subset G(\mathbb{A}_f)$ a neat compact open subgroup. Let $p$ be a
prime such that that $G$ is unramified at $p$, $K = K^pK_p$ and $K_p$
is hyperspecial.  The group $G$ has a reductive model over
$\ZZ_{(p)}$.  Let $E'$ be a finite extension of $\qq$ unramified at
$p$, which splits $G$ (necessarily $E \hookrightarrow E'$).  Let $\lambda$ be a prime of
$\ocal_E$ dividing $p$ and $\lambda'$ a prime of $\ocal_{E'}$ dividing
$\lambda$. We denote by $\ocal_{E,\lambda}$ and $\ocal_{E',\lambda}$ the localizations of $\ocal_{E}$ and $\ocal_{E'}$ at $\lambda$ and $\lambda'$ respectively. The
cocharacter $\mu$ is defined over $E'$, and the parabolic
$P_\mu$ as well as its Levi subgroup $M_\mu$ have a model over
$\Spec~\ocal_{E',\lambda'}$. Moreover, the conjugacy class of $P_{\mu}$ is
defined over $\Spec~\ocal_{E, \lambda}$. The flag variety
$\mathrm{FL}_{G,X}$ therefore has a proper smooth canonical integral
model over $\Spec~\ocal_{E, \lambda}$.   To ease notations, we keep denoting by $G$ the reductive
model of $G$ over $\ZZ_{(p)}$, by
$\mathrm{FL}_{G,X} \rightarrow \Spec~\ocal_{E, \lambda}$ the canonical
integral model of $\mathrm{FL}_{G,X} \rightarrow \Spec~E$, and by
$P_\mu$ and $M_\mu$ the models over $ \Spec~\ocal_{E', \lambda'}$ of
$P_\mu$ and $M_\mu$.

\begin{thm}[\cite{MR2669706}, \cite{MR3569319}]  There is a canonical model $\mathfrak{Sh}_{K} \rightarrow \Spec~\ocal_{E, \lambda}$ of ${Sh}_{K}$. 
\end{thm}

\begin{thm}\label{thm-compact} Assume that  $(G,X)$ is a Shimura datum of Hodge type, or  that  $(G,X)$ is a  Shimura datum of symplectic type. Let $\Sigma$ be a polyhedral cone decomposition. 
\begin{enumerate}
\item There is a canonical integral model
  $\mathfrak{Sh}^{tor}_{K,Ê\Sigma} \rightarrow \Spec~\ocal_{E,
    \lambda}$ for ${Sh}^{tor}_{K, \Sigma}$ which is smooth for
  suitable choices of $\Sigma$.
\item There is a canonical integral model $\mathfrak{P}_{K,\Sigma}$ for the principal $G^c$-torsor $P_{K,\Sigma}$ and there is a diagram : 
\begin{eqnarray*}
\xymatrix{ & \mathfrak{P}_{K,\Sigma} \ar[rd]^{\alpha} \ar[ld]_\beta & \\
\mathfrak{Sh}^{tor}_{K,\Sigma} & & \mathrm{FL}_{G,X}}
 \end{eqnarray*}
where $\alpha$ is  $G$-equivariant and $\beta$ is a $G^c$-torsor.
\end{enumerate}

\end{thm}

\begin{proof} The Siegel case is \cite{MR1083353}, the PEL case is
  \cite{MR3186092} and \cite{MR2968629}, and the Hodge case is
  \cite{MR3948111}. To our knowledge, there is no reference for the
  abelian case in general. We will explain in Section
  \ref{section-symplectictype} the proof for the case of a symplectic
  type Shimura datum by a simple reduction to the PEL case. This is a
  straightforward generalisation of the argument presented in \cite[\S
  3]{BCGP} in the case of the groups
  $\mathrm{Res}_{F/\qq} \mathrm{GSp}_{4}$.
\end{proof}

Let  $\mathrm{Rep}_{\ocal_{E',\lambda'}}(M_{\mu}/Z_s(G))$ be the category of algebraic representations of $M_\mu/Z_{s}(G)$  ($M_\mu/Z_{s}(G)$ is viewed as a reductive
group over $\Spec~\ocal_{E',\lambda'}$) over finite free $\ocal_{E',\lambda'}$-modules. Using Theorem \ref{thm-compact} we get a functor 

\begin{eqnarray*} 
 \ \mathrm{Rep}_{\ocal_{E',\lambda'} }(M_{\mu}/Z_s(G)) &\rightarrow& VB(\mathfrak{Sh}^{tor}_{K, \Sigma}) \\
  V & \mapsto& \mathcal{V}_{K, \Sigma}.
  \end{eqnarray*}
which is an integral version of the functor of section \ref{sect-automorphicvb}.  Depending on the context,  $\mathcal{V}_{K, \Sigma}$ will mean the locally free sheaf over $\mathfrak{Sh}^{tor}_{K, \Sigma}$ attached to an object $V \in \mathrm{Rep}_{\ocal_{E',\lambda'} }(M_{\mu}/Z_s(G))$ or the locally free sheaf on ${Sh}^{tor}_{K, \Sigma}$ attached to  $V_{E'} := V\otimes_{\ocal_{E',\lambda'}} E'$.  
 
The cohomology complexes
$\mathrm{R}\Gamma( \mathfrak{Sh}^{tor}_{K, \Sigma},
\mathcal{V}_{K,\Sigma})$ and
$\mathrm{R}\Gamma( \mathfrak{Sh}^{tor}_{K, \Sigma},
\mathcal{V}^{int}_{K,\Sigma}(-D_{K,\Sigma}))$ are independent of
$\Sigma$ (this is a standard computation using the structure of the
boundary, see \cite[Theorem 8.6]{Lan2016} in the PEL case) and these
are perfect complexes of $\ocal_{E',\lambda'}$-modules.  We observe
that
$$\mathrm{Im}\left ( \mathrm{H}^i( \mathfrak{Sh}^{tor}_{K, \Sigma},
  \mathcal{V}_{K,\Sigma})\otimes_{\ocal_{E',\lambda'}, \iota}
  \overline{\ZZ}_p \rightarrow \HH^i(K,V_{E'})_{\overline{\qq}_p}
\right ) $$ and
$$\mathrm{Im} \left ( \mathrm{H}^i( \mathfrak{Sh}^{tor}_{K, \Sigma},
  \mathcal{V}^{int}_{K,\Sigma}(-D_{K,\Sigma}))\otimes_{\ocal_{E',\lambda'},
    \iota} \overline{\ZZ}_p \rightarrow
  \HH^i_{cusp}(K,V_{E'})_{\overline{\qq}_p} \right ) $$ are lattices
that we denote respectively by $\HH^i(K,V)_{\overline{\ZZ}_p}$ and
$\HH^i_{cusp}(K,V)_{\overline{\ZZ}_p}$.

\subsection{Conjectures on the action of the integral Hecke algebra}
Let $\kappa$ be a dominant weight for $M_\mu/Z_{s}(G)$ and let $V_\kappa$ be the corresponding Weyl representation, defined over $\ocal_{E',\lambda'}$.  Namely, let $\kappa^\vee$ be the dominant weight $-w_{M_\mu} \kappa$ for $w_{M_\mu}$ the longest element of the Weyl group of $M_\mu$. Let $B_{M_{\mu}}$ be the Borel  of $M_\mu$ (corresponding to our choice of positive roots).   We let $V^{int}_\kappa$ be the space of  functions $f : M_\mu \rightarrow \mathbb{A}^1$ with the transformation property: $f (m b) = \kappa^\vee (b) f(m)$ for all $b \in B_{M_{\mu}}$. The right action of $G$ on itself given by $g \mapsto L(g^{-1}) $ (where $L$ is the left translation) induces a left action on the space $V_\kappa$, and $(V_\kappa)_{E'}$ is an irreducible highest weight representation of weight $\kappa$.

\begin{conj}\label{conj2} 

 The lattices $\HH^i(K,V_\kappa )_{\overline{\ZZ}_p}$ and $\HH^i_{cusp}(K,V_\kappa)_{\overline{\ZZ}_p}$ are stable under $\mathcal{H}^{int}_{p,\kappa,\iota}$.
\end{conj}

\begin{conj}\label{conj3} The algebra  $\mathcal{H}^{int}_{p,\kappa,\iota}$ acts on $\mathrm{R}\Gamma( \mathfrak{Sh}^{tor}_{K, \Sigma}, \mathcal{V}_{\kappa,K,\Sigma})$ and $\mathrm{R}\Gamma( \mathfrak{Sh}^{tor}_{K, \Sigma}, \mathcal{V}_{\kappa, K,\Sigma}(-D_{K,\Sigma}))$.
\end{conj}

Observe that Conjecture \ref{conj2} implies Conjecture \ref{conj1} by
Proposition \ref{prop-imply-conj}, and Conjecture \ref{conj3}
obviously implies Conjecture \ref{conj2}.

\section{Shimura varieties of symplectic type} \label{section-symplectictype}

This section is dedicated to Shimura varieties of symplectic type.  We
prove the missing part of Theorem \ref{thm-compact}, and we state in
Theorem \ref{main-thm-symplectic} some partial results towards
Conjecture \ref{conj3} which are proved in Section \ref{section-local-model-symp}.

\subsection{Group theoretic data}

Let $F$ be a totally real field with $[F:\qq] = d$ and let $(V, \Psi)$
be a symplectic $F$-vector space of dimension $2g$, with basis
$e_1, \cdots, e_{2g}$ and $\Psi(e_i, e_j) = 0$ if $j \neq 2g +1-i$ and
$\Psi(e_i, e_{2g-i+1}) = 1$ if $1 \leq i \leq g$.  Let
$V_0 = \langle e_1, \cdots, e_g \rangle$ and
$V_1 = \langle e_{g+1}, \cdots, e_{2g} \rangle$ be sub $F$-vector
spaces of $V$. The pairing $\Psi$ on $V$ restricts to a perfect
pairing between $V_0$ and $V_1$.  Let $G$ be the algebraic group over
$\qq$ defined by
$$G(R) = \{ (g, \nu) \in \mathrm{GL}_{F}(V \otimes_\qq R) \times (F
\otimes_\qq R)^\times |~\forall v,w \in V\otimes R, ~\Psi ( g v, g w) =
\nu \Psi (v,w) \}$$ for any $\qq$-algebra $R$.  Let $G_1$ be the
subgroup of $G$ whose elements have their similitude factor $\nu$ in
$\mathbb{G}_m$ embedded diagonaly in
$ \mathrm{Res}_{F/\qq} \mathbb{G}_{m}$. Let also $G^{der}$ be the
derived subgroup of $G$ defined by the condition $\nu =1$.

Let $T^{der}$ be the diagonal maximal torus of $G^{der}$:
$T^{der} = \{\mathrm{diag}(t_1, \cdots, t_g, t_{g}^{-1}, \cdots,
t_{1}^{-1})\}$ with $t_i \in \mathrm{Res}_{F/\qq} \mathbb{G}_m$. The
center $Z$ of $G$ is the group $ \mathrm{Res}_{F/\qq} \mathbb{G}_{m}$
embedded diagonally in $G$.  Let $T$ be the maximal torus of $G$,
which is generated by $T^{der}$ and $Z$. Let $Z_1$ be the center of
$G_1$, the image of $ \mathbb{G}_{m}$ embedded diagonally in
$G_1$. The maximal torus of $G_1$ is generated by $Z_1$ and $T^{der}$.

Let $X^\star(T)$ be the group of characters of $T$ defined over
$\overline{\qq}$, identified with tuples
$\kappa = (k_{1, \sigma}, \cdots, k_{g, \sigma}; k_{\sigma})_{\sigma
  \in \mathrm{Hom}(F, \overline{\qq})} \in \ZZ^{(g+1) d}$ satisfying
the condition $k_\sigma = \sum_i k_{i, \sigma} ~\mod 2$, via the
pairing:
$$ \langle \kappa, (zt_1, \cdots, z t_g, z t_g^{-1}, \cdots,
zt_1^{-1}) \rangle = \prod_{\sigma} \Bigl ( \sigma(z)^{k_\sigma}
  \prod_{i=1}^g \sigma(t_i)^{k_{i, \sigma} } \Bigr )$$

  We denote by $P \subset G$ the Siegel parabolic, i.e., the
  stabilizer of the Lagrangian subspace $V_0$. We denote by $M$ its
  Levi quotient (which we identify with the standard Levi subgroup of
  $P$).  Note that
  $M \simeq \mathrm{Res}_{F/\qq} \mathrm{GL}_g \times
  \mathrm{Res}_{F/\qq} \mathbb{G}_m$.

  The roots of $G$ corresponding to the Lie algebra of $M$ are by
  definition the compact roots, the other roots are called
  non-compact.  We declare that a character of $T$ is dominant for $M$
  if for all $\sigma \in \mathrm{Hom}(F, \overline{\qq})$,
  $k_{1, \sigma} \geq \cdots \geq k_{g, \sigma}$.  So our choice of
  Borel subgroup in $M$ is the usual upper triangular Borel.

  A character of $T$ is dominant for $G$ if for all
  $\sigma \in \mathrm{Hom}(F, \overline{\qq})$,
  $0 \geq k_{1, \sigma} \geq \cdots \geq k_{g, \sigma} $ (so our
  choice of non-compact positive roots are those corresponding to
  $\mathfrak{g}/\mathfrak{p}$).

We have similar definitions for $G_1$ and we denote by $T_1$, $M_1$,
\ldots the intersections of $T$, $M$, \ldots with $G_1$.  Weights for
$T_1$ are labelled by
$$\big((k_{1, \sigma}, \cdots, k_{g, \sigma})_{\sigma \in
  \mathrm{Hom}(F, \overline{\qq})}; k\big) \in \ZZ^{g d} \times \ZZ$$
with the condition that $k = \sum_{i, \sigma} k_{i,\sigma}~\mod~2$ and
a weight $\kappa$ as above pairs with an element $t \in T_1$ via the
formula:
  $$ \langle \kappa, (zt_1, \cdots, z t_g, z t_g^{-1}, \cdots,
 zt_1^{-1})\rangle =z^{k}\prod_{\sigma} \Bigl ( \prod_{i=1}^g
 \sigma(t_i)^{k_{i, \sigma} } \Bigr )$$

 \subsection{Shimura varieties of symplectic type in characteristic
   $0$} Let $\mathcal{H}^{±}_g$ be the Siegel space of symmetric matrices
 $M = A + i B \in \mathrm{M}_{g\times g}(\C)$ with $B$ definite
 (positive or negative). Let
 $X = (\mathcal{H}^{±}_g)^{ \mathrm{Hom}(F, \overline{\qq})}$ and $X_1 = (\mathcal{H}^{+}_g)^{ \mathrm{Hom}(F, \overline{\qq})} \cup (\mathcal{H}^{-}_g)^{ \mathrm{Hom}(F, \overline{\qq})} \subset X$. The group
 $G(\mathbb{R})$ acts on $X$ and its subgroup $G_1(\mathbb{R})$ stabilizes $X_1$.  The pairs $(G, X)$ and $(G_1, X_1)$ are
 Shimura data.  The Siegel parabolic $P$ is a representative of the
 conjugacy class of $P_{\mu}$.

Let $K \subset G(\mathbb{A}_f)$ be a compact open subgroup. We assume
that $K = \prod_\ell K_\ell$ and that $K$ is neat. We also assume
that $p$ is unramified in $F$ so $G(\qq_p) = \prod_{v \mid p}
\mathrm{GSp}_{2g}(F_v)$  for unramified extensions $F_v$ of
$\qq_p$. We  further assume  that  $K_p = \prod K_v \subset  G(\ZZ_p)$
where $K_v$ is either $\mathrm{GSp}_{2g}(\ocal_{F_v})$ or
$\mathrm{Si}(v)$, the Siegel parahoric subgroup of  elements with
reduction mod $p$  in $P(\ocal_{F_v}/p)$.\footnote{We could actually
  allow any parahoric level structure as in \cite[Section 3.3.1]{BCGP}, but in the sequel we shall only need to work at hyperspecial level or Siegel parahoric level.}

We let
$Sh_K(\mathbb{C}) = G(\qq) \backslash( X \times G(\mathbb{A}_f)/K)$ be
the quotient.  Let $G(\qq)^+ \subset G(\qq)$ be the subgroup of elements whose similitude factor is totally positive.

By strong approximation, we can write 

$$ G(\mathbb{A}_f) = \coprod_c G(\qq)^+ c K$$ where the elements $c \in G(\mathbb{A}_f)$ are such that  the elements $\nu(c)$ range through a set of representatives of $ F^{\times,+} \backslash (\mathbb{A}_f \otimes F)^\times / \nu(K)$ and we find that $ Sh_K(\mathbb{C}) = \coprod_c \Gamma(c, K) \backslash \mathcal{H}_g^{\mathrm{Hom}(F, \overline{\qq})}$ where $\Gamma(c,K) = G(\qq)^+ \cap c K c^{-1}$.

This is an algebraic variety over $\C$, and it has a
canonical model $Sh_K$ over $\qq$. The Shimura variety is not of PEL
type, and it is therefore useful to introduce another Shimura variety
which is of PEL type.  We  begin by rewriting
$Sh_K(\mathbb{C}) = (G(\qq)\cap K_p) \backslash( X \times
G(\mathbb{A}^p_f)/K^p)$ and then consider the Shimura variety of PEL
type (in fact an infinite union of such)
$\widetilde{Sh}_K (\mathbb{C}) = (G_1(\qq) \cap K_p) \backslash( X_1
\times G(\mathbb{A}^p_f)/K^p)$. 

By strong approximation, we find that
$$ G(\mathbb{A}^p_f) = \coprod_c (G_1(\qq)^+\cap K_p) c K^p$$ where
$c \in G(\mathbb{A}^p_f)$ ranges through a set of representatives of
$ \ZZ_{(p)}^{\times,+} \backslash(\mathbb{A}^p_f \otimes F)^\times /
\nu(K^p)$ and we find that
$ \widetilde{Sh}_K(\mathbb{C}) = \coprod_c \Gamma_1(c, K) \backslash
\mathcal{H}_g^{\mathrm{Hom}(F, \overline{\qq})}$ where
$\Gamma_1(c,K) = G_1(\qq)^+ \cap c K c^{-1}$.  It has a canonical
model $\widetilde{Sh}_K \rightarrow \Spec~\qq$ and we have a natural
morphism:
$$ \widetilde{Sh}_K \rightarrow Sh_K.$$
On the set of  geometric connected components, this map is given by :   $\ZZ_{(p)}^{\times,+} \backslash(\mathbb{A}^p_f \otimes F)^\times / \nu(K^p) \rightarrow \ocal_{F,p}^{\times, +} \backslash(\mathbb{A}^p_f\otimes F)^\times / \nu(K^p)$.

If we let $c \in G(\mathbb{A}^p_f)$ so that $\nu(c)$ defines
geometrically connected components $(\widetilde{Sh}_K)_c$ of
$\widetilde{Sh}_K$ and $(Sh_K)_c$ of $Sh_K$, then the corresponding
map $ (\widetilde{Sh}_K)_c \rightarrow (Sh_K)_c$ is a finite Galois
cover with group
$$\Delta(K) = (\ocal_{F}^{\times, +} \cap \nu(K^p))/ \nu(
\ocal_{F}^\times \cap K^p),$$ where the first intersection is taken in
$(\mathbb{A}^p_f \otimes F)^\times$ and the second in
$G(\mathbb{A}^p_f)$, with $\ocal_{F}^\times$ embedded diagonally in
the center of $G$.

\subsection{Integral models} We observe that  $\widetilde{Sh}_K$ has a
canonical model  $\widetilde{\mathfrak{Sh}}_K$ over $\ZZ_{(p)}$ which
is a moduli space of abelian varieties of dimension $gd$ with
$K$-level structure, prime to $p$ polarization, and an action of $\ocal_F$. 
More precisely, $\widetilde{\mathfrak{Sh}}_K$ represents the
functor over $\ZZ_{(p)}$ that parametrizes  equivalence classes of $(A, \iota, \lambda,
\eta, \eta_p)$, where:
\begin{enumerate}
\item $A \rightarrow \Spec~R$ is an abelian scheme,
\item $\iota: \ocal_F \rightarrow \mathrm{End} (A) \otimes \ZZ_{(p)}$ is an action,
\item $\mathrm{Lie} (A)$ is a locally free $\ocal_F \otimes_{\ZZ} R$-module of rank $g$, 
\item $\lambda: A \rightarrow A^t$ is a prime to $p$, $\ocal_F$-linear
  quasi-polarization,
  \item $\eta$ is a $K^p$-level structure, 
  \item $\eta_p$ is a $K_p$-level structure. 
\end{enumerate}

Let us spell out the definition of $K^p$-level structure. We may
assume without loss of generality that $S = \Spec~R$ is connected, and we
fix $\overline{s}$ a geometric point of $S$. The adelic Tate module
${\HH}_1( A\vert_{\overline{s}}, \A^{\infty,p})$ carries a
symplectic Weil pairing \[<,>_{\lambda}: {\HH}_1(
  A\vert_{\overline{s}}, \A^{\infty,p}) \times {\HH}_1(
  A\vert_{\overline{s}}, \A^{\infty,p}) \rightarrow {\HH}_1(
  \mathbb{G}_m\vert_{\overline{s}}, \A^{\infty,p})\] or equivalently
an $F$-linear symplectic pairing: \[<,>_{1,\lambda}: {\HH}_1( A\vert_{\overline{s}}, \A^{\infty,p}) \times {\HH}_1( A\vert_{\overline{s}}, \A^{\infty,p}) \rightarrow {\HH}_1( \mathbb{G}_m\vert_{\overline{s}}, \A^{\infty,p}) \otimes F.\] 
The level structure $\eta$ is  a  $K^p$-orbit of pairs of
isomorphisms~$(\eta_1,\eta_2)$, where
(with~$V$ the standard symplectic space defined above):
\begin{enumerate}
\item An $\ocal_F$-linear isomorphism of $\Pi_1(S, \overline{s})$-modules $\eta_1:   V \otimes_{\ZZ} \A^{\infty,p}  \simeq {\HH}_1( A\vert_{\overline{s}}, \A^{\infty,p})$.
\item An  $\ocal_F$-linear isomorphism of $\Pi_1(S, \overline{s})$-modules $\eta_2:   F \otimes_{\ZZ} \A^{\infty,p}  \simeq F \otimes_{\ZZ} {\HH}_1( \mathbb{G}_m\vert_{\overline{s}}, \A^{\infty,p})$.
\end{enumerate}
We moreover impose that the following diagram is commutative:
\begin{eqnarray*}\xymatrix{ V \otimes_{\ZZ} \A^{\infty,p}   \times V \otimes_{\ZZ} \A^{\infty,p}   \ar[rr]^{\eta_1 \times \eta_1} \ar[d]^{<,>_1}& & {\HH}_1( A\vert_{\overline{s}}, \A^{\infty,p}) \times {\HH}_1( A\vert_{\overline{s}}, \A^{\infty,p})  \ar[d]^{<,>_{1,\lambda}}\\ 
 F \otimes_{\ZZ} \A^{\infty,p}   \ar[rr]^{\eta_2} & &F \otimes_{\ZZ} {\HH}_1( \mathbb{G}_m\vert_{\overline{s}}, \A^{\infty,p})}
\end{eqnarray*}

The $K_p$ level structure $\eta_p$ is the data, for each $v \mid p$ such that $K_v = \mathrm{Si}(v)$, of a maximal totally isotropic subgroup $H_v \subset A[v]$.

A map between  quintuples $(A, \iota, \lambda, \eta, \eta_p)$ and $(A',
\iota', \lambda', \eta', \eta'_p)$ is an $\ocal_F$-linear prime to $p$
quasi-isogeny (in the sense of \cite[Definition 1.3.1.17]{MR3186092}) $f: A \rightarrow A'$  such that
\begin{itemize}
\item $f^{\star} \lambda = r \lambda'$ for a locally constant function
  $r: S \rightarrow \ZZ_{(p)}^{\times,+}$, 
\item $f (\eta_p) = \eta'_p$, and
\item   $ {\HH}_1(f) \circ \eta= \eta'$.
\end{itemize}
This last condition means that $\eta'$ is defined by ${\HH}_1(f) \circ \eta_1=  \eta'_1$ and $\eta'_2 = r^{-1} \eta_2$. Also, we have denoted $\ZZ_{(p)}^{\times,+}=\qq^\times_{>0}\cap\ZZ_{(p)}^\times$.

\begin{rem}
  Note that we allow the similitude factor in the level structure to
  be in $\A^{\infty,p} \otimes_{\qq} F(1)$, but we only allow
  quasi-isogenies with similitude factor in $\A^{\infty,p}(1)$.
\end{rem}

We now define an action of $(\ocal_{F})^{\times, +}_{(p)}$ on $\widetilde{Sh}_K$ by scaling the polarization. Namely, $x \in (\ocal_{F})^{\times, +}_{(p)}$ sends $(A, \iota, \lambda, \eta, \eta_p)$ to $(A, \iota, x\lambda, x\eta, \eta_p)$ where $x\eta = (\eta_1, x \eta_2)$. 

This action restricts to a trivial action on  the subgroup $\nu(K^p \cap \ocal_{F, (p)}^\times)$ (where $\ocal_{F, (p)}^\times$ is embedded diagonally in $G(\mathbb{A}_f)$),  because for any $x \in K^p \cap \ocal_{F, (p)}^\times$, the multiplication by $x: A \rightarrow A$ identifies the points: $(A, \iota, \lambda, \eta, \eta_p)$ and  $(A, \iota, x^{-2}\lambda, x^{-2}\eta, \eta_p)$. 

We therefore get an action of
$\Delta = \ocal_{F, (p)}^{\times, +}/\nu(K^p \cap \ocal_{F,
  (p)}^\times)$.  One can show that this action is free \cite[Lemma
3.3.13]{BCGP}. The group $\Delta$ is infinite, but the stabilizer of
any connected component of $\widetilde{Sh}_K$ is finite.

The \'etale surjective map $\widetilde{Sh}_K \rightarrow Sh_K$
identifies  $Sh_K$  as the quotient of $\widetilde{Sh}_K$ by the group
$\Delta$. The action of the group $\Delta$ extends to a free action on
$\widetilde{\mathfrak{Sh}}_K$. We can form the quotient of
$\widetilde{\mathfrak{Sh}}_K$ by $\Delta$    and this defines an
integral model $\mathfrak{Sh}_K$ for  $Sh_K$ over $\ZZ_{(p)}$ (see
\cite[Section 3.3]{BCGP}).

\subsection{Compactifications}

We have smooth toroidal compactifications
$\widetilde{\mathfrak{Sh}}^{ tor} _{K, \Sigma}$ for suitable choices
of polyhedral cone decompositions. The action of $\Delta$ extends to
an action on $\widetilde{\mathfrak{Sh}}^{tor} _{K, \Sigma}$ and this
action is free so we get a smooth integral toroidal compactification
$\mathfrak{Sh}^{tor} _{K, \Sigma}$ of $\mathfrak{Sh}_K$.  See Section
3.5 of \cite{BCGP}.

\subsection{Integral automorphic sheaves} 

Let $E'$ be a galois closure of $F$, $\lambda'$ be a place of $E'$ above $p$ and $\ocal_{E', \lambda'}$ the localization of $\ocal_{E'}$ at $\lambda$. 
Let $\mathrm{Rep}_{\ocal_{E',\lambda'}} (M_1)$ be the category of representations of $M_1$  over
finite free $\ocal_{E',\lambda'}$-modules. Let $\mathrm{FL}_{G,X} =
G/Q = G_1/Q_1$ be the flag variety.

By \cite[Proposition 6.9]{MR2968629}, over
$\widetilde{\mathfrak{Sh}}_{K, \Sigma}^{tor}$, the first relative de
Rham homology group has a canonical extension
$\mathcal{H}_{1, dR}(A/\widetilde{\mathfrak{Sh}}_{K,
  \Sigma}^{tor})^{can}$, it carries the Hodge filtration
$$0 \rightarrow \omega_{A^t} \rightarrow \mathcal{H}_{1,
  dR}(A/\widetilde{\mathfrak{Sh}}_{K, \Sigma}^{tor})^{can} \rightarrow
\mathrm{Lie}(A) \rightarrow 0$$ and a pairing
$\langle\ , \rangle_\lambda$ induced by the polarization.  We can
consider the principal $G_1$-torsor
$\widetilde{P}_{K,\Sigma} \rightarrow \widetilde{Sh}_{K,
  \Sigma}^{tor}$ of isomorphisms between
$\mathcal{H}_{1, dR}(A/\widetilde{\mathfrak{Sh}}_{K,
  \Sigma}^{tor})^{can}, \langle \ , \rangle_\lambda$ and $(V, \Psi)$.

We therefore obtain the following diagram:
\begin{eqnarray*}
\xymatrix{ & \widetilde{\mathfrak{P}}_{K,\Sigma} \ar[rd]^{\alpha} \ar[ld]_\beta & \\
\widetilde{\mathfrak{Sh}}_{K, \Sigma}^{tor} & & \mathrm{FL}_{G,X}}
\end{eqnarray*}
where $\alpha$ is  $G_1$-equivariant and $\beta$ is a $G_1$-torsor. 
Using this $G_1$-torsor we can define a functor  $ \mathrm{Rep}_{\ocal_{E',\lambda'}} (M_1)  \rightarrow VB ( \widetilde{\mathfrak{Sh}}_{K, \Sigma}^{tor}).$

Let $\kappa = ((k_{1, \sigma}, \cdots, k_{g, \sigma})_\sigma); k)$ be
a dominant weight for $M_1$. There is an associated Weyl
representation $V_\kappa$ and we denote by
$\mathcal{V}_{\kappa, K, \Sigma}$ the locally free sheaf corresponding
to it via the above functor.

\begin{ex} Over $\ocal_{E', \lambda}$ we have $\omega_A = \oplus_{\sigma} (\omega_{A})_\sigma$. Let $\sigma_0 \in \mathrm{Hom}(F, E')$. Let  $\kappa = ((k_{1, \sigma}, \cdots, k_{g, \sigma})_\sigma); k)$ with $k_{i,\sigma} = 0$ if $\sigma \neq \sigma_0$ and  $(k_{1, \sigma_0}, \cdots, k_{g, \sigma_0}) =  (0, \cdots, 0, -1)$ and $k=1$, then $\mathcal{V}_{\kappa, K, \Sigma} = \mathrm{Lie}(A)_{\sigma_0}$.
 Therefore, if  $\kappa = ((k_{1, \sigma}, \cdots, k_{g, \sigma})_\sigma); k)$ with $k_{i,\sigma} = 0$ if $\sigma \neq \sigma_0$ and  $(k_{1, \sigma_0}, \cdots, k_{g, \sigma_0}) =  (1,0, \cdots, 0)$ and $k=-1$, then $\mathcal{V}_{\kappa, K, \Sigma} = (\omega_A)_{\sigma_0}$.
\end{ex}

 Let $Z_s(G)$ be the subgroup of the center of $G$ equal to the kernel of the norm map and we let $G^c = G/Z_s(G)$. We now restate and prove the missing part of Theorem \ref{thm-compact}.
 \begin{proposition} There is a  natural commutative diagram: 
\begin{eqnarray*}
\xymatrix{ & {\mathfrak{P}}_{K,\Sigma} \ar[rd]^{\alpha'} \ar[ld]_{\beta'} & \\
{\mathfrak{Sh}}_{K, \Sigma}^{tor} & & \mathrm{FL}_{G,X}}
 \end{eqnarray*}
where $\alpha'$ is  $G$-equivariant and $\beta'$ is a $G^c$-torsor.
 \end{proposition}
 
 \begin{demo} We have a commutative diagram \begin{eqnarray*}
\xymatrix{ & \widetilde{\mathfrak{P}}_{K,\Sigma} \ar[rd]^{\alpha} \ar[ld]_\beta & \\
\widetilde{\mathfrak{Sh}}_{K, \Sigma}^{tor} & & \mathrm{FL}_{G,X}}
 \end{eqnarray*}
 We consider the $G^c$-torsor $\widetilde{\mathfrak{P}}_{K,\Sigma} \times^{G_1} G^c$. We claim that this torsor descends to ${\mathfrak{Sh}}_{K, \Sigma}^{tor}$.  In order to prove this we must exhibit a descent datum for the action of $\Delta$.
 
Let $(A, \iota, \lambda, \eta, \eta_p, \Psi)$ be an $R$-point of $\widetilde{\mathfrak{P}}_{K,\Sigma} \times^{G_1} G$, where $\Psi: V \otimes R \rightarrow \mathcal{H}_{1,dR}(A/R)^{can}$ is a symplectic isomorphism up to a similitude factor in $(F \otimes R)^\times$. 
 For any $x \in  K^p \cap \ocal_{F, (p)}^\times$, the multiplication by $x^{-1}: A \rightarrow A$ induces a natural map:
 $$ \mathcal{H}_{1, dR}(A/R)^{can} \rightarrow \mathcal{H}_{1, dR}(A/R)^{can}$$ which is scalar multiplication by $x^{-1}$ in the trivialization $\Psi$. We observe that $x \in Z_s(G)(\ocal_{F})$, and therefore in $\widetilde{\mathfrak{P}}_{K,\Sigma} \times^{G_1} G^c$ the multiplication by $x^{-1}$ induces a  canonical isomorphism between $(A, \iota, \lambda, \eta, \eta_p, \Psi)$ and $(A, \iota, x^{2}\lambda, \eta, \eta_p, \Psi)$. 
Therefore, we can simply define an  action of  $(\ocal_{F})_{(p)}^{\times, +}$ on $\widetilde{\mathfrak{P}}_{K,\Sigma} \times^{G_1} G^c$, by sending $(A, \iota, \lambda, \eta, \eta_p, \Psi)$ to $(A, \iota, x\lambda, x\eta, \eta_p, \Psi)$, and this action passes to the quotient to an action of $\Delta$. 

We can therefore descend the torsor $\widetilde{\mathfrak{P}}_{K,\Sigma} \times^{G_1} G^c$ to a $G^c$-torsor $\mathfrak{P}_{K, \Sigma} \rightarrow \mathfrak{Sh}^{tor}_{K,\Sigma}$. Moreover, one descends similarly the $P^c$-reduction of  $\widetilde{\mathfrak{P}}_{K,\Sigma} \times^{G_1} G^c$, and therefore get a map $\beta':  \mathfrak{P}_{K, \Sigma} \rightarrow  \mathrm{FL}_{G,X}$.
\end{demo}

\begin{coro}\label{coro-construction-autoVB} There is a functor  $\mathrm{Rep}_{\ocal_{E',\lambda'}} (M/Z_s(G)) \rightarrow VB ( \mathfrak{Sh}_{K, \Sigma}^{tor})$  which makes the following diagram commute: 
\begin{eqnarray*}
\xymatrix{ \mathrm{Rep}_{\ocal_{E',\lambda'}} (M_1)  \ar[r] & VB ( \widetilde{\mathfrak{Sh}}_{K, \Sigma}^{tor}) \\
 \mathrm{Rep}_{\ocal_{E',\lambda'}} ( M/Z_s(G) ) \ar[u]   \ar[r] & VB ( \mathfrak{Sh}_{K, \Sigma}^{tor}) \ar[u]}
\end{eqnarray*}
\end{coro}

\begin{rem} Let $\kappa = (k_{1, \sigma}, \cdots, k_{g, \sigma}; k_{\sigma})_{ \mathrm{Hom}(F, \overline{\qq})}$ be a  dominant weight  for $M$, with associated Weyl representation ${V}_\kappa$. The representation ${V}_\kappa$ belongs to $\mathrm{Rep}_{\ocal_{E',\lambda'}} (M/Z_s(G))$ if and only if  $k_\sigma = k$ is independent of $\sigma$. 
 \end{rem}

\subsection{Minuscule coweights and the main theorem in the symplectic case}    
\subsubsection{The general formula for minuscule coweights}
 We take $E'$ to be the Galois closure of $F$.  We let $\iota: E'
 \rightarrow \overline{\qq}_p$. Let $p$ be a prime unramified in $F$
 and  denote by $\mathfrak{p}_1, \cdots, \mathfrak{p}_m$ the prime ideals above $p$ in $F$. We let $I = \mathrm{Hom}(F, E')$ and  for each $1 \leq i \leq n$, let $$I_i = \{ \sigma \in \mathrm{Hom}( F, E'), ~\iota \circ \sigma~\textrm{induces the $\mathfrak{p}_i$-adic valuation on $F$} \}.$$

 We consider a weight
 $\kappa = ((k_{1,\sigma}, \cdots, k_{g, \sigma})_{\sigma \in I} ; k)$
 for $G$, where we assume that the parity of $\sum_i k_{i, \sigma}$ is
 independent of $\sigma$, and we choose $k \in \ZZ$ such that
 $\sum_i k_{i, \sigma} = k~\mod 2$.  We have, by Corollary
 \ref{coro-construction-autoVB}, a sheaf
 $\mathcal{V}_{\kappa, K, \Sigma}$ on
 $\mathfrak{Sh}^{tor}_{K,
   \Sigma}$. 

The spherical  Hecke algebra at $p$, denoted $\mathcal{H}_p$, is a  tensor  product of the spherical   Hecke algebras
$\mathcal{H}_{\mathfrak{p}_i}$ at each prime $\mathfrak{p}_i$ dividing $p$. Each of the algebras
$\mathcal{H}_{\mathfrak{p}_i}$ contains the following familiar
characteristic functions of  double cosets: $$T_{\mathfrak{p}_i}^{naive} = \mathrm{GSp}_{2g}(
\ocal_{F_{\mathfrak{p}_i}}) \mathrm{diag} (\mathfrak{p}_i^{-1}1_g,  1_g)
\mathrm{GSp}_{2g}( \ocal_{F_{\mathfrak{p}_i}})$$ and
$S_{\mathfrak{p}_i}^{naive} =  \mathrm{GSp}_{2g}(
\ocal_{F_{\mathfrak{p}_i}}) \mathrm{diag} (\mathfrak{p}^{-1}_i,
\mathfrak{p}^{-1}_i) \mathrm{GSp}_{2g}( \ocal_{F_{\mathfrak{p}_i}})$ and to
each of these double cosets we can associate a cohomological
correspondence over $\qq$ on the sheaf $\mathcal{V}_{\kappa, K, \Sigma}$ (see Section
\ref{section-coho-vec-bund-Hecke}). We have added the superscript  ``naive'' because the cohomological correspondence is not suitably normalized in general.

By definition $T_{\mathfrak{p}_i}^{naive} = T_{\lambda}$ (see Section
\ref{section-Hecke-algebras}) for the cocharacter $ \lambda:  t
\mapsto (\prod_{\sigma \in I_i} \mathrm{diag} ( t^{-1}1_g, 1_g)) \\ \times
(\prod_{\sigma \notin I_i}\mathrm{diag}(1_g, 1_g))$  which is given in coordinates by
$$ \Bigl (\prod_{\sigma \notin I_i} (0, \cdots, 0; 0)_\sigma \Bigr )
\times \Bigl ( \prod_{\sigma \in I_i} (-\tfrac{1}{2}, \cdots, -
\tfrac{1}{2}; -\tfrac{1}{2})\Bigr ) ,$$
and $S_{\mathfrak{p}_i}^{naive} = T_{\mu}$ for the cocharacter $ \mu:
t \mapsto (\prod_{\sigma \in I_i} \mathrm{diag} (t^{-1}1_g, t^{-1} 1_g)) \times
(\prod_{\sigma \notin I_i}\mathrm{diag}(1_g, 1_g))$  which is given in coordinates by
$$ \Bigl (\prod_{\sigma \notin I_i} (0, \cdots, 0; 0)_\sigma \Bigr )
\times \Bigl ( \prod_{\sigma \in I_i} (0, \cdots, 0;  -1)\Bigr ) .$$

\begin{rem}\label{rem-left-rightaction} The goal of this remark is to justify  the use of  the double class $$\mathrm{GSp}_{2g}(
  \ocal_{F_{\mathfrak{p}_i}}) \mathrm{diag} (\mathfrak{p}_i^{-1}1_g,
  1_g) \mathrm{GSp}_{2g}( \ocal_{F_{\mathfrak{p}_i}})$$ rather than
  the double class
  $$\mathrm{GSp}_{2g}( \ocal_{F_{\mathfrak{p}_i}}) \mathrm{diag}
  (\mathfrak{p}_i1_g, 1_g) \mathrm{GSp}_{2g}(
  \ocal_{F_{\mathfrak{p}_i}})$$ which appears in some classical
  references.  The difference can be explained as follows: the adelic
  points of $G$ act on the right on the tower of Shimura varieties,
  and therefore on the left on the cohomology; in the classical theory
  of modular forms, one usually defines a right action of the Hecke
  algebra on the space of modular forms.  Let us give some more
  details. To avoid complications with non-PEL Shimura varieties, we
  will assume that our totally real field is $\qq$. For simplicity, we
  also assume that $K \subset \mathrm{GSp}_4(\hat{\ZZ})$.  Associated
  to the element $g = \mathrm{diag}(p^{-1} 1_g, 1_g)$ we have a
  correspondence over $\qq$:
\begin{eqnarray*}
\xymatrix { &  Sh^{tor}_{gKg^{-1} \cap K, \Sigma''}\ar[ld]_{p_2} \ar[rd]^{p_1}&  \\
Sh^{tor}_{ K, \Sigma'} & & Sh^{tor}_{ K, \Sigma} }
\end{eqnarray*}
and the relation between $p_1^\star A$ and $p_2^\star A$ is given as follows (at least away from the boundary). There are symplectic isomorphisms
$ \psi_1 :  \hat{{\ZZ}}^{2g}   \rightarrow \mathrm{H}_1(p_1^\star A, \hat{\ZZ}) ~\mod gKg^{-1} \cap K$ and $ \psi_2 :  \hat{\ZZ}^{2g}   \rightarrow \mathrm{H}_1(p_2^\star A, \hat{\ZZ}) ~\mod g^{-1}Kg \cap K$
and  a commutative diagram :
\begin{eqnarray*} 
\xymatrix{  \hat{\ZZ}^{2g}\otimes \qq \ar[r]^{g}  \ar[d]^{\psi_2} &\hat{\ZZ}^{2g}\otimes \qq\ar[d]^{\psi_1} \\
\mathrm{H}_1(p_2^\star A, \hat{\ZZ}) \otimes \qq \ar[r] & \mathrm{H}_1(p_1^\star A, \hat{\ZZ})\otimes \qq}
\end{eqnarray*}
Therefore,  the lattice  $\mathrm{H}_1(p_2^\star A, \hat{\ZZ})$ contains the lattice $\mathrm{H}_1(p_1^\star A, \hat{\ZZ})$.  This means that $p_2^\star A$ appears to be the quotient of $p_1^\star A$ by a Lagrangian subgroup of $p_1^\star A[p]$. The Shimura variety $Sh_{gKg^{-1} \cap K}$ is therefore parametrizing all Lagrangian subgroups of $p_1^\star A$ and the projection $p_2$ is given by ``taking the quotient by the Lagrangian subgroup''. 
 If we had made the choice of $g = \mathrm{diag}( 1_g, p1_g)$, then we would have obtained the transposed correspondence. 

\end{rem}

Motivated by Definition \ref{defi-hecke-algebra}, we now
let
$$T_{\mathfrak{p}_i} = [V_{\lambda}] p^{\langle \lambda,
  \infty(\kappa, \iota)\rangle} = p^{ \langle \lambda, \infty(\kappa,
  \iota) \rangle - \langle \lambda, \rho \rangle}
T_{\mathfrak{p}_i}^{naive} $$ where the last equality follows from the
fact that $\lambda$ is minuscule. We now determine the value of the
coefficient
$ \langle \lambda, \infty(\kappa, \iota)) \rangle - \langle \lambda,
\rho \rangle$.

\begin{lem}\label{lem-symplectic-normalization} We have
  $ \langle \lambda, \infty(\kappa, \iota) \rangle - \langle \lambda,
  \rho \rangle = $
$$ \sum_{\sigma \in I_i} \sup_{1 \leq j \leq g} \left \{
  \frac{\sum_{\ell = 1}^j k_{\ell, \sigma} - \Bigl (\sum_{\ell = j+1}^g
    k_{\ell, \sigma} \Bigr) +k}{2} - \frac{j(j+1)}{2} \right \}$$ 
\end{lem}

\begin{demo} We have $\rho = (-1, -2, \cdots, -g; 0)_{\sigma \in I}$,
  so $(\kappa+ \rho)_\sigma = (-1,- 2, \cdots, -g; k)$ if
  $\sigma \notin I_i$ and
  $(\kappa+\rho)_\sigma = (k_{1, \sigma} - 1, k_{2,\sigma} - 2, \cdots, k_{g, \sigma}
  - g; k)$ if $\sigma \in I_i$.  By definition,
  $\infty(\kappa, \iota)_\sigma$ is the Weyl translate of
  $-(\kappa+ \rho)_\sigma$ in the dominant cone. We have
  $\lambda_\sigma = (0, \cdots, 0; 0)$ if $\sigma \notin I_i$,
  and
  $\lambda_\sigma = (-\frac{1}{2}, \cdots, -\frac{1}{2};
  -\frac{1}{2})$.  We clearly have
  $$\langle \lambda, \infty(\kappa, \iota) \rangle - \langle
  \lambda, \rho \rangle = \sum_{\sigma \in I}
  \langle \lambda_\sigma, \infty(\kappa, \iota)_\sigma) \rangle
  - \langle \lambda_\sigma, \rho_\sigma \rangle$$ and the contribution
  to the sum of any $\sigma \notin I_i$ is $0$.

  Let us fix $\sigma \in I_i$ and compute the corresponding pairing at
  $\sigma$.  We need to put each $-(\kappa +\rho)_\sigma$ in the
  dominant cone (i.e., the coordinates on the left of the ``;'' need
  to be non-positive and in decreasing order) and the pairing with
  $\lambda_\sigma$ will amount to taking the sum of the
  coordinates and multiplying it by $-\frac{1}{2}$.

  Now there will be some integer $0 \leq j \leq g$ such that the first
  $j$ coordinates of $-(\kappa + \rho)_{\sigma}$ are non-positive and
  the next $g-j$ coordinates are non-negative. In that case, the
  weight will be put in the dominant form by first changing it to
$$ (- k_{1,\sigma} + 1, \cdots, - k_{j, \sigma} + j, k_{j+1, \sigma} - j-1, \cdots, k_{g,\sigma} - g; -k)$$ (so that
all entries on the left of ``;'' are non-positive) and then applying
an element of the Weyl group of the Levi to put it in the dominant
form. Of course, this last operation is irrelevant for computing the
pairing since we will take the sum of all coordinates.

We therefore find that in this case $$\langle \lambda_\sigma,
\infty(\kappa, \iota)_\sigma \rangle = \frac{\sum_{\ell = 1}^j
  k_{\ell, \sigma} - \Bigl (\sum_{\ell = j+1}^g k_{\ell, \sigma}\Bigr
  ) +k}{2} - \tfrac{1}{2} \sum_{\ell = 1}^j \ell  +
\tfrac{1}{2}\sum_{\ell = j+1}^g \ell.$$

Now we note that for any $j' \neq j$, the corresponding sum on the RHS
above, with $j$ replaced by $j'$, is less than or equal to the sum for
$j$.  The formula then follows from the observation that
$- \langle \lambda_\sigma, \rho_\sigma \rangle = - \frac{1}{2}
\sum_{\ell =1}^g \ell$.
\end{demo} 

\bigskip

Similarly, motivated by Definition \ref{defi-hecke-algebra}, we also
set  $$S_{\mathfrak{p}_i} =   [V_{\mu}] p^{\langle \mu, \infty(\kappa,
  \iota))\rangle} =  p^{ \langle \lambda, \infty(\kappa, \iota)) \rangle - \langle \lambda, \rho \rangle} S_{\mathfrak{p}_i}^{naive}$$ and it is elementary to check that 
 $S_{\mathfrak{p}_i} = p^{\sum_{\sigma \in I_i} k} S_{\mathfrak{p}_i}^{naive}$. 

\begin{rem} The coweights $\lambda$ and $\mu$ (for varying primes
  $\mathfrak{p}_i$) are the only minuscule coweights (up to twist by a central coweight). Unfortunately,
  our techniques don't allow us to deal with non-minuscule coweights
  well, and this is why we do not consider the entire Hecke algebra.
\end{rem}

\subsubsection{The main result in the symplectic case} We have the
following partial result towards Conjecture \ref{conj3}:

\begin{thm}\label{main-thm-symplectic} There are algebra  morphisms 
$$\otimes_{0 \leq i \leq m}\ZZ[T_{\mathfrak{p}_i}, S_{\mathfrak{p}_i}, S_{\mathfrak{p}_i}^{-1} ] \rightarrow \mathrm{End} \big( \mathrm{R}\Gamma( \mathfrak{Sh}^{tor}_{K, \Sigma}, \mathcal{V}_{\kappa_0, K, \Sigma})\big)$$ and 
$$\otimes_{0 \leq i \leq m}\ZZ[T_{\mathfrak{p}_i}, S_{\mathfrak{p}_i}, S_{\mathfrak{p}_i}^{-1}] \rightarrow \mathrm{End} \big( \mathrm{R}\Gamma( \mathfrak{Sh}^{tor}_{K, \Sigma}, \mathcal{V}_{\kappa_0, K, \Sigma}(-D_{K, \Sigma}))\big)$$
extending the action over $\qq$ (see Section \ref{section-coho-vec-bund-Hecke}). 
\end{thm}

We now discuss some special cases of the theorem.

\subsubsection{$G = \mathrm{GL}_2/\qq$}

The sheaf of weight $k$ modular forms corresponds to the weight $(k;-k) := \kappa $. From the  formula of Lemma \ref{lem-symplectic-normalization},  we deduce that 
$T_p = p^{-\inf\{1,k\}} T_p^{naive}$ and $S_p = p^{-k} S_p^{naive}$ and that 
$$\mathcal{H}_{p, \kappa}^{int} = \ZZ[T_p, S_p, S_p^{-1}].$$ 
It follows from Theorem \ref{main-thm-symplectic}, that
$\mathcal{H}_{p, \kappa}^{int}$ acts on the cohomology complex of
weight $k$ modular forms. Conjecture \ref{conj3} is thus proven in
this case.

\subsubsection{$G = \mathrm{GL}_2/F$}\label{sectionHilbert}

We let $\kappa = ((k_\sigma)_\sigma; k)_\sigma$ with the property that $k$ and all the $k_\sigma$ have the same parity.  Then we find that $$T_{\mathfrak{p}_i} = p^{\sum_{\sigma \in I_i} \sup\{ \frac{k_\sigma +k}{2}-1, \frac{k-k_\sigma}{2}\}} T_{\mathfrak{p}_i}^{naive}$$ and $$S_{\mathfrak{p}_i} = p^{\sum_{\sigma \in I_i} k} S_{\mathfrak{p}_i}^{naive}.$$
We deduce that  $\mathcal{H}_{p, \kappa, \iota}^{int}  = \otimes_{i} \mathcal{H}^{int}_{\mathfrak{p}_i, \kappa}$ where 

$$\mathcal{H}^{int}_{\mathfrak{p}_i, \kappa} = \ZZ[T_{\mathfrak{p}_i}, S_{\mathfrak{p}_i}, S_{\mathfrak{p}_i}^{-1}].$$
It follows from Theorem \ref{main-thm-symplectic}, that
$\mathcal{H}_{p, \kappa, \iota}^{int}$ acts on the cohomology complex
of weight $\kappa$ modular forms. Conjecture \ref{conj3} is thus also proven in this case.

\subsubsection{$G = \mathrm{GSp}_{2g}/\qq$}

We take $\kappa = ( k_1, \cdots, k_g; - \sum k_i)$. Our choice of the central character is the standard choice in the theory of Siegel modular forms. Indeed,    the sheaf $\mathcal{V}_{\kappa, K, \Sigma}$  has the following elementary description. First, on $\mathfrak{Sh}^{tor}_{K, \Sigma}$, we have a semi-abelian scheme $A$ of dimension $g$,  and we denote by $\omega_A$ the conormal sheaf. %$$ \omega_A = \oplus_{\sigma \in I} \omega_{A, \sigma}.$$ 
We denote  by $\mathcal{T}$ the torsor of trivializations of $\omega_{A}$. This is a $\mathrm{GL}_{g}$-torsor and we let $\pi: \mathcal{T} \rightarrow \widetilde{Sh}^{tor}_{K, \Sigma}$ be the projection. Then, unravelling the definitions, we find that  $\mathcal{V}_{\kappa, K, \Sigma}  = \pi_\star \oscr_{\mathcal{T}}[\kappa^\vee]$ where  $\kappa^\vee = (- k_{g}, \cdots, -k_{1})$ (and $\pi_\star \oscr_{\mathcal{T}}[\kappa^\vee]$ is the subsheaf of $\pi_\star \oscr_{\mathcal{T}}$ of sections which transform by the character $\kappa^\vee$ under the action of the Borel).  We finally obtain  that $$T_p = p^{ \sup_{1 \leq j \leq g} \{ - \sum_{\ell = j+1}^g  k_l - \frac{j(j+1)}{2}\}} T_p^{naive}$$ and $S_p = p^{ - \sum k_i} S_p^{naive}$.

\section{Local model in the symplectic case}\label{section-local-model-symp}

In this section we will prove Theorem \ref{main-thm-symplectic}.  The
actions of the normalized Hecke operators will be defined using the
results of Section \ref{residue}, in particular the construction in
Example \ref{ex:corres}. In order to do this, we need to understand
the integrality properties of the Hecke correspondences with respect
to automorphic vector bundles and also the pullback maps on
differentials. This will be done by using the theory of local models
of Shimura varieties.  Variants of the local models we consider in
this section were introduced in \cite{MR1144439}, \cite{MR1266495},
\cite{dejong-ppav} and studied further by G\"ortz \cite{gortz}; for a
general introduction to the theory of local models the reader may
consult \cite{PRS}.

After recalling the basic facts about the local models for symplectic
groups, we prove two results, Proposition \ref{prop-formula} (on the
integrality properties of differentials) and Lemma
\ref{lem-weight-formula} (on the integrality properties on automorphic
bundles) which are used to construct normalized Hecke operators on the
local model in Proposition \ref{prop:normal}. We then transport these
computations to the Shimura variety side using Theorem
\ref{thm-main-local-model} and thereby prove Theorem
\ref{main-thm-symplectic}.

\subsection{Definition} Let $g \in \mathbb{Z}_{\geq 1}$ be an integer.  For any positive integer $t$, we let $\mathrm{Id}_t$ be the identity $t\times t$ matrix and we let $K_t$ be the anti-diagonal matrix of size $t \times t$, with coefficients $1$ on the anti-diagonal. When the context is clear, we sometimes write $K$ instead of $K_t$ and $\mathrm{Id}$ instead of $\mathrm{Id}_t$. 

Let $V_0 = \ZZ^{2g}$. We equip $V_0$ with the symplectic pairing $\psi$ given by the matrix 
$$ J=   \begin{pmatrix} % or pmatrix or bmatrix or Bmatrix or ...
      0 & K_g \\
      -K_g & 0 \\
   \end{pmatrix}$$
  We now consider modules $V_1, \cdots, V_{2g-1}= \ZZ^{2g}$ and the following chain:
   $$V_\bullet: V_0 \rightarrow V_1 \rightarrow \cdots \rightarrow V_{2g-1} \rightarrow V_0$$ where 
   the map from $V_i$ to $V_{i+1}$ is given in the canonical basis $(e_1, \cdots, e_{2g})$ of $\ZZ^{2g}$ by the map $e_j \mapsto e_j$ if $j \neq i+1$ and $e_{i+1} \mapsto p e_{i+1}$.    Whenever necessary, indices are taken modulo $2g$ so that $V_{2g}:= V_0$. 
   
   This chain is self-dual. The pairing $\psi$ and the maps in the
   chain induce pairings $\psi: V_r \times V_{2g-r} \rightarrow \ZZ$
   which can be written as $p^2 \psi'$, for a perfect pairing $\psi'$.
  
  Let us fix a set $\emptyset \neq I \subset \{ 0, 1, \cdots, g \}$.  We define the
  local model functor
  $\mathbf{M}_I: \ZZ- \mathbf{ALG} \rightarrow \mathbf{SETS}$ which
  associates to an object $R$ of $\ZZ- \mathbf{ALG}$ the set of
  isomorphism classes of commutative diagrams
  \begin{eqnarray*}
\xymatrix{ V_{i_0}\otimes_{\ZZ} R \ar[r]  & V_{i_1}\otimes_{\ZZ} R \ar[r]  & \cdots  \ar[r] & V_{i_m}\otimes_{\ZZ} R  \\
F_{i_0} \ar[u] \ar[r]  & F_{i_1} \ar[r] \ar[u] & \cdots  \ar[r]  & F_{i_m} \ar[u]}
\end{eqnarray*}
where $i_0 < i_1 \cdots <i_m $ are such that $\{ i_0, \cdots, i_m\} =
I \cup \{2g-i \, |\, i \in I\}$, the modules $F_{i_j}$, for $0 \leq j \leq m$,
are  rank $g$ locally direct factors of $V_{i_j}\otimes_{\ZZ} R$, and they are self dual in the
sense that $F_{2g-i}^\bot = F_{i}$ for all $i \in I$ (with respect to
the pairing $\psi'$). 

The functor $\mathbf{M}_I$ is represented by a projective scheme which
we denote by $M_I$. It is a closed subscheme of a product of
Grassmannians, the embedding being given by the vertical maps of
diagrams as above.

When $\emptyset \neq J \subset I$, there is an obvious map ${M}_I \rightarrow M_J$
given by forgetting the modules $F_j$ for $j \in I \setminus J$.
There is a canonical isomorphism $M_{\{0\}} \simeq M_{\{g\}}$ given by
taking $F_g \subset V_g \otimes_{\ZZ} R$ to be $F_0$, via the
tautological identification of $V_0$ and $V_g$.

\subsection{The affine Grassmannian}
%We define another  chain: 

Let $\mathcal{V}_0 = \F_p[[t]]^{2g}$. We equip $\mathcal{V}_0$ with the symplectic pairing $\psi$ given by the matrix 
$$ J=   \begin{pmatrix} % or pmatrix or bmatrix or Bmatrix or ...
      0 & K_g \\
      -K_g & 0 \\
    \end{pmatrix}$$ We now consider modules
    $\mathcal{V}_1, \cdots,\mathcal{V}_{2g-1}= \F_p[[t]]^{2g}$ and
    the chain:
  $$\mathcal{V}_\bullet = \mathcal{V}_0 \rightarrow \mathcal{V}_1
  \rightarrow \cdots \rightarrow \mathcal{V}_{2g-1} \rightarrow
  \mathcal{V}_0$$ where the map from $\mathcal{V}_i$ to
  $\mathcal{V}_{i+1}$ is given in the canonical basis
  $(e_1, \cdots, e_{2g})$ of $\F_p[[t]]^{2g}$ by the map
  $e_j \mapsto e_j$ if $j \neq i+1$ and $e_{i+1} \mapsto t e_{i+1}$.
  Whenever necessary, indices are taken modulo $2g$ so that
  $\mathcal{V}_{2g}:= \mathcal{V}_0$. Observe that
  $\mathcal{V}_\bullet \otimes_{\F_p[[t]]} \F_p = V_\bullet
  \otimes_{\ZZ} \FF_p$.

   Let $LG$ denote the loop group of $\mathrm{GSp}_{2g}$ over
   $\mathbb{F}_p$. The group $LG$ acts naturally on
   $\mathcal{V}_0 \otimes_{\F_p[[T]]} \F_p((T))$ and therefore it acts
   on the chain $\mathcal{V}_\bullet \otimes_{\F_p[[T]]} \F_p((T))$.
 
   Let $\emptyset \neq I \subset \{ 0, 1, \cdots, g \}$.  Let
   $\mathcal{V}^I_\bullet$ be the subchain of $\mathcal{V}_\bullet$
   where we keep only the modules indexed by elements $i \in I$ and
   $i' = 2g-i$ for $i \in I$.  We denote by $\mathcal{P}_I$ the
   parahoric subgroup of $LG$ of automorphisms of the chain
   $\mathcal{V}^I_\bullet$. For any $I$ as above, we define the affine
   flag variety as the ind-scheme $\mathcal{F}_I:= LG/\mathcal{P}_I$.

   \subsection{Stratification of the local model}

   It was observed by Go\"rtz \cite[\S 5]{gortz} that the fibre
   $\ov{M}_I$
   % {\color{red} (I replaced $\times \Spec \F_p$ with an overline
   %   throughout)}
   of $M_I$ over $\Spec~\F_p$  embeds as a finite union of
   $\mathcal{P}_I$ orbits in $\mathcal{F}_I$.  We recall the
   description of the map $\ov{M}_I \rightarrow \mathcal{F}_I$.  Given
   a diagram \begin{eqnarray*}
               \xymatrix{ V_{i_0}\otimes_{\ZZ} R \ar[r]  & V_{i_1}\otimes_{\ZZ} R \ar[r]  & \cdots  \ar[r] & V_{i_m}\otimes_{\ZZ} R  \\
               F_{i_0} \ar[u] \ar[r] & F_{i_1} \ar[r] \ar[u] & \cdots
                                                               \ar[r]
                                                                                                           &
                                                                                                             F_{i_m}
                                                                                                             \ar[u]}
\end{eqnarray*} corresponding to an $R$-point of $\ov{M}_I$, we can construct a new  diagram: 
\begin{eqnarray*}
\xymatrix{ \mathcal{V}_{i_0}\otimes_{\ZZ} R \ar[r]  & \mathcal{V}_{i_1}\otimes_{\ZZ} R \ar[r]  & \cdots  \ar[r] & \mathcal{V}_{i_m}\otimes_{\ZZ} R  \\
\mathcal{F}_{i_0} \ar[u] \ar[r]  & \mathcal{F}_{i_1} \ar[r] \ar[u] & \cdots  \ar[r]  & \mathcal{F}_{i_m} \ar[u] \\
t\mathcal{V}_{i_0}\otimes_{\ZZ} R \ar[r] \ar[u]  & t\mathcal{V}_{i_1}\otimes_{\ZZ} R \ar[r] \ar[u]  & \cdots  \ar[r] & t\mathcal{V}_{i_m}\otimes_{\ZZ} R \ar[u]}
\end{eqnarray*}
where all the vertical maps are inclusions and  each $\mathcal{F}_{i_j}$ is determined by the property that $\mathcal{F}_{i_j}/ t\mathcal{V}_{i_j}\otimes_{\ZZ} R  = F_{i_j} \hookrightarrow ( \mathcal{V}_{i_j} / t\mathcal{V}_{i_j})\otimes_{\ZZ} R  = V_{i_j} \otimes R$.  The chain $ \mathcal{F}_\bullet$ determines an $R$-point of $\mathcal{F}_I$.

We now recall the combinatorial description of the image of
$\ov{M}_I$ in $\mathcal{F}_I$: Fix a Borel subgroup $B$
of $\mathrm{GSp}_{2g}$ and a maximal torus $T\subset B$. This gives a
base for the root datum of $\mathrm{GSp}_{2g}$ with a corresponding
Dynkin diagram with $g$ vertices and Weyl group $W$ generated by
reflections $s_1,s_2,\dots,s_g$.  Let $\widetilde{W}$ be the extended
affine Weyl group of $\mathrm{GSp}_{2g}$. This is the semi-direct
product of the Weyl group $W$ and the cocharacter group
$\mathrm{X}_{\star}(T)$. It contains as a subgroup $W_{af}$, the affine Weyl
group of $\mathrm{GSp}_{2g}$, which is a Coxeter group with simple
reflections $s_1,s_2,\dots,s_g$ and one affine reflection
$s_{0}$.%\footnote{I prefer to call it $s_0$ rather than $s_{g+1}$} 

We give a concrete realization of $W$ and $\widetilde{W}$ following
\cite{KR}. The group $W$ can be realized as the subgroup of
$\mathcal{S}_{2g}$, the group of permutations of $\{1,2,\dots, 2g\}$,
which satisfy the condition $w(i) + w(2g+i-1) = 2g+1$ for all
$i = 1,\cdots, 2g$.  The extended affine Weyl group of
$\mathrm{GL}_{2g}$ is $\mathbb{Z}^{2g} \rtimes \mathcal{S}_{2g}$. It
can be realized as a subgroup of the group of affine transformations
of $\mathbb{Z}^{2g}$ with $\mathbb{Z}^{2g}$ acting by translation and
$\mathcal{S}_{2g}$ by permutation of the coordinates.  For an element
$v \in \ZZ^{2g}$ we will denote the corresponding translation by
$\mr{t}_v$. The group $\widetilde{W}$ is realized as the centralizer
in $\mathbb{Z}^{2g} \rtimes \mathcal{S}_{2g}$ of the element
$(1,2g)(2, 2g-1)\cdots (g,g+1) \in \mathcal{S}_{2g}$. We have
$s_i = (i,i+1)(2g+1-i, 2g-i)$ for $1 \leq i \leq g-1$,
$s_g = (g, g+1)$ and
$s_0 = \mr{t}_{(-1, 0, \cdots, 0, 1)} \rtimes ( 1,2g)$.  We let
$\mu = (1, \cdots, 1, 0, \cdots 0) \in \mathbb{Z}^{2g}$ (with both $1$
and $0$ repeated $g$ times) be the minuscule coweight corresponding to
our situation.

For each $I$ as above, we let $W_I$ be the subgroup of $\widetilde{W}$
generated by the simple reflexions $s_i, i \notin I$. This is a finite
group since $I \neq \emptyset $.  The $\mathcal{P}_I$ orbits in
$\mathcal{F}_I$ are parametrised by the double cosets
$W_I\backslash \widetilde{W} / W_I$.  For any $w \in \widetilde{W}$ we
denote the orbit corresponding to the double coset $W_I w W_I$ by
$U_{I,w}$ and the orbit closure by $X_{I,w}$.  The orbits included in
$\ov{M}_I$ are parametrized by the finite subset $Adm_I(\mu)$ of
$W_I\backslash \widetilde{W} / W_I$ of $\mu$-admissible elements as
defined in \cite[Introduction]{KR}. The open orbits---these are the
only ones we will need to know explicitly---are parametrized by the
double cosets corresponding to elements in the $W$-orbit of $\mu$
(viewed as translations in $\widetilde{W}$).

As noted earlier, whenever $J \subset I$ we have a map
$M_I \rightarrow M_J$. This map is surjective and $Adm_J(\mu)$ is the
image of $Adm_I(\mu)$ in $W_J\backslash \widetilde{W} / W_J$.

  When $I = \{ 0, \cdots, g\}$, Kottwitz and Rapoport give a description
of the set $Adm_I(\mu)$. This is a subset of
$W_{af}\, c \subset \widetilde{W}$ where
$c = \mr{t}_{(0, \cdots, 0, 1, \cdots,1)}\cdot(1, g+1)(2,g+2) \dots (g,
2g)$. We can transport the Bruhat order and the length function from
$W_{af}$ to $W_{af}\, c$ via the bijection $W_{af} \simeq W_{af}\, c$
of multiplication by $c$ on the right.  We observe that
$\mu \in W_{af} \,c$; indeed, $$\mu = w_{\mu} c,$$ where
$w_{\mu} = (s_gs_{g-1}\dots s_1)(s_g \dots s_2) \dots (s_g s_{g-1})
s_g \in W_{af}$. Observe also that the length of $\mu$ is
$\frac{g(g+1)}{2}$. The set $Adm_I(\mu)$ is precisely the subset of
$W_{af} \,c$ of elements which are $\leq$ a translation
$\mathrm{t}_{w \cdot\mu}$ for some $w \in W$ \cite[Theorem
4.5]{KR}.  Furthermore, for each $w \in Adm_I(\mu)$, the stratum
$U_{I,w}$ has dimension $\ell(w)$.

\subsection{Irreducible components} 

\subsubsection{The case that $I= \{0, \cdots, g\}$}

There are $2^g$
translations in the $W$ orbit of $\mu$ and they correspond to the open
stratum in each of the $2^g$ irreducible components of $\ov{M}_I$ when
$I = \{0, \cdots, g\}$.  These $2^g$ translations are parametrized by
$W/W_c$ where $W_c$ is the subgroup of $W$ of elements which stabilize
$\mu$. It identifies with the elements in $W \subset \mathcal{S}_{2g}$
which preserve the sets $\{1,\cdots, g\}$ and $\{g+1, \cdots, 2g\}$,
so this group is isomorphic to $\mathcal{S}_g$. The quotient $W/W_c$
has a set of representatives in $W$, denoted $W^c$, and called Kostant
representatives. An element $w \in W^c$ is characterized by the
property that it is the element of minimal length in the coset
$w W_c$.  The elements in $W^c$ are exactly the permuations $w \in W$
for which $w^{-1}(g) \geq \cdots \geq w^{-1}(1)$. Such representatives
are in bijection with functions
$s: \{1, \cdots, g \} \rightarrow \{1, \cdots, 2g\}$ which are
increasing and take exactly once one of the values $\{i, 2g+1-i\}$ for
all $1 \leq i \leq g $.  (Just set $s_w = w^{-1}(i)$). Moreover, the
length of an element $w$ with corresponding function $s$ is
$\frac{g(g+1)}{2} - \sum_{i \in \mathrm{Im}(s) \cap \{1, \cdots, g\}}
g-i+1$.

We can concretely determine an element of each of the orbits in
$\ov{M}_I$ corresponding to these $\mr{t}_{w \mu}$ for
$I= \{0, \cdots, g\}$ as follows.  The group $\widetilde{W}$ can be
viewed as a subgroup of $\mathrm{GSp}_{2g}( \F_p((t)))$, with $W$
begin represented by permutation matrices and the elements of $X_{\star}(T)$
as diagonal matrices by $\chi \in X_{\star}(T) \mapsto \chi(t)$. The element
$\mr{t}_{\mu}$ is thereby identified with
$\mathrm{diag} (t,\cdots, t, 1, \cdots 1)$.
   
We now consider the inclusion of chains:
$$ t \mathcal{V}_\bullet \subset \mr{t}_{w\mu} (\mathcal{V}_\bullet)
\subset \mathcal{V}_\bullet.$$ By reduction modulo $t$ and using the
identification
$\mathcal{V}_\bullet \otimes_{\F_p[[t]]} \F_p = V_\bullet
\otimes_{\ZZ} \FF_p$, we deduce that
$\mr{t}_{w\mu} (\mathcal{V}_\bullet)/ t \mathcal{V}_\bullet
\hookrightarrow V_\bullet$ defines an $\F_p$ point of $M_I$, which
represents the $w\mu$-orbit.

   \begin{rem} Let $F_\bullet \subset V_\bullet \otimes k$ be a $k$-point in the $w\mu$ orbit. Let $s: \{1, \cdots, g \} \rightarrow \{1,\dots, 2g\}$ be the associated function.  We actually get that the map   $$ F_{i-1} \rightarrow F_{i}$$ is an isomorphism  if and only $V_{i-1}/F_{i-1} \rightarrow V_{i} /F_{i}$ has kernel and cokernel of dimension $1$ and  if and only if $i \in \mathrm{Im}(s)$. 
  \end{rem}

\subsubsection{The case that $I = \{0\}$} The special fiber
$\ov{M}_{\{0\}}$ of $M_{\{0\}}$ is smooth and irreducible. Moreover,
there is a single orbit

\subsubsection{The case that $I = \{0,g\}$}
\begin{lem} The special fiber $\ov{M}_{\{0,g\}}$ of $M_{\{0,g\}}$ has
  $g+1$ irreducible components. There are $g+1$ open strata for the KR
  stratification, indexed by the integers $0 \leq s \leq g$. For each
  $0 \leq s \leq g$, a representative of the $s$-stratum is given by
  taking
   $$ F_0(s) = F_g(s) =   \langle e_{s+1}, \cdots, e_g,  e_{2g-s+1}, \cdots, e_{2g}\rangle$$
    This corresponds to the element $ \mathrm{diag} ( t \mathrm{Id}_s, \mathrm{Id}_{g-s},  t \mathrm{Id}_{g-s}, \mathrm{Id}_{s}) \in LG$. 
    \end{lem}

\begin{demo} It is easy to see that for all $w \in W$, there exists a $w' \in W_I$ and $0 \leq s \leq g$ such that 
$$w' \mr{t}_{w \mu} =\mathrm{diag} ( t \mathrm{Id}_s, \mathrm{Id}_{g-s},  t \mathrm{Id}_{g-s}, \mathrm{Id}_{s}).$$
This simply follows from the fact that $\mathcal{S}_g = W_I$.  It is easy to see that all the orbits corresponding to these elements are disjoint and of dimension $\frac{g(g+1)}{2}$. 
 \end{demo}

\begin{rem} We have the equality  $\mathrm{diag} ( t \mathrm{Id}_s, \mathrm{Id}_{g-s},  t \mathrm{Id}_{g-s}, \mathrm{Id}_{s})
 = w(s) (t \mathrm{Id}_g, \mathrm{Id}_g)$ where  $w(s)(i) = i$ if $i \leq s$, $w(s)(i) = 2g+1-i$ if $s+1 \leq i \leq g$. 
\end{rem}
\subsection{Local geometry of  the local model}
The local model $M_{\{0\}}$ is  smooth of relative dimension
$\frac{g(g+1)}{2}$. The  other local models $M_I$ for $I \neq \{0\} $
and $I \neq \{g\}$ are isomorphic to $M_{\{0\}}$ over $\Spec~\ZZ[1/p]$
but they have  singular special fiber at $p$. Nevertheless, we have
the following important result, the first part due to G\"ortz \cite[Theorem 2.1]{gortz}
and the second to He \cite[Theorem 1.2]{he-normality}:
\begin{thm} \label{thm-GortzHe} The local models $M_I$  are flat over
  $\ZZ$ and  $\ov{M}_I$ is reduced. Furthermore, $M_I$ is Cohen--Macaulay.
\end{thm}

\subsection{Hecke correspondence on the  affine Grassmannians}
We consider  the correspondence:
\begin{eqnarray*}
\xymatrix{ & \mathcal{F}_{\{0,g\}} \ar[ld]_{p_2} \ar[dr]^{p_1} & \\
\mathcal{F}_{\{g\}} && \mathcal{F}_{\{0\}}}
\end{eqnarray*}
We now pick the element $ w(s) \mu \in \widetilde{W}$ and restrict this map to a map: 
\begin{eqnarray*}
\xymatrix{ & U_{\{0,g\}, w(s) \mu} \ar[ld]_{p_2} \ar[dr]^{p_1} & \\
U_{\{g\}, w(s) \mu } && U_{\{0\}, w(s) \mu}}
\end{eqnarray*}

\begin{proposition}\label{prop-formula} The map on  differentials  $ \mathrm{d}p_1:  p_1^\star \Omega^1_{U_{\{0\}, w(s) \mu}} \rightarrow  \Omega^1_{U_{\{0,g\}, w(s) \mu}}$  has kernel and cokernel a locally free sheaf of rank $\frac{(g-s)(g-s+1)}{2}$.
\end{proposition}

\begin{demo} We have $$U_{\{0,g\}, w(s) \mu} \simeq \mathcal{P}_{\{0,g\}} / \big(\mathcal{P}_{\{0,g\}} \cap \mr{t}_{w(s) \mu} \mathcal{P}_{\{0,g\}} \mr{t}_{w(s) \mu}^{-1}\big)$$
$$U_{\{0\}, w(s) \mu} \simeq \mathcal{P}_{\{0\}} / \big(\mathcal{P}_{\{0\}} \cap \mr{t}_{w(s) \mu} \mathcal{P}_{\{0\}} \mr{t}_{w(s) \mu}^{-1}\big).$$
and $p_1$  is the obvious  $\mathcal{P}_{(0,g)}$- equivariant projection: $$\mathcal{P}_{\{0,g\}} / \big(\mathcal{P}_{\{0,g\}} \cap \mr{t}_{w(s) \mu} \mathcal{P}_{\{0,g\}} \mr{t}_{w(s) \mu}^{-1}\big) \rightarrow  \mathcal{P}_{\{0\}} / \big(\mathcal{P}_{\{0\}} \cap \mr{t}_{w(s) \mu} \mathcal{P}_{\{0\}} \mr{t}_{w(s) \mu}^{-1}\big)$$
Because the map is $\mathcal{P}_{(0,g)}$-equivariant,  it suffices to prove the claim in the tangent space at the identity. 

We first determine  the shape of $\mathcal{P}_{\{0\}}$ and $\mathcal{P}_{\{0,g\}}$.  The group $\mathcal{P}_{\{0\}}$ is the hyperspecial  subgroup of $LG$, whose $R$-points are $G(R[[t]])$.  
   The group  $\mathcal{P}_{\{0, g\}}$ is the Siegel parabolic  group: 
 $$ \mathcal{P}_{\{0,g\}} (R) = \bigl \{ M =  \bigl ( \begin{smallmatrix} % or pmatrix or bmatrix or Bmatrix or ...
       a & b  \\
        tc & d \\
    \end{smallmatrix} \bigr ) \in LG(R) \bigr \} $$
    where $a, b ,c ,d \in M_{g \times g}(R[[t]])$.

We now determine that  $\mathcal{P}_{\{0\}} \cap \mr{t}_{w(s) \mu} \mathcal{P}_{\{0\}} \mr{t}_{w(s) \mu}^{-1}$ consists of matrices with the following shape (the $\star$ have integral values, the rows and columns are of size $s$, $g-s$, $g-s$,  and $s$): 
   $$\begin{pmatrix} % or pmatrix or bmatrix or Bmatrix or ...
      \star  &  t\star &  \star & t\star \\
      \star & \star & \star & \star  \\
     \star &t \star & \star & t \star \\
    \star & \star & \star & \star \\
   \end{pmatrix} $$

We also determine that  $\mathcal{P}_{\{0,g\}} \cap \mr{t}_{w(s) \mu} \mathcal{P}_{\{0,g\}} \mr{t}_{w(s) \mu}^{-1}$ consists of matrices with the following shape:
$$\begin{pmatrix} % or pmatrix or bmatrix or Bmatrix or ...
      \star  &  t\star &  \star & t\star \\
      \star & \star & \star & \star  \\
    t \star &t^2 \star & \star & t \star \\
    t\star &t \star & \star & \star \\
   \end{pmatrix} $$

Passing to the Lie algebras, we easily see that the kernel of the map: 
$$\mathfrak{p}_{\{0,g\}} / \big(\mathfrak{p}_{\{0,g\}} \cap \mr{t}_{w(s) \mu} \mathfrak{p}_{\{0,g\}} \mr{t}_{w(s) \mu}^{-1}\big) \rightarrow  \mathfrak{p}_{\{0\}} / \big(\mathfrak{p}_{\{0\}} \cap \mr{t}_{w(s) \mu} \mathfrak{p}_{\{0\}} \mr{t}_{w(s) \mu}^{-1}\big)$$
is  the set of matrices of the form: 
 $$\begin{pmatrix} % or pmatrix or bmatrix or Bmatrix or ...
      0  &  0 &  0 &0\\
      0 & 0 & 0 & 0  \\
     0 & t A &0 & 0 \\
    0& 0 & 0 & 0 \\
   \end{pmatrix} $$
with $A\in M_{g-s\times g -s}(\F_p)$ satisfies $ K_{g-s}^t A K_{g-s} = A$  (the symplectic condition). 
\end{demo}

\subsection{The Hecke correspondence on the local model}
%We fix an integer $0 \leq r \leq g$. {\color{red} Is $r$ actually used?}
We consider  the correspondence:
\begin{eqnarray*}
\xymatrix{ & M_{\{0,g\}} \ar[ld]_{p_2} \ar[dr]^{p_1} & \\
M_{\{g\}} && M_{\{0\}}}
\end{eqnarray*}

\subsubsection{Sheaves on the local model}\label{section-Sheaves-loc-mod}
 Let $X$ be a scheme and $L$ be a locally free sheaf of rank $g$ over $X$. We let $T_L = \mathrm{Isom}_X ( \oscr_{X}^g,L)$ be the associated torsor. We let $\omega$ be the universal trivialization. The group $\mathrm{GL}_g$ acts on the right by $\omega \gamma = \omega \circ \gamma$. 
 Let $T$ be the standard diagonal torus in $\mathrm{GL}_g$ and let $B$ be the upper triangular Borel. 
 
Let $X^\star(T)$ be the character group of $T$.  We have
$X^\star(T)\simeq \ZZ^g$ via $(k_1, \cdots, k_g) \mapsto [
\mathrm{diag}(t_1, \cdots, t_g) \mapsto \prod t_i^{k_i}]$ and $P^+$,
the cone of dominant weights,
is given by $k_1 \geq k_2 \geq \cdots \geq k_g$.

For all $\kappa \in X^+(T)$ we denote by  $L_{\kappa} = \pi_{\star}
\oscr_{T_L}[\kappa^\vee]$   where $\pi_{\star} \oscr_{T_L}[\kappa^\vee]$ is the
subsheaf of $\pi_\star \oscr_{T_L}$ of sections $f(\omega)$ ($\omega$
a trivialization of $L$) which satisfy $f(\omega b) = \kappa^\vee(b)
f(\omega)$ and $\kappa^\vee = (-k_g, \cdots, -k_1)$ and $b \in B$. This is a locally free  sheaf over $X$.

We can in particular apply this construction to the sheaves
$L_0 = (V_0/F_0)^\vee$ on $M_{\{0\}}$ and $L_g = (V_g/F_g)^\vee$ on
$M_{\{g\}}$ to obtain sheaves $L_{0, \kappa}$ on $M_{\{0\}}$ and
$L_{g, \kappa}$ on $M_{\{g\}}$.
 
 \begin{rem} We have isomorphisms $L_g \simeq F_g$ and $L_0 \simeq F_0$ using the pairing, but these isomorphisms are not equivariant for the action of the center of the group $G$ (there is a ``Tate'' twist). 
  \end{rem}

 \subsubsection{The map $L_{g,\kappa} \rightarrow L_{0,\kappa}$}

 The natural map $V_0 \rightarrow V_g$ induces a map
 $p_1^{\star} V_0/F_0 \rightarrow p_2^{\star} V_g/F_g$ over
 $M_{\{0,g\}}$ and by duality a map
 $p_2^{\star} L_g \rightarrow p_1^{\star} L_0$ that we denote by
 $\alpha$. The map $\alpha$ is an isomorphism on the generic fibre of
 $M_{\{0,g\}}$, so it induces an isomorphism
 $\alpha^\star: p_2^\star L_{g, \kappa} \rightarrow p_1^\star L_{0,
   \kappa}$ on the generic fibre for each $\kappa$.  We now
 investigate the integral properties of this map.

 \begin{lem}\label{lem-weight-formula}  Let $\kappa = (k_1, \cdots, k_g)$. Let $0 \leq s \leq g$. Let $\xi$ be the generic point of $U_{\{0,g\}}(s)$.  The map $\alpha^\star$  induces  a map $\alpha^\star: (p_2^\star L_{g,\kappa})_{\xi} \rightarrow p^{ k_g + \cdots + k_{g-s+1}} (p_1^\star L_{0,\kappa})_{\xi}$ over the local ring $\oscr_{M_{0,g}, \xi}$. 
 \end{lem} 
 
 \begin{demo}  We first check that over $U_{\{0,g\}}(s)$, the map $\alpha$ has kernel and cokernel a locally free sheaf of rank $s$. Indeed, it is enough to check this at the point corresponding to $\mr{t}_{w(s)\mu}$, in which case  the corresponding diagram is:
 \begin{eqnarray*}
 \xymatrix{ \mathcal{V}_0 \ar[rr]^{\mathrm{diag}( t 1_g, 1_g)} && \mathcal{V}_g \\
 \mathcal{V}_0 \ar[rr]^{\mathrm{diag}( t 1_g, 1_g)}\ar[u]^{\mr{t}_{w(s) \mu}}& & \mathcal{V}_g \ar[u]^{\mr{t}_{w(s)\mu}}}
 \end{eqnarray*}
 and our claim is simply that the map $\mathcal{V}_0/\mr{t}_{w(s) \mu} \mathcal{V}_0 \rightarrow \mathcal{V}_g/\mr{t}_{w(s) \mu} \mathcal{V}_g$ has kernel of rank $s$. This is obvious.

 We can work  over the completion $R$ of  $\oscr_{M_{0,g}, \xi}$, which has uniformizing element $p$. We also denote by $v$ the $p$-adic valuation on $R$ normalized by $v(p)=1$. We fix isomorphisms $L_0 \simeq R^g$ and $L_g \simeq R^g$ such that $\alpha= \mathrm{diag}( p 1_s, 1_{g-s})$  in these bases. We have $$GL_g(R) = \coprod_{w \in \mathcal{S}_g} \mathrm{Iw} w \U(R)$$ the Iwahori decomposition with $\mathrm{Iw}$ the matrices which are upper triangular mod $p$ and $\U$ the unipotent radical of $\B$. Let $f \in L_{r,\kappa}$. Then  for $i \in \mathrm{Iw}$ and $w \in \mathcal{S}_g$, 
\begin{eqnarray*}
\alpha^\star f( iw) &=& f( \alpha^{-1} i w) \\
&=& f( \alpha^{-1} i \alpha w w^{-1} \alpha^{-1} w )\\
&=&  f (\alpha^{-1} i \alpha w) \kappa^\vee(w^{-1} \alpha^{-1} w)
\end{eqnarray*}
Since $\alpha^{-1} i \alpha w \in \mathrm{GL}_g(R)$, we deduce that
$v\big(f (\alpha^{-1} i \alpha w) \kappa^\vee(w^{-1} \alpha^{-1}
w)\big) \geq v(\kappa^\vee(w^{-1} \alpha^{-1} w)) \geq k_g + k_{g-1} +
\cdots + k_{(g-s+1)}$ for all $w$ since
$k_1 \geq k_2\geq \dots \geq k_g$.
 \end{demo}

\subsubsection{The cohomological correspondence}

We may now construct a cohomological correspondence. By Proposition \ref{prop-trace} and Theorem \ref{thm-GortzHe}, we have a fundamental class  $p_1^\star \oscr_{M_{\{0\}}} \rightarrow p_1^! \oscr_{M_{\{0\}}}$.  Moreover, the sheaf $p_1^! \oscr_{M_{\{0\}}}$ is a CM sheaf. 

There is also a map
$\alpha: p_2^\star L_{g, \kappa} \dashrightarrow p_1^\star
L_{0,\kappa}$ defined on the generic fibre, so that putting everything
together, we have a generically defined map (the naive cohomological
correspondence):
$$ T^{naive}: p_2^\star L_{g, \kappa} \dashrightarrow p_1^! L_{0, \kappa}$$

We may now normalize this correspondence. 

\begin{proposition} \label{prop:normal} Let $T = p^{- \inf_j \big \{\sum_{\ell =j+1}^{g}
    k_\ell + \frac{j(j+1)}{2}\big\}} T^{naive}$. Then $T$ is a true cohomological correspondence: 
$$ T: p_2^\star L_{g, \kappa} \rightarrow p_1^! L_{0, \kappa}$$
\end{proposition}
\begin{demo}
  Because $p_1^! L_{0, \kappa}$ is a CM sheaf, any generically
  defined map from a locally free sheaf into $p_1^! L_{0, \kappa}$ is
  defined globally if it is defined in codimension $1$. So it is
  enough to check that $T$ is defined at the generic points of all the
  components of $\ov{M}_{\{0,g\}}$.  Let $0 \leq s \leq g$ and let
  $\xi$ be the generic point of the stratum $U_{\{0,g\}}(s)$. At this
  point, we see that
  $T^{naive}: (p_2^\star L_{g, \kappa} )_\xi \rightarrow p^{
    \sum_{\ell = g-s+1}^{g} k_\ell + \frac{(g-s)(g-s+1)}{2}} (p_1^!
  L_{0, \kappa})_\xi$ by combining Lemma \ref{lem-weight-formula} and
  Proposition \ref{prop-formula}.
\end{demo}

\subsubsection{Proof of Theorem \ref{main-thm-symplectic}}\label{proof-main-symp} The main point of the theorem is to construct the action of $T_{\mathfrak{p}_i}$. The action of $S_{\mathfrak{p}_i}$ is by automorphisms (they are some generalized diamond operators) and the commutativity of the various operators is rather formal.

We fix some prime $\mathfrak{p}_i$. We let $K(\mathfrak{p}_i) =  K \cap t_i K t_i^{-1}$ for $t_i = \mathrm{diag} (\varpi_i^{-1}1_g,  1_g) \subset G( \mathbb{A}_f)$ where $\varpi_i$ is the finite ad\`ele which is $1$ at all places different from $\mathfrak{p}_i$ and $p$ at $\mathfrak{p}_i$. 

We claim that  there is a Hecke correspondence: 

\begin{eqnarray*}
\xymatrix{ & \mathfrak{Sh}^{tor}_{K(\mathfrak{p}_i), \Sigma''} \ar[rd]^{p_1} \ar[ld]_{p_2} & \\ 
\mathfrak{Sh}^{tor}_{K, \Sigma'} & & \mathfrak{Sh}^{tor}_{K, \Sigma}}
\end{eqnarray*}
which extends the usual Hecke correspondence  on the generic fiber. In order to construct this correspondence we use the PEL Shimura variety. We claim that there is a diagram :

\begin{eqnarray} \label{eq:corr}
\xymatrix{  & \widetilde{\mathfrak{Sh}}_{K(\mathfrak{p}_i)} \ar[rd]^{\widetilde{p}_1} \ar[d]\ar[ld]_{\widetilde{p}_2} & \\
\widetilde{\mathfrak{Sh}}_{K} \ar[d] & \mf{Sh}_{K(\mathfrak{p}_i)} \ar[rd]^{p_1} \ar[ld]_{p_2} &\widetilde{\mathfrak{Sh}}_{K} \ar[d]  \\ 
\mf{Sh}_{K} & & \mf{Sh}_{K}}
\end{eqnarray}
where the lower hat is simply the quotient of the top hat by the action of $\Delta$. The right square is cartesian by definition. We explain how to define
$\tilde{p}_2$ in order to make the left square cartesian.  We take a
totally positive element $x_i \in F^{\times, +}$ which has the
property that its $\mathfrak{p}_i$-adic valuation is exactly $1$, but
that its $\mathfrak{p}_j$-adic valuation is $0$ for all $j \neq i$.
Given a point $(A, \iota, \lambda, \eta, \eta_p)$ in
$\widetilde{\mathfrak{Sh}}_{K(\mathfrak{p}_i)}$ where $\eta_p$
corresponds to a maximal totally isotropic subgroup $H$ of
$A[\mathfrak{p}_i]$ we define
$\tilde{p}_2( (A, \iota, \lambda, \eta, \eta_p)) = (A', \iota',
\lambda', \eta')$ where $A' = A/H$, $\iota'$ and $\eta'$ have the
obvious definitions, and $\lambda'$ is defined by descending the
polarization $x_i \lambda$ to $A/H$. This indeed defines a prime to
$p$ polarization.  

The vertical maps in the above diagram are \'etale and surjective. We note that $\tilde{p}_2$ is not canonical (because of the ambiguity in the choice of $x_i$), but that $p_2$ is canonical.

Let
$\kappa = ((k_{1,\sigma}, \cdots, k_{g, \sigma}) ; k)_{\sigma \in I}$
be a dominant weight for $M_{\mu}/Z_c(G)$:
$k_{1, \sigma} \geq \cdots \geq k_{g,\sigma}$, and both $k$ and
$\sum_{i} k_{i, \sigma}$ have the same parity. After inverting $p$,
there is a map (denoted $T_{t_i}$ in section \ref{section-coho-vec-bund-Hecke}) 
$T_{\mathfrak{p}_i}^{naive} : p_2^\star \mathcal{V}_{\kappa,
  \Sigma} \rightarrow p_1^! \mathcal{V}_{\kappa, K, \Sigma}$ respecting
the cuspidal subsheaves and therefore inducing
$p_2^\star \mathcal{V}_{\kappa, K, \Sigma}(-D_{K,Ê\Sigma}) \rightarrow
p_1^! \mathcal{V}_{\kappa, K, \Sigma}(-D_{K, \Sigma})$.

We will prove that $p_1^! \mathcal{V}_{\kappa, K, \Sigma}$ and $p_1^! \mathcal{V}_{\kappa, K, \Sigma}(-D_{K, \Sigma})$ are CM sheaves over $\ocal_{E', \lambda'}$ and that   $$T_{\mathfrak{p}_i} = p^{\sum_{\sigma \in I_i} \sup_{1 \leq j \leq g} \Big\{ \frac{\sum_{\ell = 1}^j k_{\ell, \sigma} - \sum_{\ell = j+1}^g k_{\ell, \sigma} +k}{2} - \frac{j(j+1)}{2} \Big\}} T_{\mathfrak{p}_i}^{naive}$$ is a well defined map integrally. 

Actually, once we prove that  $p_1^! \mathcal{V}_{\kappa, K, \Sigma}$ is
a CM sheaf, it will be enough to check that the map
$T_{\mathfrak{p}_i}$ is  defined in codimension $1$. Since it is
well defined in characteristic $0$, and the boundary is flat over
$\ZZ_p$, it will be enough to check that the map $T_{\mathfrak{p}_i}$
is  defined on the interior $\mf{Sh}_{K(\mathfrak{p}_i)}$ of $\mathfrak{Sh}^{tor}_{K(\mathfrak{p}_i), \Sigma''}$ .

The main idea is to reduce everything to local model
computations. This is slightly delicate since our Shimura datum is
only of abelian type, but we can reduce to working with a PEL Shimura
datum.

We
can pull back the cohomological correspondence
$T_{\mathfrak{p}_i}^{naive}$ over $\qq$ to a cohomological
correspondence
$\widetilde{T_{\mathfrak{p}_i}^{naive}}: \tilde{p}_2^\star
\mathcal{V}_{\kappa} \rightarrow \tilde{p}_1^!
\mathcal{V}_{\kappa}$. It is enough to prove everything for the latter
correspondence.

Now we have a local model diagram  of correspondences:
\begin{eqnarray} \label{eq:lmd}
\xymatrix{ && \widetilde{\mathfrak{P}}_{K(\mathfrak{p}_i)} \ar@<-.5ex>[d]_{q_2} \ar@<.5ex>[d]^{q_1} \ar[lld]_{f} \ar[rrd]^{g}&& \\ 
 \widetilde{\mathfrak{Sh}}_{K(\mathfrak{p}_i)} \ar@<-.5ex>[d]_{\tilde{p}_2} \ar@<.5ex>[d]^{\tilde{p}_1}& & \widetilde{\mathfrak{P}}_{K} \ar[lld]_{h} \ar[rrd]^{e}& &   M^{loc}_{K(\mathfrak{p}_i)} \ar@<-.5ex>[d]_{t_2} \ar@<.5ex>[d]^{t_1}  \\
 \widetilde{\mathfrak{Sh}}_{K} & & & &   M^{loc}_{K}}
 \end{eqnarray}
By definition,  $M^{loc}_K = \prod_{\sigma \in I} M_{\{0\}}$ and
$M^{loc}_{K(\mathfrak{p}_i)} = \prod_{\sigma \in I_i} M_{\{0,g\}}
\prod_{\sigma \notin I_i} M_{\{0\}}$.  The projection $t_1$ is  the
product of the  projections $p_1: M_{\{0,g\}} \rightarrow M_{\{0\}}$
at places $\sigma \in I_i$ and the identity otherwise. The projection
$t_2$ is   the product of the  projections $p_2: M_{\{0,g\}} \rightarrow M_{\{g\}} \simeq M_{\{0\}}$ at $\sigma \in I_i$ and the identity if $\sigma \notin I_i$.  
The map $h$ is the torsor of symplectic trivialisations of $\mathcal{H}_{1,dR}(A/\widetilde{Sh}_{K})$ for  $A$ the universal abelian scheme. The map $f$ is the torsor of symplectic  trivialisations  of the chain $\mathcal{H}_{1,dR}(A/\widetilde{Sh}_{K}) \rightarrow \mathcal{H}_{1,dR}((A/H)/\widetilde{Sh}_{K}) \rightarrow \mathcal{H}_{1,dR}(A/\widetilde{Sh}_{K})$  (i.e., isomorphisms with the chain $\prod_{\sigma \in I_i} ( V_0 \rightarrow V_g \rightarrow V_0) \times \prod_{\sigma \notin I_i} (V_0 \stackrel{\mathrm{Id}}\rightarrow V_0 \stackrel{p \mathrm{Id}}\rightarrow V_0)$).
The maps $g$ and $e$ are given by the Hodge filtration. 

This diagram is commutative, the diagonal maps are smooth and the
diagonal maps going to the left are surjective, but the squares are
not cartesian!  We have the following theorem:
\begin{thm}\label{thm-main-local-model} Let $\bar{x}: \Spec(k) \rightarrow \tilde{\mathfrak{P}}_{K(\mathfrak{p_i})}$. Let $\bar{y} = f(\bar{x})$, $\bar{z} = \tilde{p}_1(\bar{y})$, $\bar{y}' = g(\bar{x})$, $\bar{z}' = {t}_1(\bar{y}')$. Then there are isomorphisms between the strict henselizations: 
$$ \oscr_{\widetilde{\mathfrak{Sh}}_{K}, \bar{z}} \simeq \oscr_{{M}^{loc}_{K}, \bar{z}'} $$ and 
$$ \oscr_{\widetilde{\mathfrak{Sh}}_{K(\mathfrak{p}_i)}, \bar{y}} \simeq \oscr_{{M}^{loc}_{K(\mathfrak{p}_i)}, \bar{y}'}.$$

Moreover, there is a commutative diagram   between the maps on  Zariski cotangent spaces at $\bar{y}, \bar{y}'$ and $\bar{z}, \bar{z'}$: 

\begin{eqnarray*}
\xymatrix{ \mathfrak{m}_{\oscr_{\widetilde{\mathfrak{Sh}}_{K}, \bar{z}}}/ \mathfrak{m}^2_{\oscr_{\widetilde{\mathfrak{Sh}}_{K}, \bar{z}}}  \ar[d]^{\mathrm{d}\tilde{p}_1} \ar[r]^{\sim}& \mathfrak{m}_{\oscr_{M^{loc}_K, \bar{z}'}}/ \mathfrak{m}^2_{\oscr_{M^{loc}_K, \bar{z}'}} \ar[d]^{\mathrm{d}t_1}   \\
 \mathfrak{m}_{\oscr_{\widetilde{\mathfrak{Sh}}_{K(\mathfrak{p}_i)}, \bar{y}}}/ \mathfrak{m}^2_{\oscr_{\widetilde{\mathfrak{Sh}}_{K(\mathfrak{p}_i)}, \bar{y}}} \ar[r]^{\sim}& \mathfrak{m}_{\oscr_{M^{loc}_{K(\mathfrak{p}_i)}, \bar{y}'}}/ \mathfrak{m}^2_{\oscr_{M^{loc}_{K(\mathfrak{p}_i)}, \bar{y}'}}}
\end{eqnarray*}

\end{thm}

\begin{proof}   The first point is the main result of local model theory. The second point is an immediate consequence of Grothendieck--Messing deformation theory  (see \cite[Thm. 2.1]{dejong-ppav} for a precise statement of  this theory). 
\end{proof}

\begin{coro} \label{cor:cm} The map $\widetilde{\mathfrak{Sh}}^{tor}_{K, \Sigma}
  \rightarrow \Spec~\ocal_{E', \lambda'}$ is smooth and the map
  $\widetilde{\mathfrak{Sh}}^{tor}_{K(\mathfrak{p}_i), \Sigma}
  \rightarrow \Spec~\ocal_{E', \lambda'}$ is a CM map.
\end{coro}
\begin{demo} Over the interior of the moduli space, this follows from
  the previous theorem. The description of the integral toroidal
  compactification in \cite{MR3948111} shows that the property holds
  everywhere.
\end{demo}

 \bigskip
 
It follows that the cohomological correspondence can be extended to a rational map from a locally free sheaf to a CM sheaf $\widetilde{T_{\mathfrak{p}_i}^{naive}}: p_2^\star \mathcal{V}_{\kappa, K, \Sigma} \dashrightarrow p_1^! \mathcal{V}_{\kappa, K, \Sigma}.$
 In particular, this corollary implies that it is enough to work with the interior of
the Shimura variety.

 For any $\mu = (\mu_1, \cdots, \mu_g)$, with
 $\mu_1 \geq \cdots \geq \mu_g$, we have already defined two sheaves
 $p_1^\star L_{0,\mu}$ and $p_2^\star L_{g, \mu}$ over
 $M_{\{0,g\}}$, and a rational map
 $\alpha^\star: p_2^\star L_{g, \mu} \dashrightarrow p_1^\star
 L_{g, \mu}$.
 
 Let us write $\kappa_\sigma = (k_{1, \sigma},\cdots, k_{g, \sigma})$ for all $\sigma$.
 We define sheaves
 $L_{0, \kappa} = \boxtimes_\sigma L_{0, \kappa_{\sigma}}$ and
 $L_{g, \kappa} = \boxtimes_{\sigma \in I_i} L_{g, \kappa_{\sigma}}
 \boxtimes_{\sigma \notin I_i} L_{0, \kappa_{\sigma}} $ on $M_K^{loc}$
 and a map
 $$\beta^\star = \boxtimes_{\sigma \in I_i} \alpha^\star
 \boxtimes_{\sigma \notin I_i} \mathrm{Id}: t_2^\star L_{g, \kappa}
 \rightarrow t_1^\star L_{0, \kappa}.$$
 
 \begin{lem}\label{lem-comparision-model-shim} Over $\widetilde{\mathfrak{P}}_{K(\mathfrak{p}_i)}$,  we have a commutative diagram: 
\begin{eqnarray*}
 \xymatrix{ g^\star t_2^\star L_{g, \kappa} \ar[rrr]^{ p^{-\sum_{\sigma \in I_i} \frac{k + \sum_jk_{j,\sigma}}{2}}g^\star \beta^\star} \ar[d]^{\sim} &&& g^\star t_1^\star L_{0, \kappa} \ar[d]^{\sim} \\
 f^\star p_2^\star \mathcal{V}_{\kappa, K} \ar[rrr] &&& f^\star p_1^\star \mathcal{V}_{\kappa, K}}
 \end{eqnarray*}
 
 \end{lem}

 \begin{demo} The sheaf $(V_0/L_0)^\vee$ on the $\sigma$-component of
   the local model $M_K^{loc}$ corresponds to the sheaf
   $\omega_{A, \sigma}$ by definition. Therefore, the sheaves
   $L_{0,\kappa}$ and $L_{g,\kappa}$ correspond to the representations
   of $M_{\mu}$ of highest weight
   $(k_{1,\sigma}, \cdots, k_{g,\sigma}; - \sum_{i}
   k_{i,\sigma})_{\sigma \in I}$.  There are isomorphisms of sheaves
   over $\widetilde{\mf{P}}_{K(\mathfrak{p}_i)}$:
   $g^\star t_2^\star L_{g, \kappa} \simeq f^\star p_2^\star
   \mathcal{V}_{\kappa, K}$ and
   $g^\star t_1^\star L_{0, \kappa} \simeq f^\star p_1^\star
   \mathcal{V}_{\kappa, K}$ but these are not $G$-equivariant
   isomorphisms. We can make them $G$-equivariant as follows. Over
   $M_{\{0\}} = M_{\{g\}} = G/P$ we have a $G$-equivariant sheaf
   $\mathcal{L}$ corresponding to the similtude character of $G$
   (viewed as a $P$-representation). This sheaf has a trivialization
   (given by the similitude character of $G$), but its $G$-equivariant
   structure is not trivial.  There is a canonical map
   $ p_2^\star \mathcal{L} \rightarrow p_1^\star \mathcal{L}$ over
   $M_{\{0,g\}}$ which is multiplication by $p^{-1}$ in the
   trivializations.  We can twist
   $L_{0, \kappa} = \boxtimes_\sigma L_{0, \kappa_{\sigma}}$ to
   $L'_{0, \kappa} = \boxtimes_\sigma L_{0, \kappa_{\sigma}} \otimes
   \mathcal{L}^{\frac{ k + \sum_{\ell} k_{\ell, \sigma}}{2}}$ and
   $L_{g, \kappa}$ to
   $$L'_{g, \kappa} = \boxtimes_{\sigma \in I_i} L_{g,
     \kappa_{\sigma}} \otimes \mathcal{L}^{\frac{ k + \sum_{\ell}
       k_{\ell, \sigma}}{2}}\boxtimes_{\sigma \notin I_i} L_{0,
     \kappa_{\sigma}}\otimes \mathcal{L}^{\frac{ k + \sum_{\ell}
       k_{\ell, \sigma}}{2}}.$$ Therefore, we have a commutative
   diagram over $M^{loc}_{K(\mathfrak{p}_i)}$:
\begin{eqnarray*}
 \xymatrix{  t_2^\star L_{g, \kappa} \ar[rrr]^{ p^{-\sum_{\sigma \in I_i} \frac{k + \sum_jk_{j,\sigma}}{2}} \beta^\star} \ar[d]^{\sim} &&&  t_1^\star L_{0, \kappa} \ar[d]^{\sim} \\
  t_2^\star L'_{g, \kappa} \ar[rrr]^{(\beta')^\star}   &&&  t_1^\star L_{0, \kappa} }
 \end{eqnarray*}
 for $(\beta')^\star$ the natural map coming from the $G$-equivariant
 structure. After twisting we have a commutative diagram:
\begin{eqnarray*}
 \xymatrix{ g^\star t_2^\star L'_{g, \kappa} \ar[rrr]^{ g^\star (\beta')^\star} \ar[d]^{\sim} &&& g^\star t_1^\star L'_{0, \kappa} \ar[d]^{\sim} \\
 f^\star p_2^\star \mathcal{V}_{\kappa, K} \ar[rrr] &&& f^\star p_1^\star \mathcal{V}_{\kappa, K}}
 \end{eqnarray*} 
\end{demo} 

\bigskip

We can now conclude the proof of Theorem \ref{main-thm-symplectic}.  Let $\xi$ be a generic point of the special fiber of $\widetilde{\mathfrak{Sh}}_{K(\mathfrak{p}_i)}$. It  corresponds on the local model $M^{loc}_{K(\mathfrak{p}_i)}$ to a point in the stratum $\prod_{\sigma \in I_i} U_{\{0,g\}, w(s_{\sigma})\mu} \times \prod_{\sigma \notin I_i} U_{\{0\}, \mu}$. 

Using the definition of $\beta$ in terms of $\alpha$, Lemma
\ref{lem-weight-formula}, Theorem \ref{thm-main-local-model} and Lemma
\ref{lem-comparision-model-shim}, we deduce that on the local ring at
$\xi$, we have a map
$$\widetilde{T_{\mathfrak{p}_i}^{naive}}: (\tilde{p}_2^\star
\mathcal{V}_{\kappa})_{\xi} \rightarrow p^{ \sum_{\sigma \in I_i}
  \frac{(g-s_\sigma)(g-s_\sigma +1)}{2} + \frac{-\sum_{\ell =
      1}^{g-s_\sigma} k_{\ell, \sigma} + \sum_{\ell = g-s_\sigma+1}^g
    k_{\ell, \sigma} +k}{2}} ( \tilde{p}_1^!
\mathcal{V}_{\kappa})_\xi.$$ We conclude using the CM property in
Corollary \ref{cor:cm} that we have a well defined cohomological
correspondence
\[
  \widetilde{T_{\mathfrak{p}_i}} := 
 p^{\sum_{\sigma \in I_i} \sup_{1 \leq j \leq g} \Big\{ \frac{\sum_{\ell = 1}^j k_{\ell, \sigma} - \sum_{\ell = j+1}^g k_{\ell, \sigma} +k}{2} - \frac{j(j+1)}{2} \Big\}}
  \widetilde{T_{\mathfrak{p}_i}^{naive}}: \tilde{p}_2^\star
  \mathcal{V}_{\kappa} \rightarrow \tilde{p}_1^!
  \mathcal{V}_{\kappa}
\]
which, using the diagram \eqref{eq:corr}, shows that we have a
cohomological correspondence
\[
  T_{\mathfrak{p}_i} :=
   p^{\sum_{\sigma \in I_i} \sup_{1 \leq j \leq g} \Big\{ \frac{\sum_{\ell = 1}^j k_{\ell, \sigma} - \sum_{\ell = j+1}^g k_{\ell, \sigma} +k}{2} - \frac{j(j+1)}{2} \Big\}}
  T_{\mathfrak{p}_i}^{naive}: {p}_2^\star
  \mathcal{V}_{\kappa} \rightarrow {p}_1^!
  \mathcal{V}_{\kappa}\, .
\]

\section{Unitary Shimura varieties}\label{section-unitary} This section is dedicated to unitary Shimura varieties. 

\subsection{The Shimura datum} Let $F$ be a totally real field, and let $L$ be a totally imaginary quadratic extension of $F$. We denote by $c \in \mathrm{Gal}(L/F)$ the complex conjugation. We let $I = \mathrm{Hom}(F, \overline{\qq})$ and for all $\sigma \in I$ we chose an extension $\tau: L \rightarrow \overline{\qq}$ of $\sigma$ to $L$. Therefore, $\mathrm{Hom}(L, \overline{\qq}) = I \coprod I \circ c$. 

Let $V$ be a $K$ vector space of dimension $n$, together with a
hermitian form $\langle\, , \rangle$. We assume that this form is not
definite at at least one real place of $F$.  Let $G$ be the reductive
group over $\qq$ of similitudes of $(V, \langle\, , \rangle)$. Namely:
$$ G = \{ (g,c) \in \mathrm{Res}_{L/\qq} \mathrm{GL}(V) \times \mathbb{G}_m,~\langle g v, g w \rangle = c \langle v,  w \rangle~\forall v,w \in V\}$$

We have natural isomorphisms $F \otimes_\qq \mathbb{R} = \mathbb{R}^I$
and $L \otimes_\qq \mathbb{R} = \mathbb{C}^I$ given by
$x \otimes y \mapsto ( \tau(x) y)_{\tau \in I}$.  We let
$V_{\mathbb{R}} = V \otimes_{\qq} \mathbb{R} = \oplus V_{\R, \tau} $
where $V_{\mathbb{R}, \tau} = V \otimes_{F, \tau} \mathbb{R}$. We
chose an isomorphism $V_{\R, \tau} \simeq \mathbb{C}^n$ where $L$ acts
on $\mathbb{C}^n$ via $\tau$ and the hermitian pairing induced on
$\mathbb{C}^n$, denoted by $\langle\, , \rangle_\tau$, is of signature
$(p_\tau, q_\tau)$ and is in the standard form
$\sum_{i=1}^{p_\tau} z_i \bar{z_i} - \sum_{i=p_\tau+1}^{n} z_i
\bar{z_i}$.

We deduce that
$G_{\mathbb{R}} = \mathrm{G}(\prod_{\tau \in I} U(p_\tau, q_\tau))$.
The natural action of $G_{\mathbb{R}}$ on
$V_{\mathbb{R}} \simeq \oplus_{\tau \in I} \C^n$ gives an embedding
$G_{\mathbb{R}} \subset \prod_{\tau \in I}
\mathrm{Res}_{\C/\R}\mathrm{GL}_n $.

Our Shimura datum is  $(G, X)$ where $X$ is the $G(\mathbb{R})$-orbit of the homomorphism $h_0: \mathrm{Res}_{\C/\R} \mathbb{G}_m \rightarrow G_{\mathbb{R}}$ given by $h_0(z) = \prod_{\tau} z_{p_\tau, q_\tau}$ where $z_{p_\tau, q_\tau} \in \mathrm{GL}_n(\mathbb{C})$ is the diagonal matrix $\mathrm{diag}( z \mathrm{1}_{p_\tau}, \bar{z} \mathrm{1}_{q_\tau})$. 

The centralizer of $h_0$ is $ K_\infty \times \mathbb{R}^{\times,+} $ where $K_\infty =  \prod_{\tau} U(p_\tau)(\mathbb{R}) \times U(q_\tau)(\mathbb{R})$ is a maximal compact subgroup and $\mathbb{R}^{\times,+}$ is the connected component of the identity in the center of $G(\mathbb{R})$.

\subsection{The flag variety}\label{sect-flag-unitary}
The embedding $G_{\mathbb{R}} \rightarrow \prod_{\tau \in I}
\mathrm{Res}_{\C/\R}\mathrm{GL}_n $ induces, after extending scalars
to $\mathbb{C}$ and projecting $\mathrm{Res}_{\C/\R}\mathrm{GL}_n
\times_{\Spec~\mathbb{R}} \Spec~\C = \mathrm{GL}_n \times
\mathrm{GL}_n $ onto the first factor, a morphism $G_{\mathbb{C}}
\rightarrow  (\prod_{\tau \in I} \mathrm{GL}_n)$ of algebraic groups
over $\Spec~\C$. We thus get an isomorphism    $G_{\mathbb{C}} \rightarrow  (\prod_{\tau \in I} \mathrm{GL}_n) \times \mathbb{G}_m$ whose second component is the similitude factor.   The cocharacter $\mu_0$ attached to $h_0$  is given by 
 $ \mu_0 (z) = \prod_{\tau} \mathrm{diag}( z 1_{p_\tau}, 1_{q_\tau}) \times z$. 
 
 We deduce that a representative of $P_{\mu}$ is given by the group $(\prod_\tau P_{p_{\tau}, q_{\tau}}) \times \mathbb{G}_m \subset G_{\mathbb{C}}$ with $P_{p_\tau, q_{\tau}}$ the standard  parabolic subgroup of $GL_{n}$ of lower triangular matrices with  blocks of size $p_\tau$ and $q_\tau$,  with Levi $\mathrm{GL}_{p_\tau} \times \mathrm{GL}_{q_\tau}$.

The Borel embedding is the map 

$$ X \rightarrow \mathrm{FL}_{G,X}$$
 sending $h$  to the Hodge filtration $\mathrm{Fil}_h = \oplus \mathrm{Fil}_{h, \tau}$ (i.e., the subspace stabilized by $P_\mu$) on $$V_\mathbb{R} \otimes_{\mathbb{R}} \mathbb{C} = \oplus_{\tau}   V_{\mathbb{R}, \tau} \otimes_{\R} \C.$$
 We have  $V_{\mathbb{R}, \tau} \otimes_{\R} \C \simeq \mathbb{C}^n  \otimes_{\mathbb{R}}\mathbb{C} \simeq  \oplus_{\tau}( \mathbb{C}^n \oplus \mathbb{C}^n)$ where the last map is given for each $\tau$ by 
\begin{eqnarray*}
\mathbb{C}^n  \otimes_{\mathbb{R}}\mathbb{C} &\rightarrow & \mathbb{C}^n \oplus \mathbb{C}^n \\
v \otimes x &\mapsto& (vx, \overline{v} x)
\end{eqnarray*}
We denote by $V_{\C, \tau,+} $ and $V_{\C, \tau, -}$ the two factors
in this isomorphism. The pairing $\langle \ , \rangle_\tau$ induces a
perfect pairing between $V_{\C, \tau,+} $ and $V_{\C, \tau, -}$.  We
can therefore think of $\mathrm{FL}_{G,X}$ as a product of
Grassmannians parametrizing for each $\tau$ a direct summand
$\mathrm{Fil}_{\tau} = \mathrm{Fil}_{\tau, +} \oplus
\mathrm{Fil}_{\tau,-} \subset V_{\C, \tau,+} \oplus V_{\C, \tau, -}$,
where $\mathrm{Fil}_{\tau, +}$ has rank $q_\tau$,
$\mathrm{Fil}_{\tau, -}$ has rank $p_\tau$, and they are orthogonal
with each other for the pairing $\langle \ , \rangle_\tau$ (therefore
$\mathrm{Fil}_{\tau}$ is determined by $\mathrm{Fil}_{\tau,+}$).

The filtration  at the point $h_0$, is given by  $<e_{p_\tau+1}, \cdots, e_n> \oplus <e_1, \cdots, e_{p_\tau}> \subset \mathbb{C}^n \oplus \mathbb{C}^n$ for each $\tau$ (in the canonical basis $e_1, \cdots, e_n$ of $\C^n$).

We now choose the  diagonal maximal torus $S \subset G_{\mathbb{R}}$. Then $S =  (\prod_{\tau} \mathrm{U}(1)^n) \times \mathbb{G}_m/\mu_2$ and its character group is the subgroup of  $(\ZZ^n)^I \times \ZZ$ of elements $((a_{1,\tau}, \cdots, a_{n,\tau})_{\tau}; k)$ with the condition that $\sum_{i,\tau} a_{i, \tau} = k \mod 2$.  We have $\mathrm{U}(1) \times_{\Spec~\R} \Spec~\C = \{ (z_1,z_2) \in \mathbb{G}_m \times \mathbb{G}_m,~z_1 z_2 = 1\}$, and the projection on the first coordinate induces an isomorphism $\mathrm{U}(1) \times_\R \C \simeq \mathbb{G}_m$.  We have $$S_{\C} \simeq  (\prod_{\tau} \mathbb{G}_m^n) \times \mathbb{G}_m/\mu_2  \hookrightarrow  G_{\mathbb{\C}} \simeq (\prod_{\tau \in I} \mathrm{GL}_n) \times \mathbb{G}_m$$
and this map is given explicitly by $((x_{1,\tau}, \cdots, x_{n,\tau})_{\tau};t) \mapsto \prod_\tau \mathrm{diag}(tx_{1,\tau},\cdots,tx_{n,\tau})_\tau \times t^2$. 

We choose the upper triangular Borel in $G_{\C}$ (this choice is compatible with our conventions in Section \ref{sec-grouptheoretic}) and the corresponding dominant cone in $X^\star(S_{\C})$  is given by the condition  $a_{1,\tau} \geq  \cdots \geq a_{n,\tau}$  for all $\tau$.  
There is also an associated dominant cone for the Levi $\prod_{\tau} (\mathrm{GL}_{p_\tau} \times \mathrm{GL}_{q_\tau}) \times \mathbb{G}_m$ which is given by the conditions $a_{1,\tau} \geq  \cdots \geq a_{p_\tau,\tau}$ and $a_{p_\tau +1,\tau} \geq  \cdots \geq a_{n,\tau}$ for all $\tau$. 

The (trivial) vector bundle $V_{\mathbb{R}} \otimes_{\mathbb{R}} \mathbb{C}$ over $\mathrm{FL}_{G,X}$ is associated with a representation of $G_{\mathbb{C}}$.  This representation is the direct sum of the standard $n$-dimensional representation  and its complex conjugate.  For any $\tau_0 \in I$, the direct factors  $V_{\C, \tau_0,+}$ and $V_{\C, \tau_0, -}$  correspond respectively to the representations   of $G_{\mathbb{C}}$  with highest weight
$ ((1,\cdots, 0)_{\tau_0}, (0,\cdots, 0)_{\tau \neq \tau_0}; 1)$ and $ ((0,\cdots, -1)_{\tau_0}, (0,\cdots, 0)_{\tau \neq \tau_0}; 1)$.

The vector bundles $V_{\C, \tau_0,+}/ \mathrm{Fil}_{\tau_0, +}$ and $V_{\C, \tau_0,-}/ \mathrm{Fil}_{\tau_0, -}$   correspond respectively to the two irreducible representations of the Levi $\prod_{\tau} (\mathrm{GL}_{p_\tau} \times \mathrm{GL}_{q_\tau}) \times \mathbb{G}_m$ with highest weights:
$((1,\cdots, 0)_{\tau_0}, (0,\cdots, 0)_{\tau \neq \tau_0}; 1)$ and  $((0,\cdots,  -1)_{\tau_0}, (0,\cdots, 0)_{\tau \neq \tau_0}; 1)$.

\subsection{Integral model} 

The reflex field $E$ is the  fixed field in $\overline{\qq}$ of the
subgroup of $\mathrm{Gal}(\overline{\qq}/\qq)$ consisting of all
elements acting trivially on the set $\{(p_\tau, q_\tau)\}_{\tau}$. 

We now let $p$ be a prime unramified in $L$.  We choose an
$\ocal_K$-lattice $V_{\mathbb{Z}} \subset V$  and we assume that $V_{\mathbb{Z}} \otimes
\mathbb{Z}_{(p)}$ is self dual for the pairing $\langle\, , \rangle$.   The choice of this lattice gives an integral model for $G$, and this model is reductive over $\ZZ_{(p)}$. Let $K= K^pK_p \subset G(\mathbb{A}_f)$ be a compact open subgroup with $K_p$ hyperspecial. Let $\lambda$ be a finite place of $E$ over $p$.

The Shimura variety $\mathfrak{Sh}_K$ represents the functor on the category of noetherian  $\ocal_{E,\lambda}$-algebras that sends a noetherian  $\ocal_{E,\lambda}$-algebra to the set of  equivalence classes of $(A, \iota, \lambda,
\eta)$, where:
\begin{enumerate}
\item $A \rightarrow \Spec~R$ is an abelian scheme,
\item $\lambda: A \rightarrow A^t$ is a prime to $p$  quasi-polarization,
\item $\iota: \ocal_L \rightarrow \mathrm{End} (A) \otimes \ZZ_{(p)}$ is a homomorphism of algebras with involution,
\item $\mathrm{Lie} (A)$  satisfies the determinant condition (see Remark \ref{rem-det-condition}),
  \item $\eta$ is a $K^p$-level structure. 
\end{enumerate}

\begin{rem}\label{rem-det-condition} The determinant condition is the following. We let $E'$ be
  an extension of $E$ which contains the Galois closure of $L$. We let $\lambda'$ be a place of $E'$ above $\lambda$.  It is
  enough to spell out the condition when $R$ is an
  $\ocal_{E',\lambda'}$-algebra. In that situation, we have
  $$\mathrm{Lie} (A) = \bigoplus_{\tau} \mathrm{Lie} (A)_{\tau,+}
  \oplus \mathrm{Lie} (A)_{\tau, -}$$ where
  $\mathrm{Lie} (A)_{\tau,+} = \mathrm{Lie} (A) \otimes_{\ocal_L
    \otimes R, \tau} R$ and
  $\mathrm{Lie} (A)_{\tau,-} = \mathrm{Lie} (A) \otimes_{\ocal_L
    \otimes R, \tau \circ c} R$.  The condition is that
  $\mathrm{Lie} (A)_{\tau, +}$ is a locally free $R$-module of rank
  $p_\tau$ and $\mathrm{Lie} (A)_{\tau, -}$ is locally free of rank
  $q_\tau$.
\end{rem}

\subsection{The vector bundle dictionary}\label{vector-bundle-dico} This section follows  \cite[Section 5.5]{MR3177267}. We now work over $\ocal_{E',\lambda}$ as in the preceeding remark. We have $\omega_A = \bigoplus_{\tau} \omega_{A, \tau,+} \oplus \omega_{A, \tau,-}$. 
It follows  from the discussion towards the end of Section \ref{sect-flag-unitary}  that $\omega_{A, \tau_0,+}$ corresponds to the irreducible  representation of $(\prod_\tau P_{p_{\tau}, q_{\tau}}) \times \mathbb{G}_m$ of highest weight $\kappa = ((a_{1, \tau},\cdots, a_{n, \tau})_\tau; k)$ with $(a_{1, \tau_0},\cdots, a_{n, \tau_0}) = ( 0, \cdots,0,  -1, 0, \cdots, 0)$ (with $-1$ in position $p_\tau$), $k=-1$, $(a_{1, \tau},\cdots, a_{n, \tau}) = (0, \cdots, 0)$ if $\tau \neq \tau_0$. 
 We deduce similarly that $\omega_{A, \tau_0,-}$ corresponds to the weight $\kappa = ((a_{1, \tau},\cdots, a_{n, \tau})_\tau; k)$ with $(a_{1, \tau_0},\cdots, a_{n, \tau_0}) = ( 0, \cdots,0,  1, 0, \cdots, 0)$ (with $1$ in position $p_\tau+1$), $k=-1$, $(a_{1, \tau},\cdots, a_{n, \tau}) = (0, \cdots, 0)$ if $\tau \neq \tau_0$.

 We now choose integers  $a_{1,\tau} \geq \cdots \geq a_{p_{\tau}, \tau}$ and $b_{1,\tau} \geq \cdots \geq b_{q_{\tau}, \tau}$ for all $\tau \in I$ and consider the representation of the Levi $M_\mu$ of the parabolic $P_\mu$ with  highest weight $\kappa =( ( -a_{p_\tau, \tau}, \cdots, -a_{1,\tau}, b_{1,\tau}, \cdots, b_{q_{\tau}, \tau}); - \sum_{\ell, \tau} a_{\ell, \tau} - \sum_{\ell, \tau} b_{\ell, \tau}) $. The sheaf $\mathcal{V}_{\kappa,K}$ associated with this weight has the following description. We denote by $\pi_{\tau,+} : \mathcal{T}_{\tau,+} \rightarrow \mathfrak{Sh}_{K}$ and $\pi_{\tau,-}  : \mathcal{T}_{\tau,-} \rightarrow  \mathfrak{Sh}_{K}$ the $\mathrm{GL}_{p_\tau}$ and $\mathrm{GL}_{q_\tau}$ torsors of trivialization of $\omega_{A, \tau,+}$ and $\omega_{A, \tau,-}$
 We let $\omega_{A, \tau, +}^{(a_{\tau})} = (\pi_{\tau,+})_\star (\oscr_{\mathcal{T}_{\tau,+}})[a_{1,\tau}, \cdots, a_{p_\tau,\tau}]$ and $\omega_{A, \tau, -}^{(b_{\tau})} = (\pi_{\tau,+})_\star (\oscr_{\mathcal{T}_{\tau,+}})[-b_{q_\tau,\tau}, \cdots, -b_{1,\tau}]$ (where the bracket means the isotypic part for the representation of the upper triangular Borel given by the characters $(a_{1,\tau}, \cdots, a_{p_\tau,\tau})$ and $(-b_{q_\tau,\tau}, \cdots, -b_{1,\tau})$ respectively). 
 
Then we find that $\mathcal{V}_{\kappa,K}$ is   simply $\otimes_{\tau} \omega_{A, \tau, +}^{(a_{\tau})}  \otimes \omega_{A, \tau,-}^{(b_{\tau})}$ that is usually considered in the theory of modular forms on unitary groups (we mean that our normalization of the central character is the usual one).

 \subsection{Dual group}

 Recall that we have $G_{\mathbb{C}} \simeq ( \prod_\tau\mathrm{GL}_n) \times \mathbb{G}_m$ with diagonal torus $S_\C = (\prod_{\tau}\mathbb{G}_m) \times \mathbb{G}_m/\mu_2$ embedded via  $((x_{1,\tau}, \cdots, x_{n,\tau})_{\tau};t) \mapsto \prod_\tau \mathrm{diag}(tx_{1,\tau},\cdots,tx_{n,\tau}) \times t^2$. Moreover, our choice of positive roots is given by the upper triangular Borel. 
 The character group $X^\star(S_\C)$ is the subgroup of $(\ZZ^n)^I \times \ZZ$ of elements $((a_{1,\tau}, \cdots, a_{n,\tau})_\tau ; k)$ with the condition that $\sum_{i,\tau} a_{i,\tau} = k ~\mod 2$. We have  $((a_{1,\tau}, \cdots, a_{n,\tau})_\tau; k). ((x_{1,\tau}, \cdots, x_{n,\tau})_{\tau};t)  = \prod_{i,\tau} x_{i,\tau}^{a_{i,\tau}} t^k$. 
 
 The cocharacter group $X_\star(S_\C)$ identifies  with the subgroup of $(\frac{1}{2} \ZZ^n)^I \times \frac{1}{2}\ZZ$ of elements $((r_{1,\tau}, \cdots, r_{n,\tau}); r)$ with $r+ r_{i,\tau} \in \ZZ$ for all $(i,\tau)$. 
 To an element $((r_{1,\tau}, \cdots, r_{n,\tau}); r)$ we associate the cocharacter $$t \mapsto ( (t^{r_{1,\tau}}, \cdots, t^{r_{n,\tau}})_{\tau}; t^r) = \prod_\tau \mathrm{diag}((t^{r_{1,\tau+r}}, \cdots, t^{r_{n,\tau}+r})) \times t^{2r}.$$ 
 The pairing between characters and cocharacters is given $$\langle
 ((a_{1,\tau}, \cdots, a_{n,\tau})_\tau; k), ((r_{1,\tau}, \cdots,
 r_{n,\tau})_{\tau}; r) \rangle = \sum_{i,\tau}  a_{i,\tau}r_{i,\tau}
 + rk .$$

 We have the usual identification $\hat{G} \simeq G_\C =  \prod_{\tau}( \mathrm{GL}_n) \times \mathbb{G}_m$ with standard torus $\hat{S} = S_\C$.  We use the standard pinning.  %The group $\mathrm{Gal}(\overline{\qq}/\qq)$ acts on $S_\C$ by permuting the $\tau$'s. 
 The complex conjugation acts on $\hat{S}$ by $$((x_{1,\tau}, \cdots, x_{n,\tau})_{\tau};t) \mapsto ((x_{1,\tau}^{-1}, \cdots, x_{n,\tau}^{-1})_{\tau};t)$$ and on the full $\hat{G}$ by $(g,c) \mapsto ( \Phi_N~^tg^{-1} \Phi_N^{-1} c,c)$ where $\Phi_N$ is the anti-diagonal matrix with alternating $1$ and $-1$'s on the anti-diagonal.

\subsection{Formulas for the minuscule coweights} 

We let $\iota: E' \hookrightarrow \overline{\qq}_p$ be an embedding
continuous for the $\lambda'$-adic topology on the source and $p$-adic
topology on the target. Let $\mathfrak{p}_1, \cdots, \mathfrak{p}_m$
be the primes above $p$ in $F$. We assume that they all split in
$L$. We have a partition $I = \coprod I_{i}$ where $I_{i}$ is the set
of embeddings $\tau: F \rightarrow E'$ for which $\iota \circ \tau$
induces the $\mathfrak{p}_i$-adic topology.  The local $L$-group at
$p$ is simply a product
$\prod_{i} (\mathrm{GL}_{n})^{I_i} \times \mathbb{G}_m \rtimes
\Gamma$, where $\Gamma = \hat{\mathbb{Z}}$ is the absolute Galois
group of $\FF_p$. It acts by permutation on each of the sets $I_i$.

We fix $0 \leq i \leq m$ and $1 \leq j \leq n$. We denote by $V_{j,i}$
the representation of $\hat{G}$ whose underlying vector space is
$\otimes_{\tau \in I_i}(\Lambda^{j} \C^n)^\vee$, with action of
$\hat{G} = ( \prod_\tau\mathrm{GL}_n) \times \mathbb{G}_m$ given by
the dual of the $j$-th exterior standard action of the $\tau$-factor
for all $\tau \in I_i$, and the inverse scalar multiplication by
$\mathbb{G}_m$.  The corresponding highest weight $\lambda_{i,j}$ is:
$((a_{1, \tau},\cdots, a_{n, \tau}); -2-\sharp I_i )$ with
$(a_{1, \tau},\cdots, a_{n, \tau}) = (0, \cdots, 0, -1, \cdots, -1)$
(with $j$ many $-1$s ) if $\tau \in I_i$,
$(a_{1, \tau},\cdots, a_{n, \tau}) = (0, \cdots, 0)$ if
$\tau \notin I_i$. This representation extends to a representation of
the local $L$-group.

The corresponding cocharacter of $G_{\C}$ is 
 $$ t \mapsto \Big (\prod_{\tau \notin I_i} 1_{n} \Big ) \times \Big
 (\prod_{\tau \in I_i} \mathrm{diag}( 1_{n-j}, t^{-1} 1_{j}) \Big ) \times t^{-1}$$
 which is given in coordinates by
 $((r_{1, \tau},\cdots, r_{n, \tau}); -\frac{1}{2})$ with $(r_{1,
   \tau},\cdots, r_{n, \tau}) = (\frac{1}{2}, \cdots, \frac{1}{2},
 -\frac{1}{2}, \cdots, -\frac{1}{2})$ (with $n-j$ many $\frac{1}{2}$s) if $\tau \in I_i$,  $(a_{1, \tau},\cdots, a_{n, \tau}) = (\frac{1}{2}, \cdots, \frac{1}{2})$ if $\tau \notin I_i$. 
 
 We let $T_{\mathfrak{p}_i, j}^{naive}$ be the Hecke operator associated to $T_{\lambda_{i,j}}$. We have an embedding $G(\qq_p) \hookrightarrow \prod_{i} \mathrm{GL}_n({F}_{\mathfrak{p}_i}) \times \mathrm{GL}_n({F}_{\mathfrak{p}_i})$, and the operator $T_{\mathfrak{p}_i, j}^{naive}$  is associated to the double coset
  $$ \mathrm{G}(\ZZ_p) \Big (\prod_{l \neq i} 1_n \times
  \mathfrak{p}_l^{-1} 1_n\Big ) \times (\mathrm{diag} ( 1_{n-j}, \mathfrak{p}_i ^{-1}1_{j}) \times  \mathrm{diag} ( \mathfrak{p}^{-1}_i 1_{n-j}, 1_{j}))  \mathrm{G}(\ZZ_p). $$
  
  \begin{rem} Our definition  of the Hecke operators $T_{\mathfrak{p}_i, j}^{naive}$ depends on the choice of a  prime above each $\mathfrak{p}_k$ in $L$ (. This means that some symmetry has been broken.   For an explanation of the use of the double class $$ \mathrm{G}(\ZZ_p) \Big (\prod_{l \neq i} 1_n \times
  \mathfrak{p}_l^{-1} 1_n\Big ) \times (\mathrm{diag} ( 1_{n-j}, \mathfrak{p}_i ^{-1}1_{j}) \times  \mathrm{diag} ( \mathfrak{p}^{-1}_i 1_{n-j}, 1_{j}))  \mathrm{G}(\ZZ_p) $$ rather than its inverse, we refer to Remark \ref{rem-left-rightaction}. \end{rem}
 
As in Section \ref{vector-bundle-dico}, we choose integers  $a_{1,\tau} \geq \cdots \geq a_{p_{\tau}, \tau}$ and $b_{1,\tau} \geq \cdots \geq b_{q_{\tau}, \tau}$ for all $\tau \in I$ and we  consider the representation of the Levi $M_\mu$ of the parabolic $P_\mu$ with  highest weight $\kappa =( ( -a_{p_\tau, \tau}, \cdots, -a_{1,\tau}, b_{1,\tau}, \cdots, b_{q_{\tau}}); - \sum_{\ell, \tau} a_{\ell, \tau} - \sum_{\ell, \tau} b_{\ell, \tau})$. 

\begin{rem} In the rest of this section, we work with the weight $\kappa$, and therefore we fix a particular normalization of  the central character, as justified in Section \ref{vector-bundle-dico}. We leave it to the reader to formulate Lemma \ref{lem-formula-unitary} for another choice of  central character.   Theorem \ref{main-thm-unitary} holds for any normalization of the central character.
\end{rem}

Motivated by Definition \ref{defi-hecke-algebra}, we let
$T_{\mathfrak{p}_i, j} = p^{\langle \lambda_{i,j}, \infty(\kappa,
  \iota) \rangle - \langle \lambda_{i,j}, \rho \rangle}
T_{\mathfrak{p}_i, j}^{naive}$.  We now find a formula for the
coefficient
$\langle \lambda_{i,j}, \infty(\kappa, \iota) \rangle - \langle
\lambda_{i,j}, \rho \rangle$.

\begin{lem}\label{lem-formula-unitary} We have $\langle \lambda_{i,j}, \infty(\kappa, \iota) \rangle - \langle \lambda_{i,j}, \rho \rangle = $
$$  {\sum}_{\tau \notin I_i} \Big( - \sum_{\ell=1}^{q_{\tau}}
b_{\ell, \tau}\Big) + \sum_{\tau \in I_i} \biggl ( \sup_{ 0 \leq r_\tau
    \leq p_\tau, 0 \leq s_\tau \leq q_\tau, ~s_{\tau} + r_{\tau} =
    n-j} \Big \{ - \sum_{\ell = r_{\tau} + 1}^{p_{\tau}} a_{\ell,
    \tau}- \sum_{\ell = q_{\tau} - s_{\tau}+1}^{q_{\tau}} b_{\ell,
    \tau}  - r_{\tau}(q_{\tau}-s_\tau) \Big\}\biggr)$$
\end{lem}

\begin{demo} We observe that $\rho = (( \tfrac{n-1}{2}, \cdots,
  \tfrac{1-n}{2}); 0)$ so $\kappa + \rho$ is equal to
  $$  ((-a_{p_\tau, \tau}+ \tfrac{n-1}{2}, \cdots,- a_{1,\tau} +\tfrac{n-1}{2} - p_\tau+1, b_{1,\tau} - \tfrac{n-1}{2} + q_{\tau}-1, \cdots, b_{q_{\tau}} -\tfrac{n-1}{2}); - \sum_{\ell, \tau} a_{\ell, \tau} - \sum_{\ell, \tau} b_{\ell, \tau}) .$$
  Now we need to compute the pairing
  $\langle \lambda_{i,j}, \infty(\kappa, \iota) \rangle - \langle
  \lambda_{i,j}, \rho \rangle$. 
The product $\langle \lambda_{i,j}, \infty(\kappa, \iota)
\rangle$ decomposes as $$\sum_\tau \langle (\lambda_{i,j})_{\tau}, \infty(\kappa, \iota)_{\tau} \rangle_\tau
- \Big (\frac{1}{2} \sum_{\ell, \tau} a_{\ell, \tau} +\sum_{\ell, \tau}
b_{\ell, \tau} \Big ).$$
For all $\tau \notin I_i$, we have $$\langle (\lambda_{i,j})_{\tau}, \infty(\kappa, \iota)_{\tau} \rangle_{\tau}
=  \frac{1}{2} \Big (\sum_{\ell} a_{\ell, \tau} - \sum_{\ell} b_{\ell,
  \tau}\Big ).$$
For all $\tau \in I_i$, there will be integers  $0 \leq r_{\tau} \leq
p_\tau$ and $0 \leq s_{\tau} \leq q_{\tau}$  such that $s_{\tau} +
r_{\tau} = n- j $ and such that the first $n-j$ coordinates of
$\infty(\kappa, \iota)_\tau$ put in the dominant form will be
(possibly not in this order)  $$\{ -a_{1,\tau} + \tfrac{n-1}{2} -
p_\tau+1, \cdots, -a_{r,\tau} + \tfrac{n-1}{2} - p_{\tau} + r_\tau,
b_{q_{\tau}} - \tfrac{n-1}{2}, b_{q_\tau-s_{\tau} + 1} -
\tfrac{n-1}{2}+ s_{\tau} -1\} .$$
In this case, an easy computation shows  that  
$$ \langle (\lambda_{i,j})_{\tau}, \infty(\kappa, \iota)_{\tau}
\rangle_{\tau} = \frac{1}{2} \Big (\sum_{\ell = 1}^{r_\tau} a_{\ell, \tau} -\sum_{\ell = r_\tau +1}^{p_\tau} a_{\ell, \tau} + \sum_{\ell = 1}^{ q_{\tau}-s_\tau} b_{\ell, \tau} - \sum_{\ell = q_\tau-s_{\tau}+1}^{q_\tau} b_{\ell, \tau} \Big)  + $$ $$ -\tfrac{n-1}{2}r_\tau +  \tfrac{r_\tau(r_\tau -1)}{2} + (p_\tau-r_\tau) r_\tau  + \tfrac{n-1}{2}s_\tau  - \tfrac{s_\tau(s_{\tau} -1)}{2}.$$ 
Moreover, we have the pairing
$\langle \lambda_{i,j}, \rho \rangle_{\tau} = \tfrac{j(n-j)}{2}$ if
$\tau \in I_i$ and $\langle \lambda_{i,j}, \rho \rangle_{\tau}= 0$
otherwise. So we finally get that for all $\tau \in I_i$
\begin{multline*}  \langle (\lambda_{i,j})_{\tau}, \infty(\kappa,
  \iota)_{\tau} \rangle_{\tau}  - \langle \lambda_{i,j}, \rho
  \rangle_{\tau}
  \\
  =
  \frac{1}{2} \Big(\sum_{\ell = 1}^{r} a_{\ell, \tau} -\sum_{\ell=
   r_\tau +1}^{p_\tau} a_{\ell, \tau} + \sum_{\ell = 1}^{
   q_{\tau}-s_\tau} b_{\ell, \tau} - \sum_{\ell =
   q_\tau-s_{\tau}+1}^{q_\tau} b_{\ell, \tau} \Big) - r_\tau
 (q_{\tau}-s_{\tau})
 \end{multline*}
 while for all $\tau \notin I_i$, we have 
  $$  \langle (\lambda_{i,j})_{\tau}, \infty(\kappa,
  \iota)_{\tau} \rangle_{\tau}  - \langle \lambda_{i,j}, \rho
  \rangle_{\tau} = \frac{1}{2} \Big (\sum_{\ell} a_{\ell, \tau} -
  \sum_{\ell} b_{\ell, \tau}\Big ).$$
   It remains to add the factor  $\sum_{\tau} -(\tfrac{1}{2} \sum_{\ell, \tau} a_{\ell, \tau} +\sum_{\ell, \tau} b_{\ell, \tau})$ to conclude.  
  
  \end{demo}
  
  \subsubsection{The main theorem for unitary groups}
  
 We have that $\mathcal{H}_{p, \kappa, \iota }^{int} = \otimes_{0 \leq i \leq m} \ZZ[ T_{\mathfrak{p}_i,j},~0 \leq j \leq n, T_{\mathfrak{p}_i,0}^{-1}, T_{\mathfrak{p}_i,n}^{-1}]$ and the following result gives some evidence towards Conjecture \ref{conj3}. 
 
 \begin{thm}\label{main-thm-unitary}The operators $T_{\mathfrak{p}_i,j},~0 \leq j \leq n,~1 \leq i \leq m$, act on $\mathrm{R}\Gamma(\mathfrak{Sh}^{tor}_{K,\Sigma} , \mathcal{V}_{\kappa, K, \Sigma})$  and $\mathrm{R}\Gamma(\mathfrak{Sh}^{tor}_{K,\Sigma} , \mathcal{V}_{\kappa, K, \Sigma}(-D_{K, \Sigma}))$.
 \end{thm}
 
 \begin{rem} We do not know how to prove the commutativity of the
   operators $T_{\mathfrak{p}_i,j}$ in general. The main point is that
   although the operators $T_{\mathfrak{p}_i,j}$ are associated to
   minuscule coweights, the composition of two such operators will not
   in general be associated to a minuscule coweight. 
 \end{rem}
 \begin{coro}\label{coro-conj2-unit} Conjecture \ref{conj2} holds in this case.
 \end{coro}
 \begin{demo} Conjecture \ref{conj2} asserts that
   $$\mathrm{Im} \big( \HH^i(\mathfrak{Sh}^{tor}_{K,\Sigma} ,
   \mathcal{V}_{\kappa, K, \Sigma}) \rightarrow \HH^i(
   \mathfrak{Sh}^{tor}_{K,\Sigma} , \mathcal{V}_{\kappa, K, \Sigma})
   \otimes_{\ZZ_p} \qq_p \big)$$ is a lattice in
   $\HH^i( \mathfrak{Sh}^{tor}_{K,\Sigma} , \mathcal{V}_{\kappa,
     \Sigma}) \otimes_{\ZZ_p} \qq_p$ stable under the action of
   $\mathcal{H}_{p, \kappa, \iota }^{int}$. This follows from
   TheoremÊ\ref{main-thm-unitary}.
 \end{demo}
 
 \subsubsection{Example:  the group $\mathrm{GU}(2,1)$} Let us assume that $F = \qq$, and that the signature $(p_\tau, q_\tau) = (2,1)$ for a chosen embedding $\tau: L \rightarrow \overline{\qq}$. Let $p$ be a prime that splits in $L$. We choose an embedding $\iota: L \hookrightarrow \overline{\qq}_p$ which corresponds to a prime $\mathfrak{p}$ above $p$. We have $\mathfrak{p} \mathfrak{p}^c = p$. We have an isomorphism  $G_{\qq_p} = \mathrm{GL}_3 \times \mathbb{G}_m$

 The spherical Hecke algebra at $p$ is the  polynomial algebra
 generated by the  characteristic functions (identified with double cosets):
  \begin{enumerate}
  \item $T_0^{naive} =  K_p  ( \mathrm{diag}(1,1,1) \times p^{-1}) K_p$, $(T_0^{naive})^{-1}$, 
 \item $T_1^{naive} =  K_p   (\mathrm{diag}(1,1,p^{-1}) \times p^{-1}) K_p$,
  \item $T_2^{naive} = K_p  ( \mathrm{diag}(1,p^{-1},p^{-1}) \times p^{-1}) K_p$,
  \item $T_3^{naive} = K_p (\mathrm{diag}(p^{-1},p^{-1},p^{-1}) \times p^{-1}) K_p$, $(T_3^{naive})^{-1}$, 
  \end{enumerate}
  
  There  are associated representations of the dual group $V_{j} = (\Lambda^j \C^3)^\vee $ (and action of $\mathbb{G}_m$ by the inverse character),  for $0 \leq j \leq 3$. The relation given by the Satake transform is  (observe that all these representations are minuscule): 
   \begin{enumerate}
 \item $T^{naive}_0 =  [V_{0,\tau}] $,
 \item $T^{naive}_1 =  p [V_{1,\tau}]$,
  \item $T^{naive}_2 = p[V_{2,\tau}]$,
  \item $T^{naive}_3 = [V_{3,\tau}]$.
  \end{enumerate}

 We now pick integers $(k_1,k_2, k_3)$ and consider the automorphic
 vector bundle: $$\mathrm{Sym}^{k_1-k_2} \omega_{A, \tau, +} \otimes
 \det^{k_2} \omega_{A, \tau, +} \otimes \det^{k_3} \omega_{A, \tau,-}
 . $$ 
 It is the sheaf $\mathcal{V}_{\kappa, K, \Sigma}$ for $\kappa  = (-k_2,-k_1, k_3; -k_2-k_1-k_3)$. It follows from the construction that the sheaf $\det \mathcal{H}_{1,dR} ( A)_{\tau, +}$ is tivial (viewed as a sheaf, the equivariant action is not trivial). It is therefore harmless to assume that $k_3$ is constant and we assume that $k_3=1$. 
 
 We conclude that:
  $$ \mathcal{H}^{int}_{p, \kappa, \iota} = \ZZ\big [
  \big(p^{-1}T_0^{naive}\big)^{±1}, p^{-1- \inf\{1, k_2\}}T_1^{naive},
  p^{-1- \inf\{k_2, k_1+k_2\}}T_2^{naive},
  \big(p^{-k_1-k_2}T_3^{naive}\big)^{±1} \big ]$$

\section{Local models in the linear case}\label{sect-local-model-unitary}

In this section we will prove Theorem \ref{main-thm-unitary}.  As
in the symplectic case, the actions of the normalized Hecke operators
will be defined using the results of Section \ref{residue}, in
particular the construction in Example \ref{ex:corres}. In order to do
this, we need to understand the integrality properties of the Hecke
correspondences with respect to automorphic vector bundles and also
the pullback maps on differentials. This will be done by using the
theory of local models of Shimura varieties.

After recalling the basic facts about the local models for general linear
groups, we prove two results, Proposition \ref{prop-formula2} (on the
integrality properties of differentials) and Lemma
\ref{lem-weight-formula2} (on the integrality properties on automorphic
bundles) which are used to construct normalized Hecke operators on the
local model in Proposition \ref{norm-unitary}. These computations
imply Theorem \ref{main-thm-unitary} in the same way as in the proof
of Theorem \ref{main-thm-symplectic}.

\subsection{Definition} Let $n,p,q \in \mathbb{Z}_{\geq 1}$ be
integers such that $n=p+q$.  We now consider modules
$V_0, \cdots, V_{n-1}= \ZZ^{n}$ and the following chain:
  
  $$V_\bullet: V_0 \rightarrow V_1 \rightarrow \cdots \rightarrow V_{n-1} \rightarrow V_0$$ where 
   the map from $V_i$ to $V_{i+1}$ is given in the canonical basis $(e_1, \cdots, e_{n})$ of $\ZZ^{n}$ by the map $e_j \mapsto e_j$ if $j \neq i+1$ and $e_{i+1} \mapsto p e_{i+1}$.    Whenever  necessary, indices are taken modulo $n$ so that $V_{n}:= V_0$ and the chain $V_\bullet$ can be seen as an infinite chain.

  Fix a set $\emptyset \neq I \subset \{ 0, 1, \cdots, n-1 \}$.  We define the local model  functor  $\mathbf{M}_I: \ZZ- \mathbf{ALG} \rightarrow \mathbf{SETS}$ which associates to an object $R$ of $\ZZ- \mathbf{ALG}$  the set of isomorphism classes of commutative diagrams  
\begin{eqnarray*}
\xymatrix{ \ar[r]& V_{i_0}\otimes_{\ZZ} R \ar[r]  & V_{i_1}\otimes_{\ZZ} R \ar[r]  & \cdots  \ar[r] & V_{i_m}\otimes_{\ZZ} R \ar[r] & \\
\ar[r] &F_{i_0} \ar[u] \ar[r]  & F_{i_1} \ar[r] \ar[u] & \cdots  \ar[r]  & F_{i_m} \ar[r]\ar[u]&}
\end{eqnarray*}
where $i_0 < i_1 \cdots <i_m $ are such that
$\{ i_0, \cdots, i_m\} = I$ and the modules $F_{i_j}$ for
$0 \leq j \leq m$ are rank $q$ locally direct factors.

The functor $\mathbf{M}_I$ is represented by a projective scheme $M_I$
which is a closed subscheme of a product of Grassmannians.

When $\emptyset \neq J \subset I$, there is an obvious map
${M}_I \rightarrow M_J$ given by forgetting the modules $F_j$ for
$j \in I \setminus J$.

\subsection{The affine Grassmannian}
   
  We now consider modules $\mathcal{V}_0, \cdots, \mathcal{V}_{n-1}= \F_p[[t]]^{n}$ and the  chain:
  $$\mathcal{V}_\bullet = \mathcal{V}_0 \rightarrow \mathcal{V}_1 \rightarrow \cdots \rightarrow \mathcal{V}_{n-1} \rightarrow \mathcal{V}_0$$ where 
   the map from $\mathcal{V}_i$ to $\mathcal{V}_{i+1}$ is given in the canonical basis $(e_1, \cdots, e_{n})$ of $\F_p[[t]]^{n}$ by the map $e_j \mapsto e_j$ if $j \neq i+1$ and $e_{i+1} \mapsto t e_{i+1}$.    Whenever necessary, indices are taken modulo $n$ so that $\mathcal{V}_{n}:= \mathcal{V}_0$ and the chain $\mathcal{V}_\bullet$ is made infinite. Observe that $\mathcal{V}_\bullet \otimes_{\F_p[[t]]} \F_p = V_\bullet \otimes_{\ZZ} \FF_p$.

   Let $LG$ denote the loop group of $\mathrm{GL}_{n}$ over
$\mathbb{F}_p$. The group $LG$ acts naturally on $\mathcal{V}_0 \otimes_{\F_p[[T]]} \F_p((T))$ and therefore it acts 
 on the chain $\mathcal{V}_\bullet \otimes_{\F_p[[T]]} \F_p((T))$. 
 
 Let $ \emptyset \neq I \subset  \{ 0, 1, \cdots, n -1\}$ and let $\mathcal{V}^I_\bullet$ be the subchain of $\mathcal{V}_\bullet$ where we keep only the modules indexed by elements $i \in I$ (modulo $n$). 
 
 We denote by $\mathcal{P}_I$ the parahoric subgroup of $LG$ of
 automorphisms of the chain $\mathcal{V}^I_\bullet$. For any $I$, we
 define the flag variety as the ind-scheme
 $\mathcal{F}_I:= LG/\mathcal{P}_I$.

 \subsection{Stratification of the local model}
  As in the case
 of symplectic groups, the special fibre $\ov{M}_I$ of $M_I$ at $p$
 embeds as a finite union of $\mathcal{P}_I$ orbits in
 $\mathcal{F}_I$.

We recall the description of the map $\ov{M}_I\rightarrow \mathcal{F}_I$. 
Given a diagram \begin{eqnarray*}
\xymatrix{ V_{i_0}\otimes_{\ZZ} R \ar[r]  & V_{i_1}\otimes_{\ZZ} R \ar[r]  & \cdots  \ar[r] & V_{i_m}\otimes_{\ZZ} R  \\
F_{i_0} \ar[u] \ar[r]  & F_{i_1} \ar[r] \ar[u] & \cdots  \ar[r]  & F_{i_m} \ar[u]}
\end{eqnarray*} corresponding to an $R$-point of $\ov{M}_I$, we can construct a new  diagram: 
\begin{eqnarray*}
\xymatrix{ \mathcal{V}_{i_0}\otimes_{\ZZ} R \ar[r]  & \mathcal{V}_{i_1}\otimes_{\ZZ} R \ar[r]  & \cdots  \ar[r] & \mathcal{V}_{i_m}\otimes_{\ZZ} R  \\
\mathcal{F}_{i_0} \ar[u] \ar[r]  & \mathcal{F}_{i_1} \ar[r] \ar[u] & \cdots  \ar[r]  & \mathcal{F}_{i_m} \ar[u] \\
t\mathcal{V}_{i_0}\otimes_{\ZZ} R \ar[r] \ar[u]  & t\mathcal{V}_{i_1}\otimes_{\ZZ} R \ar[r] \ar[u]  & \cdots  \ar[r] & t\mathcal{V}_{i_m}\otimes_{\ZZ} R \ar[u]}
\end{eqnarray*}
where all the vertical maps are inclusions and  each $\mathcal{F}_{i_j}$ is determined by the property that $\mathcal{F}_{i_j}/ t\mathcal{V}_{i_j}\otimes_{\ZZ} R  = F_{i_j} \hookrightarrow ( \mathcal{V}_{i_j} / t\mathcal{V}_{i_j})\otimes_{\ZZ} R  = V_{i_j} \otimes R$.  The chain $ \mathcal{F}_\bullet$ determines an $R$-point of $\mathcal{F}_I$.

We now recall the combinatorial description of the image of $\ov{M}_I$
in $\mathcal{F}_I$: Fix a Borel subgroup $B$ of $\mathrm{GL}_{n}$ and
a maximal torus $T\subset B$. This gives a base for the root datum of
$\mathrm{GL}_{n}$. Let $\widetilde{W}$ be the extended affine Weyl
group of $\mathrm{GL}_{n}$. This is the semi-direct product of the
Weyl group $W = \mathcal{S}_{n} $ and the cocharacter group
$\mathrm{X}_{\star}(T) = \ZZ^n$.  It can be realized as a subgroup of the
group of affine transformations of $\mathbb{Z}^{n}$ with
$\mathbb{Z}^{n}$ acting by translation and $\mathcal{S}_{n}$ by
permutation of the coordinates.  Let $s_i = (i,i+1)$ for
$1 \leq i \leq g-1$, and
$s_0 = \mr{t}_{(-1, 0, \cdots, 0, 1)}\rtimes( 1,n)$ be the usual choice
of simple reflections.  We let
$\mu = (1, \cdots, 1, 0, \cdots 0) \in \mathbb{Z}^{n}$ (where $1$
occurs $p$ times) be the minuscule coweight corresponding to our
situation.

For each $\emptyset \neq I$ as above, we let $W_I$ be the subgroup of $\widetilde{W}$ generated by the simple reflexions $s_i, i \notin I$. This is a finite group.  The $\mathcal{P}_I$ orbits in $\mathcal{F}_I$ are parametrised by the double cosets $W_I\backslash \widetilde{W} / W_I$.  For any $w \in \widetilde{W}$ we denote the orbit corresponding to the double coset $W_I w W_I$ by $U_{I,w}$ and the orbit closure by $X_{I,w}$.
The orbits included in $\ov{M}_I$ are parametrized by  the finite subset $Adm_I(\mu)$ of $W_I\backslash \widetilde{W} / W_I$ of $\mu$-admissible elements. 

Whenever $J \subset I$ we have a map $M_I \rightarrow M_J$. This map is surjective and $Adm_J(\mu)$ is the image of $Adm_I(\mu)$ in $W_J\backslash \widetilde{W} / W_J$.

When $I = \{ 0, \cdots, n-1\}$, Kottwitz and Rapoport give a
description of the set $Adm_I(\mu)$.  The set $Adm_I(\mu)$ is
precisely the subset of $\widetilde{W}$ of elements which are $\leq$
(in the Bruhar order) a translation $\mr{t}_{w\mu}$ for an element
$w \in W$.  In particular, the orbits that are open in $\ov{M}_I$
correspond to the translations $\mr{t}_{w\mu}$.

\subsection{Irreducible components} 

\subsubsection{The case that $I= \{0, \cdots, n-1\}$}
  There are
actually $\frac{n!}{p!q!}$ translations $w\cdot \mu$ and they
correspond to the open stratum in each of the $\frac{n!}{p!q!}$
irreducible components of $\ov{M}_I$ when $I = \{0, \cdots, n-1\}$.

These $\frac{n!}{p!q!}$ translations are parametrized by $W/W_c$ where
$W_c$ is the subgroup of $W$ of elements which stabilize $\mu$. It
identifies with the elements in $W \subset \mathcal{S}_{n}$ which
preserve the sets $\{1,\cdots, p\}$ and $\{p+1, \cdots, n\}$. This
group is isomorphic to $\mathcal{S}_p \times \mathcal{S}_q$.

We can concretely determine an element of each of the orbits
corresponding to these $\mr{t}_{w \mu}$ in $M_I$ for
$I= \{0, \cdots, n-1\}$ as follows.  The group $\widetilde{W}$ can be
viewed as a subgroup of $\mathrm{GL}_n( \F_p((t)))$. The element
$\mr{t}_{\mu}$ is identified with
$\mathrm{diag} (t,\cdots, t, 1, \cdots 1)$.
   
We now consider the inclusion of chains:
$$ t \mathcal{V}_\bullet \subset \mr{t}_{w\mu} \mathcal{V}_\bullet
\subset \mathcal{V}_\bullet.$$
By reduction modulo $t$ and using the
identification
$\mathcal{V}_\bullet \otimes_{\F_p[[t]]} \F_p = V_\bullet
\otimes_{\ZZ} \FF_p$, we deduce that
$\mr{t}_{w\mu} \mathcal{V}_\bullet/ t \mathcal{V}_\bullet
\hookrightarrow V_\bullet$ defines an $\F_p$-point of $M_I$, which
represents the $w\mu$-orbit.

   \subsubsection{The case that $\sharp I = 1$} For any $j, j' \in I$,
   the spaces $M_{\{j\}}$ and $M_{\{j'\}}$ are canonically
   isomorphic. The special fiber $\ov{M}_{\{0\}}$ of $M_{\{0\}}$ is
   smooth and irreducible of dimension $pq$. Moreover, there is a
   single orbit.

\subsubsection{The case that $I = \{0,j\}$}

\begin{lem} The special fiber $\ov{M}_{\{0,j\}}$ of $M_{\{0,j\}}$ has
  $\inf \{ p, j\} - \sup \{0, j-q\} +1$ irreducible components, indexed by integers
  $\sup \{0, j-q\} \leq r \leq \inf \{ p, j\} $. For each
  $\sup \{0, j-q\}  \leq r \leq \{ p, j\} $, a representative of the $r$-stratum is
  given by taking
   $$ F_0(r) = F_j(r) =   \langle e_{r+1}, \cdots, e_p,  e_{n-j+r+1}, \cdots, e_{n}\rangle$$
   This corresponds to the element $ \mr{t}_{w(r)\mu} =  \mathrm{diag} ( t \mathrm{Id}_r, \mathrm{Id}_{j-r},  t \mathrm{Id}_{p-r}, \mathrm{Id}_{q-j+r}) \in LG$. 
  \end{lem}

  \subsection{Local geometry of  the local model}
The local model $M_{\{0\}}$ is  smooth of relative dimension
${pq}$. The  other local models $M_I$  are isomorphic to $M_{\{0\}}$
over $\Spec~\ZZ[1/p]$ but they have  singular special fiber at
$p$. Nevertheless, we have the following analogue of Theorem
\ref{thm-GortzHe}, the first
part due to G\"ortz \cite{gortz-unitary} and the second to He
\cite[Theorem 1.2]{he-normality}
\begin{thm}\label{thm-GortzHe2} The local models $M_I$ are flat over
  $\ZZ$ and $\ov{M}_I$ is reduced.  Furthermore, $M_I$ is Cohen--Macaulay.
\end{thm}

\subsection{Hecke correspondence of the  affine Grassmannians}
We consider  the correspondence:
\begin{eqnarray*}
\xymatrix{ & \mathcal{F}_{\{0,j\}} \ar[ld]_{p_2} \ar[dr]^{p_1} & \\
\mathcal{F}_{\{j\}} && \mathcal{F}_{\{0\}}}
\end{eqnarray*}
We now pick the element $ w(r) \mu \in \widetilde{W}$ and restrict
this diagram to get maps:
\begin{eqnarray*}
\xymatrix{ & U_{\{0,j\}, w(r) \mu} \ar[ld]_{p_2} \ar[dr]^{p_1} & \\
U_{\{j\}, w(r) \mu } && U_{\{0\}, w(r) \mu}}
\end{eqnarray*}
\begin{proposition}\label{prop-formula2} The map on  differentials  $ \mathrm{d}p_1:  p_1^\star \Omega^1_{U_{\{0\}, w(r) \mu}/\F_p} \rightarrow  \Omega^1_{U_{\{0,j\}, w(r) \mu}/\F_p}$  has kernel and cokernel a locally free sheaf of rank $(j-r)(p-r)$.
\end{proposition}

\begin{demo} We have $$U_{\{0,j\}, w(r) \mu} \simeq \mathcal{P}_{\{0,j\}} / \big(\mathcal{P}_{\{0,j\}} \cap \mr{t}_{w(r) \mu} \mathcal{P}_{\{0,j\}} \mr{t}_{w(r) \mu}^{-1}\big)$$
$$U_{\{0\}, w(r) \mu} \simeq \mathcal{P}_{\{0\}} /
\big(\mathcal{P}_{\{0\}} \cap \mr{t}_{w(r) \mu} \mathcal{P}_{\{0\}}
\mr{t}_{w(r) \mu}^{-1}\big) ,$$
and $p_1$  is the obvious  $\mathcal{P}_{(0,j)}$- equivariant
projection: $$\mathcal{P}_{\{0,j\}} / \big(\mathcal{P}_{\{0,j\}} \cap
\mr{t}_{w(r) \mu} \mathcal{P}_{\{0,j\}} \mr{t}_{w(r) \mu}^{-1}\big)
\rightarrow  \mathcal{P}_{\{0\}} / \big(\mathcal{P}_{\{0\}} \cap
\mr{t}_{w(r) \mu} \mathcal{P}_{\{0\}} \mr{t}_{w(r) \mu}^{-1}\big) .$$
Because the map is $\mathcal{P}_{(0,j)}$-equivariant,  it suffices to prove the claim in the tangent space at the identity. 

We first determine  the shape of $\mathcal{P}_{\{0\}}$ and $\mathcal{P}_{\{0,j\}}$.  The group $\mathcal{P}_{\{0\}}$ is the hyperspecial  subgroup of $LG$, whose $R$-points are $G(R[[t]])$.  
 The group  $\mathcal{P}_{\{0, j\}}$ is the parahoric  group with shape: 
 $$ \mathcal{P}_{\{0,g\}} (R) = \{ M =    \begin{pmatrix} % or pmatrix or bmatrix or Bmatrix or ...
       a & b  \\
        tc & d \\
    
    \end{pmatrix} \in LG(R) \} $$
    where $a \in M_{j \times j}(R[[t]])$, $d \in M_{n-j \times n- j}(R[[t]])$, $c \in M_{n-j \times j}(R[[t]])$ and $d \in  M_{j \times n-j}(R[[t]])$.

One computes that  $\mathcal{P}_{\{0\}} \cap \mr{t}_{w(r) \mu} \mathcal{P}_{\{0\}} \mr{t}_{w(s) \mu}^{-1}$ consists of matrices of the following shape (the $\star$ have integral values, the rows and columns are of size $r$, $j-r$, $p-r$,  and $q-j+r$): 
   $$\begin{pmatrix} % or pmatrix or bmatrix or Bmatrix or ...
      \star  &  t\star &  \star & t\star \\
      \star & \star & \star & \star  \\
     \star &t \star & \star & t \star \\
    \star & \star & \star & \star \\
   \end{pmatrix} $$

One also computes that  $\mathcal{P}_{\{0,j\}} \cap \mr{t}_{w(r) \mu} \mathcal{P}_{\{0,j\}} \mr{t}_{w(r) \mu}^{-1}$ consists of matrices with the following shape:
$$\begin{pmatrix} % or pmatrix or bmatrix or Bmatrix or ...
      \star  &  t\star &  \star & t\star \\
      \star & \star & \star & \star  \\
    t \star &t^2 \star & \star & t \star \\
    t\star &t \star & \star & \star \\
   \end{pmatrix} $$

Passing to the Lie algebras, we easily see that the kernel of the map
$$\mathfrak{p}_{\{0,j\}} / \big(\mathfrak{p}_{\{0,j\}} \cap \mr{t}_{w(r) \mu} \mathfrak{p}_{\{0,j\}} \mr{t}_{w(r) \mu}^{-1}\big) \rightarrow  \mathfrak{p}_{\{0\}} / \big(\mathfrak{p}_{\{0\}} \cap \mr{t}_{w(r) \mu} \mathfrak{p}_{\{0\}} \mr{t}_{w(r) \mu}^{-1}\big)$$
is  the set of matrices of the form: 
 $$\begin{pmatrix} % or pmatrix or bmatrix or Bmatrix or ...
      0  &  0 &  0 &0\\
      0 & 0 & 0 & 0  \\
     0 & t A &0 & 0 \\
    0& 0 & 0 & 0 \\
   \end{pmatrix} $$
with $A\in M_{p-r\times j-r}(\F_p)$.

\end{demo} 

\subsection{ The Hecke correspondence on the local model}

 We consider  the correspondence:
\begin{eqnarray*}
\xymatrix{ & M_{\{0,j\}} \ar[ld]_{p_2} \ar[dr]^{p_1} & \\
M_{\{j\}} && M_{\{0\}}}
\end{eqnarray*}

\subsubsection{Sheaves on the local model}
 Let $\kappa = (a_1, \cdots, a_p, b_{1}, \cdots, b_p)$, with $a_1 \geq
 \cdots \geq a_p$ and $b_{1} \geq \cdots \geq b_q$, be a weight.    We have a  locally free rank $p$ sheaf $(V_0/F_0)^\vee$, and a locally free rank $q$ sheaf $F_0$ over $M_0$, and similarly over $M_j$.  We let $L_0 = (V_0/F_0)^\vee \oplus F_0$ and $L_j  = (V_j/F_j)^\vee \oplus F_j$.  
 To the weight $\kappa$ we can naturally attach locally free sheaves  $L_{0, \kappa}$ and $L_{j,\kappa}$ using the procedure described in Section \ref{section-Sheaves-loc-mod}.

 \subsubsection{The map $L_{j,\kappa} \rightarrow L_{0,\kappa}$}

 The natural map $V_0 \rightarrow V_j$ induces a map $V_0/F_0 \rightarrow V_j/F_j$ and by duality a map $V_j/F_j \rightarrow V_0/F_0$ that we denote by $\alpha_1$. The map $V_j \rightarrow V_0$ induces a map $F_j \rightarrow F_0$ denoted by $\alpha_2$. We let $\alpha = (\alpha_1, \alpha_2): L_j \rightarrow L_0$. The map $\alpha$ is an isomorphism over $\ZZ[1/p]$ and it induces an isomorphism $\alpha^\star: p_2^\star  L_{j, \kappa} \rightarrow  p_1^\star L_{0, \kappa}$.  We now investigate the integral properties of this map. 
 
 \begin{lem}\label{lem-weight-formula2}  Let $\kappa = (k_1, \cdots,
   k_n)$. Let $\sup \{0, j-q\} \leq r \leq \inf\{p, j\}$. Let $\xi$ be the generic point of $U_{\{0,j\}}(r)$.  The map $\alpha^\star$  induces  a map $$\alpha^\star: (p_2^\star L_{j,\kappa})_{\xi} \rightarrow p^{ a_p + \cdots + a_{p-r+1} } p^{ b_{q} + \cdots + b_{j-r}} (p_1^\star L_{0,\kappa})_{\xi}$$ over the local ring $\oscr_{M_{0,j}, \xi}$. 
 \end{lem} 
 
 \begin{demo}  We first check that over $U_{\{0,j\}}(r)$, the map $\alpha_1$ has kernel and cokernel a locally free sheaf of rank $r$ and the map $\alpha_2$ has kernel and cokernel a locally free sheaf of rank $q-j+r$. Indeed, it is enough to check this at the point corresponding to $\mr{t}_{w(r)\mu}$, in which case  the corresponding diagram is:
 \begin{eqnarray*}
 \xymatrix{ \mathcal{V}_0 \ar[rr]^{\mathrm{diag}( t 1_j, 1_{n-j})} && \mathcal{V}_j   \ar[rr]^{\mathrm{diag}(1_j, t1_{n-j})} && V_0 \\
 \mathcal{V}_0 \ar[rr]^{\mathrm{diag}( t 1_j, 1_{n-j})}\ar[u]^{\mr{t}_{w(r) \mu}}& & \mathcal{V}_j \ar[u]^{\mr{t}_{w(r)\mu}} \ar[rr]^{\mathrm{diag}(1_j, t1_{n-j})} && V_0 \ar[u]^{\mr{t}_{w(r) \mu}}}
 \end{eqnarray*}
 and our claim is simply that the map $\mathcal{V}_0/\mr{t}_{w(r) \mu} \mathcal{V}_0 \rightarrow \mathcal{V}_j/\mr{t}_{w(r) \mu} \mathcal{V}_j$ has kernel of rank $r$ and $\mr{t}_{w(r) \mu} \mathcal{V}_j/t V_j \rightarrow \mr{t}_{w(r) \mu} \mathcal{V}_0/tV_0$ has kernel of rank $q-j+r$ . This is obvious. We now conclude the proof as in Lemma \ref{lem-weight-formula}. 
 
 \end{demo}

\subsubsection{The cohomological correspondence}

We may now construct a cohomological correspondence. By Proposition \ref{prop-trace} and Theorem \ref{thm-GortzHe2}, we have a fundamental class  $p_1^\star \oscr_{M_{\{0\}}} \rightarrow p_1^! \oscr_{M_{\{0\}}}$.  Moreover, the sheaf $p_1^! \oscr_{M_{\{0\}}}$ is a CM sheaf. 

There is also a rational map $\alpha: p_2^\star L_{j, \kappa} \dashrightarrow p_1^\star L_{0,\kappa}$, so that putting everything together, we have a rational map (the naive cohomological correspondence):
$$ T^{naive}: p_2^\star L_{j, \kappa} \dashrightarrow p_1^! L_{0, \kappa}$$

We may now normalize this correspondence. 
\begin{proposition} \label{norm-unitary} Let $T = p^{- \inf \big \{\sum_{\ell =p-r +1}^{p} a_\ell + \sum_{\ell =j-r}^{q} b_q+  (j-r)(p-r)\big\}} T^{naive}$. Then $T$ is a true cohomological correspondence: 
$$ T: p_2^\star L_{j, \kappa} \rightarrow p_1^! L_{0, \kappa}$$
\end{proposition}
\begin{demo}
Because $p_1^! L_{0, \kappa}$ is a $CM$ sheaf, any rational map from a
locally free sheaf into $p_1^! L_{0, \kappa}$ is well defined if it is
well defined in codimension $1$. So it is enough to check that it is
well defined on the generic points of all the components in the
special fiber.  Let $\sup\{0, j-q\} \leq r \leq \inf\{j, p\}$. Let $\xi$ be the generic point of the stratum
$U_{\{0,j\}}(r)$. We then see that $T^{naive}:  (p_2^\star L_{j, \kappa} )_\xi \rightarrow  p^{ \sum_{\ell =p-r +1}^{p} a_\ell + \sum_{\ell =j-r}^{q} b_q+  (j-r)(p-r)} (p_1^! L_{0, \kappa})_\xi$ by combining Lemma \ref{lem-weight-formula2} and Proposition \ref{prop-formula2}. 
\end{demo}

\subsubsection{Proof of Theorem \ref{main-thm-unitary}} Given
Proposition \ref{norm-unitary}, this is very similar to the proof of
the main theorem in the symplectic case (see Section
\ref{proof-main-symp}), and left to the reader.

\section{Weakly regular automorphic representations and Galois representations} 

In this section we  study a class of automorphic forms that  realize in the coherent cohomology of Shimura varieties and have associated Galois representations.

\subsection{Weakly regular, odd, essentially (conjugate) self dual
  algebraic cuspidal automorphic representations}

Let $L$ be a $CM$ or totally real field, $F$ the maximal totally real
subfield of $L$, and $c$ the complex conjugation (trivial if $L=F$ is
a totally real field). We let $I = \mathrm{Hom}(F, \C)$ and
$J = \mathrm{Hom}(L, \C)$.

Let $\pi$ be a cuspidal automorphic representation of
$ \mathrm{GL}_n/L$ (we do not assume that the central character of
$\pi$ is unitary).  We define below certain properties $(1)$, $(2)$,
$(3)$ and $(4)$ of such representations. We will say that a $\pi$
satisfying these properties is a weakly regular, algebraic, odd,
essentially (conjugate) self dual, cuspidal automorphic
representation.

\begin{enumerate}
\item \textbf{Essentially (conjugate)  self dual.} We say that $\pi$ is essentially  self dual when $L$ is totally real and essentially conjugate self dual  when $L$ is $CM$ if: 
$$ \pi^c = \pi^\vee \otimes \chi,$$
where
$\chi_0: \mathbb{A}_{F}^\times/{F}^\times \rightarrow
\mathbb{C}^\times$ is a character such that $\chi_0(-1_v)$ is
independent of $v$ for all $v \mid \infty$, and
$ \chi = \chi_0 \circ \mathrm{N}_{L/F} \circ \det$.  We note that in
the $CM$ case, the character $\chi_0$ is not unique because we can
multiply $\chi_0$ by the character associated to the quadratic
extension $L/F$ without changing $\chi$. This will however change the
sign of $\chi_0(-1_v)$.
\item \textbf{$C$-algebraic.} We say that  $\pi$ is $C$-algebraic if
  the infinitesimal character  $\lambda = ( ( \lambda_{1, \tau},
  \cdots, \lambda_{n, \tau})_{\tau \in J})$ of $\pi_\infty$ lies in
  $\big((\ZZ^n + \frac{n-1}{2} \ZZ^n)/
  \mathcal{S}_n\big)^{J}$.
\item \textbf{Weakly regular.} We say that a $C$-algebraic $\pi$ is
  weakly regular if for each $\tau \in J$, after applying a
  permutation in $\mathcal{S}_n$ to the indices, we have
  $ \lambda_{1, \tau}> \cdots> \lambda_{[n/2], \tau}$ and
  $ \lambda_{[n/2] +1, \tau} > \cdots > \lambda_{n, \tau}$. We say
  that a $C$-algebraic $\pi$ is regular if for each $\tau \in J$,
  after applying a permutation in $\mathcal{S}_n$ we have
  $ \lambda_{1, \tau}> \cdots> \lambda_{n, \tau}$.
\item \textbf{Odd.} We consider an oddness condition on a $C$-algebraic, essentially (conjugate) self dual $(\pi, \chi)$
  which is given by the existence of a pole at $s=1$ for a certain
  $L$-function.

If $L$ is $CM$, the oddness condition is   that  $L(s,  \mathrm{Asai}^{(-1)^{n-1}\epsilon(\chi_0)}(\pi)\otimes \chi_0^{-1})$ has a pole at $s=1$. 
 
 If $L=F$ is totally real, the oddness condition is   that $L(s, \Lambda^2 \pi \otimes \chi^{-1})$ has a pole at $s=1$ if  $\epsilon(\chi_0) = 1$ and $n$ is even, and that $L(s, \mathrm{Sym}^2 \pi \otimes \chi^{-1})$ has a pole at $s=1$ is $\epsilon(\chi_0) = - 1$ and $n$ is even or $\epsilon(\chi_0)=1$ and $n$ is odd. We will see below that $\epsilon(\chi_0) = 1$ when $n$ is odd and $L$ is totally real. 
\end{enumerate}

Let us explain the definition of the above $L$-functions and of the sign $\epsilon(\chi_0)$. We  first recall the definition of the Asai representation. Assume that $L$ is $CM$. Consider the group $(\mathrm{GL}_n(\C) \times \mathrm{GL}_n(\C)) \rtimes \mathrm{Gal}(L/F)$ with $\mathrm{Gal}(L/F)$ acting by permutation of the factors (in other words, this is  (a finite form of) the $L$-group over $F$  of $ \mathrm{Res}_{L/F} \mathrm{GL}_n$ ). 

The tensor product of the standard representations give a
representation $\C^n \otimes \C^n$ of
$\mathrm{GL}_n(\C) \times \mathrm{GL}_n(\C)$. For
$\epsilon \in \{-1,1\}$, it extends to representations
$\mathrm{Asai}^\epsilon : (\mathrm{GL}_n(\C) \times \mathrm{GL}_n(\C))
\rtimes \mathrm{Gal}(L/F) \rightarrow \mathrm{GL}(\C^n \otimes \C^n) $
by putting
$\mathrm{Asai}^\epsilon(c) (v\otimes w) = \epsilon w \otimes v$.

We now define the sign of $\chi_0$: It follows from the
$C$-algebraicity of $\pi$ that $\chi_0$ is algebraic. Therefore, there
is an integer $q$ such that $\chi_0 = \chi_0^f \vert\,  . \, \vert^q$ where $\chi_0^f$ is a finite order character. We let $\epsilon(\chi_0) = \chi_0 (-1_v) (-1)^q$ for any place $v \mid \infty$.  

We observe that when $n$ is odd and $L$ is totally real, $\epsilon(\chi_0)=1$. Indeed, if $c$ is the central character of $\pi$, we find that $c^2 = \chi_0^n$, and since $c$ is an algebraic Hecke character,   we deduce that $q$ is even and $\chi_0(-1_v)=1$. 

One may also verify that in the $CM$ case, the oddness condition only depends on $\chi$ and not on $\chi_0$. 

The following  shows that the oddness condition is often implied by the other assumptions: 
\begin{thm} Let $\pi$ be a weakly regular, algebraic, essentially (conjugate)
  self dual, cuspidal automorphic representation of
  $\mathrm{GL}_n/L$. Then $\pi$ is odd unless possibly when $n$
  is even and for all $\tau \in J$, there is an  ordering of the infinitesimal character $\lambda = ( ( \lambda_{1, \tau},
  \cdots, \lambda_{n, \tau})_{\tau \in J})$ of $\pi_\infty$ such that
  $\lambda_{i, \tau} = \lambda_{i +n/2, \tau}$ for all $1 \leq i \leq n/2$ and $\tau \in J$.
\end{thm}
\begin{demo} Lemma \ref{lem-odd-bc2}, reduces the proof of the theorem to the $CM$ case. The $CM$ case follows from 
Theorem \ref{theomok2}. 
\end{demo}

\begin{rem} Let $\pi_1$ and $\pi_2$ be two cuspidal automorphic
  representations of $\mathrm{GL}_2/\qq$, both $L$-algebraic and with trivial
   infinitesimal character. Assume that $(\pi_1)_\infty$ is
  a limit of discrete series (and therefore corresponds to a non-trivial parameter $W_{\mathbb{R}} \rightarrow \mathrm{GL}_2(\C)$ via the local Langlands correspondence), while $(\pi_2)_\infty$ corresponds to
  the trivial parameter
  $W_{\mathbb{R}} \rightarrow \mathrm{GL}_2(\C)$.  Then $\pi_1 \otimes \vert \det \vert^{\frac{1}{2}} $ is $C$-algebraic,  arises
  from a weight one modular form and is odd, while $\pi_2  \otimes \vert \det \vert^{\frac{1}{2}}$ arises from
  a Maass form of Galois type, and is not odd.
  Thus, $\pi_1$ is an  automorphic form which has a realization in the coherent cohomology of modular curves, while $\pi_2$ has no such realization. We observe that we cannot distinguish $(\pi_1)_\infty$ from $(\pi_2)_\infty$ by looking at the infinitesimal character, it is  the oddness property of $\pi_1$ or $\pi_2$ that distinguishes them. 
\end{rem}

\subsubsection{Oddness and base change} In this section we let $L$
be a $CM$ field and $F$ its maximal totally real subfield. If $\pi$ is
a cuspidal automorphic representation of $\mathrm{GL}_n/F$, we let
$Res_L(\pi)$ be the base change lift of $\pi$ to an automorphic
representation of $\mathrm{GL}_n/L$ (see \cite{MR1007299}).

\begin{proposition}\label{prop-odd-bc} Let $L$ be a $CM$ field and $F$  its maximal totally real subfield. Let $\chi_{L/F}$ be the associated quadratic
  character. Let $\pi$ be a weakly regular, algebraic, odd, essentially
  self dual cuspidal automorphic representation of
  $\mathrm{GL}_n/F$. Assume that $\pi \neq \pi \otimes
  \chi_{L/F}$. Then $Res_L(\pi)$ is a weakly regular, algebraic, odd,
  essentially conjugate self dual, cuspidal automorphic representation
  of $\mathrm{GL}_n/L$.
  \end{proposition}
\begin{demo} This follows from Lemma
  \ref{lem-odd-bc} and Lemma \ref{lem-odd-bc2}.
\end{demo}

\begin{lem}\label{lem-odd-bc}
  Let $L$ be a $CM$ field and $F$ be its maximal totally real
  subfield. Let $\chi_{L/F}$ be the associated quadratic
  character. Let $\pi$ be a weakly regular, algebraic, essentially self
  dual, cuspidal automorphic representation of
  $\mathrm{GL}_n/F$. Assume that $\pi \neq \pi \otimes
  \chi_{L/F}$. Then $Res_L(\pi)$ is a weakly regular, algebraic,
  essentially conjugate self dual, cuspidal automorphic representation
  of $\mathrm{GL}_n/L$.

\end{lem}
\begin{proof} Since $\pi \neq \pi \otimes \chi_{L/F}$, we deduce from \cite[Theorem 4.2]{MR1007299} that
  $Res_L(\pi)$ is cuspidal. Since $Res_L(\pi) = Res_L(\pi)^c$, we deduce that
  $Res_L(\pi)$ is essentially conjugate self dual. Let $v\mid \infty$
  be a place of $F$. Let $w$ be a place of $L$ above $v$. Then
  $Res_L(\pi)_w$ and $\pi_v$ have the same infinitesimal character which
  is $C$-algebraic and weakly regular. %The oddness property follows from Lemma
  %\ref{lem-odd-bc}
\end{proof}

\begin{lem}\label{lem-odd-bc2}     Let $\pi$ be a cuspidal automorphic 
  representation of $\mathrm{GL}_n/F$ satisfying the assumptions of Lemma \ref{lem-odd-bc}, with  base change $Res_L(\pi)$. Then $\pi$ is odd if and only if $Res_{L}(\pi)$ is
  odd.
\end{lem}  
\begin{demo} We have an $L$-group ``diagonal'' embedding
  $\mathrm{GL}_n(\C) \times \mathrm{Gal}(L/F) \hookrightarrow
  (\mathrm{GL}_n(\C) \times \mathrm{GL}_n(\C)) \rtimes
  \mathrm{Gal}(L/F)$, where  $\mathrm{GL_n}(\C) \times \mathrm{Gal}(L/F)$ is  (a finite form of) the $L$-group of $\mathrm{GL}_n/F$. We have the decomposition  $$\mathrm{Asai}^+\vert_{\mathrm{GL}_n(\C) \times \mathrm{Gal}(L/F)} = \Lambda^2\C^n \otimes \chi_{L/F} \bigoplus \mathrm{Sym}^2 \C^n$$ and
$$\mathrm{Asai}^-\vert_{\mathrm{GL}_n(\C) \times \mathrm{Gal}(L/F)} =
\Lambda^2\C^n  \bigoplus \mathrm{Sym}^2 \C^n \otimes \chi_{L/F} .$$
It follows that
$L(s, \mathrm{Asai}^{(-1)^{n-1} \epsilon(\chi_0)}(Res_{L}(\pi)) \otimes
\chi^{-1}_0) = $
$$L(s, \Lambda^2 \pi \otimes \chi_{L/F}^{\epsilon(\chi_0)(-1)^n}
\otimes \chi_0^{-1}) L(s, \mathrm{Sym}^2 \pi \otimes
\chi_{L/F}^{\epsilon(\chi_0)(-1)^{n+1}} \otimes \chi_0^{-1}).$$ By
\cite[Theorem 4.1]{MR1610812}, neither of the $L$-functions
$L(s, \Lambda^2 \pi \otimes \chi_{L/K}^{\epsilon(\chi_0)(-1)^n}
\otimes \chi_0^{-1})$ and
$L(s, \mathrm{Sym}^2 \pi \otimes
\chi_{L/F}^{\epsilon(\chi_0)(-1)^{n+1}} \otimes \chi^{-1}_0)$ vanishes
at $s=1$. We therefore deduce that $Res_L(\pi)$ is odd if 
$\pi$ is odd. Conversely, if we assume that $Res_L(\pi)$ is odd, then at least one of $L(s, \Lambda^2 \pi \otimes \chi_{L/K}^{\epsilon(\chi_0)(-1)^n}
\otimes \chi_0^{-1})$ or $L(s, \mathrm{Sym}^2 \pi \otimes
\chi_{L/F}^{\epsilon(\chi_0)(-1)^{n+1}} \otimes \chi^{-1}_0)$ has a pole at $s=1$. But notice that $L(s, \pi \otimes \pi \otimes \chi_{L/K} \otimes \chi^{-1}_0)$ is holomorphic at $s=1$ since $\pi \neq \pi \otimes \chi_{L/K}$. It is now easy to deduce that $\pi$ is odd (by examining the parity of $n$ and the sign $\epsilon(\chi_0)$). 
\end{demo} 

\subsubsection{Descent to a unitary group} In this section we consider
a $CM$ field $L$. We say that $\pi$ on
$\mathrm{GL}_n/L$ is conjugate self dual (as opposed to essentially
conjugate self dual) if $\pi^c = \pi^\vee$.  It is not hard to prove
that if $\pi$ is essentially conjugate self dual, there exists an
algebraic Hecke character
$\psi : \mathbf{A}_L^\times/L^\times \rightarrow \C^\times$ such that
$\pi \otimes \psi$ is conjugate self dual.

\begin{thm}\label{theomok} Let $L$ be a $CM$-field and $\pi$  a weakly
  regular, algebraic,  odd, conjugate self dual, cuspidal automorphic representation of $\mathrm{GL}_n/L$. Then there exists a cuspidal
  automorphic representation $\tilde{\pi}$ of the quasi-split unitary
  group $\mathrm{U}(n)/F$ such that $\pi$ is the transfer of $\tilde{\pi}$ for
  the standard base change embedding and $\tilde{\pi}_\infty$ is a non
  degenerate limit of discrete series. 
\end{thm}

\begin{proof} %After twisting $\pi$ with a character, we can assume that $\pi^\vee = \pi^c$. 
This follows from the main results of \cite{MR3338302}. We give some explanations. 

Let $\mathrm{U}(n)/F$ be the quasi-split unitary group. Its $L$-group
$~^L\mathrm{U}(n)/F$ (over $F$) is isomorphic to  $\mathrm{GL}_n(\C) \rtimes
\mathrm{Gal}(L/F)$ with the complex conjugation $c$ acting by $g
\mapsto \Phi_n ~^t g^{-1} \Phi_n^{-1}$ where $\Phi_n$ is the
anti-diagonal matrix with alternating $1$ and $-1$ on the anti-diagonal. 

Let $v$ be a place of $F$ and let $W_{F_v}$ be the Weil group of
$F_v$.  First assume that $v$ does not split in $L$. The local
$L$-group $~^L\mathrm{U}(n)/F_v$ is isomorphic to
$\mathrm{GL}_n(\C) \rtimes W_{F_v}$ where an element of
$W_{F_v} \setminus W_{L_v}$ acts via
$g \mapsto \Phi_n ~^t g^{-1} \Phi_n^{-1}$. If $v$ splits, the local
$L$-group $~^L\mathrm{U}(n)/F_v$ is isomorphic to
$\mathrm{GL}_n(\C) \times W_{F_v}$.

 Let us denote by $G= \mathrm{Res}_{L/F} \mathrm{GL}_n$.  The $L$-group of $G$ is isomorphic to  $(\mathrm{GL}_n(\C) \times \mathrm{GL}_n(\C)) \rtimes \mathrm{Gal}(L/F)$ with the complex conjugation acting by permuting the factors.  We similarly have local $L$-groups
$~^LG/F_v = (\mathrm{GL}_n(\C) \times \mathrm{GL}_n(\C)) \rtimes W_{F_v}$  if $v$ does not split and $~^LG/F_v  = (\mathrm{GL}_n(\C) \times \mathrm{GL}_n(\C)) \times W_{F_v}$  if $v$ splits. 

We have the standard base change embedding of $L$-groups $\xi: ~^L\mathrm{U}(n) \rightarrow ~^LG$ which sends $g \rtimes 1$ to  $(g, ~^tg^{-1}) \rtimes 1$ and $1 \rtimes c$ to $( \Phi_n, \Phi_n^{-1}) \rtimes c$.  We have similar local versions  $\xi_v : ~^L\mathrm{U}(n)/F_v \rightarrow ~^LG/F_v $ for any place $v$ of $F$.

Let $v$  be a place of $F$. Denote by $L_{F_v}$ the  local Langlands group of $F_v$.  We let $\Phi(n)_v$ be the set of isomorphism classes of parameters $L_{F_v} \rightarrow ~^LG/F_v(\C)$ and we let $\Phi(\mathrm{U}(n))_v$ be the set of isomorphism classes of parameters $ L_{F_v} \rightarrow ~^L \mathrm{U}(n)/F_v(\C)$. 

In the case that $v$ does not split in $L$, we recall how one can
identify parameters $L_{L_v} \rightarrow \mathrm{GL}_n(\C)$ with
parameters $L_{F_v} \rightarrow ~^LG/F_v(\C)$. We choose an element
$w_c \in L_{F_v} \setminus L_{L_v}$.   To $\rho: L_{L_v} \rightarrow
GL_n(\C)$, we associate the parameter $\rho': L_{F_v} \rightarrow
~^LG/F_v(\C)$ defined by $\rho'(\sigma) = ( \rho(\sigma), \rho( w^{-1}_c
\sigma w_c)) \rtimes 1$ if $\sigma \in L_{L_v}$  and $\rho'(w_c) = ( \rho(w_c^2),1) \rtimes c$.

By \cite[Lemma 2.2.1]{MR3338302}, if $v$ does not split, the natural map  $ \xi_v^\star : \Phi(\mathrm{U}(n))_v \rightarrow \Phi(n)_v$, given by  $\phi \mapsto \xi_v \circ \phi$, induces a bijection between parameters $ \phi: L_{F_v} \rightarrow ~^L\mathrm{U}(n)/F_v(\C)$,   and parameters $\rho: L_{L_v} \rightarrow GL_n(\C)$ for which there exists an invertible matrix $A$ with:

\begin{enumerate}
\item $~^t \rho^c A \rho = A$, where $\rho^c (\sigma) = \rho (w^{-1}_c \sigma w_c)$. 
\item $~^tA = (-1)^{n-1} A \rho(w_c^2)$.  
\end{enumerate}

Indeed, to such a parameter $\rho$ we associate the parameter $\phi$
defined by  $\phi(\sigma) = \rho(\sigma)\rtimes 1$ if $\sigma \in
L_{L_v}$  and 
$\phi(w_c) = C\rtimes c$ where $C = A^{-1} \Phi_n$. 

Now let us consider a place $v$ of $F$ that splits in $L$.  Let $w$
and $\bar{w}$ be the two places in $L$ above $v$.  We have canonical
isomorphisms $L_{L_w} = L_{F_v}$ and $L_{L_{\bar{w}}} = L_{F_v}$.  A
parameter $\rho: L_{F_v} \rightarrow ~^LG/F_v(\C) $ can be written as
$\rho(\sigma) = (\rho_w(\sigma), \rho_{\bar{w}}(\sigma)) \times
\sigma$ for all $\sigma \in L_{F_v}$, where
$\rho_w : L_{L_w} \rightarrow \mathrm{GL}_n(\C)$ and
$\rho_{\bar{w}} : L_{L_{\bar{w}}} \rightarrow \mathrm{GL}_n(\C)$.

In this case,  the natural map  $\xi_v^\star : \Phi(\mathrm{U}(n))_v \rightarrow \Phi(n)_v$, given by  $\phi \mapsto \xi_v \circ \phi$, induces a bijection between parameters $ \phi: L_{F_v} \rightarrow ~^L \mathrm{U}(n)/F_v(\C)$,   and parameters $\rho$ satisfying $\rho_w \simeq \rho_{\bar{w}}^{\vee}$.

For any place $v$ of $F$ there is associated to $\pi_v$ a local
Langlands parameter $\phi_{\pi_v} : L_{F_v} \rightarrow ~^L G/F_v(\C)$
(\cite{ht}, \cite{MR1738446}). Since $\pi$ is conjugate self dual and
odd, it follows from \cite[Theorems 2.4.10, 2.5.4]{MR3338302} that
this parameter arises from a unique parameter
$\tilde{\phi}_{\pi_v} : L_{F_v} \rightarrow ~^L\mathrm{U}(n)/F_v(\C)$. By
\cite[Theorem 2.5.1]{MR3338302}, there is a local packet $\Pi_v$
associated to $\tilde{\phi}_{\pi_v}$.

By \cite[Theorem 2.5.1]{MR3338302}, since $\pi$ is cuspidal automorphic, any representation $\tilde{\pi} = \otimes_v \tilde{\pi}_v$ with $\tilde{\pi}_v \in \Pi_v$ is cuspidal automorphic on $\mathrm{U}(n)/F$.

Let us describe in more details the parameter $\tilde{\phi}_{\pi_v}$ at a place $v \mid \infty$. 
 Associated to $\pi_v$ there is a parameter 
$$\rho_v :  L_{L_v} =  \C^\times \rightarrow \mathrm{GL}_n(\C)$$ which is conjugated to $$z \mapsto \mathrm{diag} ( (z/\bar{z})^{\lambda_{1,v}}, \cdots, (z/\bar{z})^{\lambda_{n,v}}).$$ This parameter is conjugate self dual with respect to the standard orthogonal form given by $A = \mathrm{Id}$. %Moreover, unless $n$ is even and $\lambda_{i, v} = \lambda_{i +[n/2], v}$, it does not respect  any non degenerate symplectic form (and therefore determines the sign of the transfer). 

Using the recipe of  \cite[Lemma 2.2.1]{MR3338302} (observe that $(-1)^{n-1}\rho_v(w_c^2) = 1$), we deduce that the parameter $\tilde{\phi}_{\pi_v}$ corresponding to $\Pi_v$  is the non degenerate limit of discrete
series parameter given by:

$$ L_{F_v} = W_{\mathbb{R}} \rightarrow ~^L\mathrm{U}(n)/\mathbb{R}(\C)$$
with $\phi(z) = \mathrm{diag} ( (z/\bar{z})^{\lambda_{1,v}}, \cdots, (z/\bar{z})^{\lambda_{n,v}}) \rtimes 1$ and $\phi(j) = \Phi_n \rtimes c$ where $j \in W_{\mathbb{R}} \setminus \C^\times$ satisfies $j^2=-1$. 

\end{proof}

\begin{thm}\label{theomok2} Let $\pi$ be a weakly regular, algebraic, conjugate self dual,  cuspidal automorphic representatioin
  $\pi$ of $\mathrm{GL}_n/L$.  Let $\lambda = (\lambda_{i,\tau})$ be
  its infinitesimal character. Then $\pi$ is automatically odd unless
  possibly when $n $ is even and for all $\tau$,
  $\lambda_{i, \tau} = \lambda_{i +n/2, \tau}$ for some ordering of
  the infinitesimal character.
  \end{thm}
  \begin{demo} This again follows from the results of \cite{MR3338302}. If $\pi$ is not odd, we deduce that $L(s, \mathrm{Asai}^{(-1)^{n}} (\pi))$ has a pole at $s=1$. Let us fix a  character $\chi_{-} : \mathbb{A}_L^\times/L^\times \rightarrow \C^\times$ verifying: $\chi_{-}^c = \chi_{-}^{-1}$ and $\chi_{-}\vert \mathbb{A}_F^\times $ corresponds to the quadratic character of $L/F$. 
  We can define a twisted $L$-group embedding $\xi_{-}: ~^L\mathrm{U}(n) \rightarrow ~^LG$ (for the Weil group form of the $L$-groups) which sends $g \rtimes 1$ to  $(g, ~^tg^{-1}) \rtimes 1$, $1 \rtimes \sigma$ to $(\chi_{-}(\sigma), \chi_{-}^{-1}(\sigma))\rtimes \sigma$ if $\sigma \in W_{F}$, and  $1 \rtimes c$ to $( -\Phi_n, \Phi_n^{-1}) \rtimes w_c$ for $w_c \in W_{L} \setminus W_{F}$.  A similar argument as in the proof of theorem \ref{theomok} shows that $\pi$ descends via the twisted $L$-group embedding to an automorphic representation of $\mathrm{U}(n)$. %In particular, $\tilde{\pi}_\infty$ would descend by the twisted $L$-group embedding.  
  For each place $v \mid \infty$ the parameter of $\pi_v$,  $ \rho_v : L_{L_v} =  \C^\times \rightarrow \mathrm{GL}_n(\C)$  is conjugated to $z \mapsto \mathrm{diag} ( (z/\bar{z})^{\lambda_{1,v}}, \cdots, (z/\bar{z})^{\lambda_{n,v}})$. For $\pi_v$ to descend via the twisted $L$-group embedding,   there should exist a non degenerate symplectic form $A$ on  $\C^n$ such that  $^t\rho_v(\bar{z}) A \rho_v(z) = A$ for all $z \in \C^\times$. It is easy to see that there is no such symplectic form, unless when $n$ is even and $\lambda_{i, v} = \lambda_{i +n/2, v}$ for some ordering of the $(\lambda_{i,v})$. 
  \end{demo}

\subsection{Automorphic Galois representations} In this section we let
again $L$ be a $CM$ or totally real field.   We first recall the
following  result concerning Galois representations attached to
regular, essentially (conjugate) self dual, cuspidal automorphic
representations of $\mathrm{GL}_n/L$ (see, e.g., \cite{MR3272052}, \cite{BLGGT}).
\begin{thm}[Bellaiche, Caraiani, Chenevier, Clozel, Harris, Kottwitz,
  Labesse, Shin, Taylor, \ldots]\label{bigth}   Let $\pi$ be a  regular, algebraic, (essentially) conjugate self dual  cuspidal automorphic representation of $\mathrm{GL}_n/L$.   In particular  $\pi^c = \pi^\vee \otimes \chi$ and the infintesimal character of $\pi$ is $\lambda = ( ( \lambda_{1, \tau}, \cdots, \lambda_{n, \tau})_{\tau \in J})$ with   $\lambda_{1, \tau} > \cdots > \lambda_{n, \tau}$. 
  Let $\iota: \C \simeq \overline{\qq}_p$.  There is a continuous
  Galois representation
  $\rho_{\pi, \iota}: G_L \rightarrow \mathrm{GL}_n(\overline{\qq}_p)$
  such that:
  \begin{enumerate}
  \item $\rho_{\pi, \iota}^c \simeq \rho_{\pi, \iota}^\vee \otimes \epsilon_p^{1-n} \otimes \chi_{\iota}$ where $\chi_{\iota}$ is the $p$-adic realization of $\chi$ and $\epsilon_p$ is the cyclotomic character,
  \item $\rho_{\pi, \iota}$ is pure,
\item $\rho_{\pi, \iota}$ is de Rham at all places dividing $p$, with  $\iota^{-1} \circ \tau$-Hodge-Tate weights: $(-\lambda_{n,\tau} + \frac{n-1}{2}, \cdots, -\lambda_{1,\tau} + \frac{n-1}{2})$, 
 \item  for all finite  place 
$v$ one has:
$$ WD (\rho_{\pi, \iota} \vert_{G_{F_v}})^{F-ss} = \mathrm{rec} (\pi_v \otimes \vert \det \vert_v^{\frac{1-n}{2}}).$$
\end{enumerate}
 \end{thm}
 
 %\begin{rem} If $\lambda$ gives rise to  a limit of discrete series  $\lambda_{[n/2],\tau} = \lambda_{[n/2] +1, \tau}$, one still knows that the  $\tau$-Hodge-Tate weights are in the form $(\lambda_{1,\tau} - \frac{1-n}{2}, \cdots, \lambda_{n,\tau} - \frac{1-n}{2})$.  
% \end{rem}
 
 \begin{rem}
 Our convention is that  the reciprocity law is normalized by sending
 geometric Frobenius  to a uniformizing element. Moreover, the cyclotomic character has Hodge-Tate weight $-1$. 
 \end{rem}

 In the situation that regular is replaced by the weaker assumption of being weakly regular and odd, we have the following  weaker theorem: 
 
 \begin{thm}\label{smallth}   Let $\pi$ be a  weakly regular, algebraic, odd, (essentially) conjugate self dual, cuspidal automorphic representation of $\mathrm{GL}_n/L$.   In particular,  $\pi^c = \pi^\vee \otimes \chi$.  There is a continuous Galois representation  $\rho_{\pi, \iota}:  G_L \rightarrow \mathrm{GL}_n(\overline{\qq}_p)$ such that: 
  \begin{enumerate}
  \item $\rho_{\pi, \iota}^c \simeq \rho^\vee \otimes \epsilon_p^{1-n} \otimes \chi_{\iota}$ where $\chi_{\iota}$ is the $p$-adic realization of $\chi$,
 \item $\rho_{\pi, \iota}$ is unramified at all finite places $v \nmid p$ for which $\pi_v$ is unramified and one has:
$$ WD (\rho_{\pi, \iota} \vert_{G_{F_v}})^{F-ss} = \mathrm{rec} (\pi_v \otimes \vert \det \vert_v^{\frac{1-n}{2}}).$$
\end{enumerate}
 \end{thm}
 
 \begin{proof} By the patching technique of \cite{Sorensen}, we may
   reduce to the case of a $CM$ field. Using Theorem \ref{theomok},
   there is a $\tilde{\pi}$ on the quasi-split unitary group
   $\mathrm{U}(n)/F$, which transfers to $\pi$. Moreover,
   $\tilde{\pi}_\infty$ is a non degenerate limit of discrete series
   and therefore realizes in the coherent cohomology of a unitary
   Shimura variety. We can then apply the main result of
   \cite{MR3512528} or \cite{MR3989256} to conclude. The point is that
   the Hecke eigensystem of $\tilde{\pi}$ is a $p$-adic limit of Hecke
   eigensystems of regular, essentially conjugate self dual, automorphic
   representations to which Theorem \ref{bigth} applies.
 \end{proof}

 One  conjectures that $\rho_{\pi, \iota}$ in Theorem \ref{smallth} satisfies the stronger properties of Theorem \ref{bigth}. In particular, it should be  de Rham with
 Hodge--Tate weights $(-\lambda_{n,\tau} + \frac{n-1}{2}, \cdots,
 -\lambda_{1,\tau} + \frac{n-1}{2})$, for $\lambda = ( ( \lambda_{1, \tau}, \cdots, \lambda_{n, \tau})_{\tau \in J})$ the infinitesimal character of $\pi_\infty$,  and  $ WD (\rho_{\pi, \iota}
 \vert_{G_{F_v}})^{F-ss} = \mathrm{rec} (\pi_v \otimes \vert \det
 \vert_v^{\frac{1-n}{2}})$ for all finite place. This has been verified in some special cases (e.g., for weight $1$ modular forms). 
 
 We can prove the following very weak instance of local-global
 compatibility at places dividing $p$:

 \begin{thm} \label{odd}
   Let $\pi$ be a weakly regular, algebraic, odd, essentially (conjugate) self dual, cuspidal automorphic representation of $\mathrm{GL}_n/L$ with  infinitesimal character $\lambda = (\lambda_{i, \tau}, 1\leq i \leq n, \tau \in \mathrm{Hom} (L, \overline{\qq}))$ and $\lambda_{1,\tau} \geq \cdots \geq \lambda_{n,\tau}$. Let $p$ be a prime unramified in $L$ and  let $w\mid p$ be a finite place  in $L$. Assume also that $\pi_w$ is spherical, and corresponds to a semi-simple conjugacy class $\mathrm{diag}(a_1,\cdots, a_n) \in \mathrm{GL}_n(\overline{\qq})$ by the Satake isomorphism. We let $\iota: \overline{\qq} \rightarrow \overline{\qq}_p$ be an embedding and $v$ the associated $p$-adic valuation normalized by $v(p)=1$.  After permuting we assume that $v(a_1) \leq \cdots \leq v(a_n)$. Let $I_w \subset \mathrm{Hom} (L, \overline{\qq})$ be the set of embeddings $\tau$ such that $\iota \circ \tau$ induces the $w$-adic valuation on $L$. 
 Then we have $$\sum_{i=1}^k v(a_i) \geq \sum_{\tau \in I_w}
 \sum_{\ell=1}^k - \lambda_{\ell, \tau}$$ for $1 \leq k \leq n$, with equality if $k= n$. 
 \end{thm}

 \begin{demo} We first reduce to the $CM$ case because if $L$ is
   totally real, we can consider a $CM$ quadratic extension $L'$ of
   $L$ such that $w$ splits in $L'$ and $Res_{L'}(\pi)$ is cuspidal
   (for example, it suffices to take $L'$ to be ramified over $L$ at
   some finite place $v_0$ of $L$ where $\pi_{v_0}$ is
   unramified). Now assume $L$ is $CM$ and let $F$ be its maximal totally real
   subfield. We next explain how one can reduce to the case where all
   primes $v \mid p$ in $F$ split in $L$.  We can construct a totally
   real quadratic extension $F'$ of $F$ with the property that $p$ is
   unramified in $F'$ and  for all places $w \mid p$ of $F$, $w$
   splits in $F'$ if and only if $w$ splits in $L$. We may also impose
   the additional condition that $F'$ ramifies at some place $v_0$ of
   $F$ which is inert in $L$ and such that $\pi_{v_0}$ is
   unramified. We now set $L' = F'L$. It is clearly sufficient to
   prove the result for the base change of $\pi$ to $L'$. Note that
   this base change is still weakly regular, algebraic, odd, essentially
   conjugate self dual, cuspidal. The cuspidality follows
   from our choice of $L'$, it being ramified at $v_0$.  The oddness is
   another application of Shahidi's theorem that
   $L(s, \mathrm{Asai}^{(-1)^{n-1} \epsilon(\chi_0)}(\pi) \otimes
   \chi^{-1}_0 \otimes \chi_{L'/L}) $ does not vanish at $s=1$.

We  have thus reduced to the situation that $L$ is a $CM$ field and
all primes  $v \mid p$ in $F$ split in $L$. It is useful to make the further restriction that there exists a quadratic imaginary extension $L_0$ of $\qq$   such that $L = L_0F$. We can achieve this by considering a quadratic imaginary extension $L_0$ of $\qq$ such that $p$ splits in $L_0$ and ramifies at a prime $q$ such that there is a place $v_0$ of $L$ above $q$ which is unramified over $\qq$ and such that $\pi_{v_0}$ is spherical. We may now replace $L$ by $LL_0$ which is a $CM$ field containing a quadratic imaginary field and $\pi$ by $Res_{LL_0}(\pi)$. 

By Theorem \ref{theomok}, there is a cuspidal automorphic
representation $\tilde{\pi}$ of $\mathrm{U}(n)/F$ which transfers to
$\pi$ for the standard base change embedding and is a limit of
discrete series at infinity (with infinitesimal character given by
$\lambda$, see the end of the proof of Theorem \ref{theomok}). The
result now essentially follows from Corollary \ref{coro-conj2-unit},
but there is a subtlety in that Corollary \ref{coro-conj2-unit}
applies to the group $\mathrm{GU}(n)$ and not $\mathrm{U}(n)$. We will
reduce to this case.

Let $\tilde{Z}$ be the center of $\mathrm{GU}(n)$. We observe that
$\tilde{Z} \times \mathrm{U}(n) \rightarrow \mathrm{GU}(n)$ is a
surjective map of algebraic groups ($\mathrm{U}(n)$ is viewed as  an algebraic group over
$\qq$ by Weil restriction). Let $c$ be the central character of
$\tilde{\pi}$. Let $c'$ be an extension of $c$ to an algebraic
automorphic character of $Z$ which is unramified at all places
dividing $p$.  We claim that the cuspidal automorphic representation
$\tilde{\pi}$ admits an ``extension'' to a cuspidal automorphic
representation $\tilde{\pi}'$ of $ \mathrm{GU}(n)$ with central
character $c'$. We explain the meaning of this ``extension''. By
definition $\tilde{\pi}'$ will realize in the space of cusp forms with
central character $c'$ on $\mathrm{GU}(n)$, say
$\mathfrak{A}_{cusp}(\mathrm{GU}(n))_{c'}$. There is a well defined
restriction map
$res : \mathfrak{A}_{cusp}(\mathrm{GU}(n))_{c'} \rightarrow
\mathfrak{A}_{cusp}(\mathrm{U}(n))_{c}$ and we say that $\tilde{\pi}'$
extends $\tilde{\pi}$ if $\tilde{\pi}$ is a constituent of
$res (\tilde{\pi}')$. Observe that the very general result
\cite[Theorem 4.14]{MR2918491} implies that ${\pi}$ admits an
extension for some choice of ${c}'$. We will prove the slightly
stronger result that in our case, for any choice of ${c}'$ lifting
$c$, the map $res$ is surjective.  Our argument follows
\cite{Chenevier} which proves a similar result under slightly
different hypotheses.  We take
$\phi \in \mathfrak{A}_{cusp}(\mathrm{U}(n))_{c}$. One first extends
$\phi$ to a function $\phi_1$ on
$\tilde{Z}(\mathbb{A}_\qq) \mathrm{U}(n)(\mathbb{A}_\qq) \subset
\mathrm{GU}(n)(\mathbb{A}_\qq) $ satisfying
$\phi_1(zg) = c'(z) \phi(g)$ for all
$(z,g) \in \tilde{Z}(\mathbb{A}_\qq)\times
\mathrm{U}(n)(\mathbb{A}_\qq)$. We then extend $\phi_1$ to a function
$\phi_2$ on
$\mathrm{GU}(n)(\qq)\tilde{Z}(\mathbb{A}_\qq)
\mathrm{U}(n)(\mathbb{A}_\qq) \subset \mathrm{GU}(n)(\mathbb{A}_\qq) $
by letting $\phi_2(\gamma g) = \phi_1(g)$ for all
$(\gamma, g) \in \mathrm{GU}(n)(\qq) \times \tilde{Z}(\mathbb{A}_\qq)
\mathrm{U}(n)(\mathbb{A}_\qq) $.  To check that $\phi_2$ is well
defined, it suffices to prove that
$\mathrm{GU}(n)(\qq) \cap \tilde{Z}(\mathbb{A}_\qq)
\mathrm{U}(n)(\mathbb{A}_\qq) \subseteq \tilde{Z}(\qq)
\mathrm{U}(n)(\qq)$. This actually amounts to proving that
$\qq^\times \cap \nu(\tilde{Z}(\mathbb{A}_\qq)) \subseteq
\nu(\tilde{Z}(\qq))$, which follows from Hasse's theorem that in a
quadratic extension any local norm is a global norm.  Finally, we
claim that
$\mathrm{GU}(n)(\qq)\tilde{Z}(\mathbb{A}_\qq)
\mathrm{U}(n)(\mathbb{A}_\qq)$ is of finite index in
$\mathrm{GU}(n)(\mathbb{A}_\qq)$ (the quotient is dominated by
$(\qq^\times \mathrm{N}_{L_0/\qq} \mathbb{A}_{L_0}^\times) \backslash
\mathbb{A}_\qq^\times$). We can therefore extend $\phi_2$ by zero to a
function $\phi_3$ defined on $\mathrm{GU}(n)(\mathbb{A}_\qq)$. The
verification that
$\phi_3 \in \mathfrak{A}_{cusp}(\mathrm{GU}(n))_{c'}$ is made in
\cite{Chenevier}. By construction, $res(\phi_3) = \phi$.

 Since all places $v\mid p$ in $F$
split in $L$, we have that
$\mathrm{GU}(n)(\qq_p) = \mathrm{U}(n)(\qq_p) Z(\qq_p)$, so we deduce
that $\tilde{\pi}'_v$ is spherical at all places $v \mid p$. Moreover,
the Satake parameters of $\tilde{\pi}'_v$ and $\tilde{\pi}_v$ are
related via the map on dual groups
$\widehat{\mathrm{GU}(n)} = \mathbb{G}_m \times
\mathrm{GL}_n^{[F:\qq]} \rightarrow \widehat{\mathrm{U}(n)} =
\mathrm{GL}_n^{[F:\qq]}$ (which forgets the $\mathbb{G}_m$). There is
a similar story at archimedean places and we deduce that
$\tilde{\pi}'_\infty$ is a limit of discrete series.  We may apply
Corollary \ref{coro-conj2-unit} to $\tilde{\pi}'$, and we deduce the
validity of Conjecture \ref{conj1} in this case, which is the
inequality:
 $$ \mathrm{Newt}_{\iota}( \chi) \leq  \frac{1}{\vert
  \Gamma/\mathrm{Stab}_\Gamma(\infty(\kappa, \iota))\vert}
\sum_{\gamma \in \Gamma/\mathrm{Stab}_\Gamma(\infty(\kappa, \iota))}
-w_0(\gamma \cdot\infty(\kappa, \iota)).$$

We project this identity under the map $\widehat{\mathrm{GU}(n)} \rightarrow \widehat{\mathrm{U}(n)}$ and then further via the map $\mathrm{GL}_n^{[F:\qq]} \rightarrow \mathrm{GL}_n^{[F_v:\qq_p]}$ for a chosen prime $v \in F$ below the prime $w \in L$. We can now unravel the meaning of this identity. 

The local $L$ group of $\mathrm{Res}_{F_v/\qq_p} \mathrm{GL}_n$ is $\mathrm{GL}_n^{[F_v:\qq_p]} \rtimes \Gamma$ where $\Gamma = \hat{\ZZ}$ acts by permutation of the factors. 

The Satake parameter of $\pi_w = \tilde{\pi}_v$ is given by the conjugacy class of $\mathrm{diag}(a_1,\cdots, a_n)$, and via the Newton map it goes to $$ \frac{1}{[F_v:\qq_p]}( v(a_n), \cdots, v(a_1))^{[F_v:\qq_p]}$$
 in $(P_{\R}^+)^{\Gamma}$ where $(P_{\R}^+)^{\Gamma}$ is the subset of $\Gamma$-invariants in the cone of dominant weights $P_\R^+ \subset (\R^{n})^{[F_v:\qq_p]}$.  
 The projection of the infinitesimal character is $(\lambda_{i, \tau})_{\tau \in I_w}$ and the average  $\sum_{\gamma \in \Gamma/\mathrm{Stab}_\Gamma(\infty(\kappa, \iota))}
-w_0(\gamma \cdot\infty(\kappa, \iota))$
gives:  $$\frac{1}{[F_v:\qq_p]}\Bigl( \sum_{\tau \in I_w} -\lambda_{n,
  \tau}, \cdots, \sum_{\tau \in I_w}- \lambda_{1,\tau}\Bigr )^{[F_v:\qq_p]}.$$
The theorem is proven. 

\end{demo}

\renewcommand{\MR}[1]{}
\providecommand{\bysame}{\leavevmode\hbox to3em{\hrulefill}\thinspace}
\providecommand{\MR}{\relax\ifhmode\unskip\space\fi MR }
% \MRhref is called by the amsart/book/proc definition of \MR.
\providecommand{\MRhref}[2]{%
  \href{http://www.ams.org/mathscinet-getitem?mr=#1}{#2}
}
\providecommand{\href}[2]{#2}

%\bibliographystyle{amsalpha}
%\bibliography{heckeoperators}

\begin{thebibliography}{BLGGT14}

\bibitem[AC89]{MR1007299}
James Arthur and Laurent Clozel, \emph{Simple algebras, base change, and the
  advanced theory of the trace formula}, Annals of Mathematics Studies, vol.
  120, Princeton University Press, Princeton, NJ, 1989. \MR{MR1007299
  (90m:22041)}

\bibitem[AMRT10]{AMRT}
Avner Ash, David Mumford, Michael Rapoport, and Yung-Sheng Tai, \emph{Smooth
  compactifications of locally symmetric varieties}, second ed., Cambridge
  Mathematical Library, Cambridge University Press, Cambridge, 2010, With the
  collaboration of Peter Scholze. \MR{2590897}

\bibitem[BCGP18]{BCGP}
George Boxer, Frank Calegari, Toby Gee, and Vincent Pilloni, \emph{Abelian
  surfaces over totally real fields are potentially modular}, preprint, 2018.

\bibitem[BG14]{MR3444225}
Kevin Buzzard and Toby Gee, \emph{The conjectural connections between
  automorphic representations and {G}alois representations}, Automorphic forms
  and {G}alois representations. {V}ol. 1, London Math. Soc. Lecture Note Ser.,
  vol. 414, Cambridge Univ. Press, Cambridge, 2014, pp.~135--187. \MR{3444225}

\bibitem[BH98]{bruns1998cohen}
W.~Bruns and H.J. Herzog, \emph{Cohen--{M}acaulay rings}, Cambridge Studies in
  Advanced Mathematics, Cambridge University Press, 1998.

\bibitem[BLGGT14]{BLGGT}
Thomas Barnet-Lamb, Toby Gee, David Geraghty, and Richard Taylor,
  \emph{Potential automorphy and change of weight}, Ann. of Math. (2)
  \textbf{179} (2014), no.~2, 501--609. \MR{3152941}

\bibitem[Bor79]{MR546608}
A.~Borel, \emph{Automorphic {$L$}-functions}, Automorphic forms,
  representations and $L$-functions (Proc. Sympos. Pure Math., Oregon State
  Univ., Corvallis, Ore., 1977), Part 2, Proc. Sympos. Pure Math., XXXIII,
  Amer. Math. Soc., Providence, R.I., 1979, pp.~27--61. \MR{MR546608
  (81m:10056)}

\bibitem[BT72]{bruhat-tits}
F.~Bruhat and J.~Tits, \emph{Groupes r\'{e}ductifs sur un corps local}, Inst.
  Hautes \'{E}tudes Sci. Publ. Math. (1972), no.~41, 5--251. \MR{0327923}

\bibitem[Car79]{cartier}
P.~Cartier, \emph{Representations of {$p$}-adic groups: a survey}, Automorphic
  forms, representations and {$L$}-functions ({P}roc. {S}ympos. {P}ure {M}ath.,
  {O}regon {S}tate {U}niv., {C}orvallis, {O}re., 1977), {P}art 1, Proc. Sympos.
  Pure Math., XXXIII, Amer. Math. Soc., Providence, R.I., 1979, pp.~111--155.
  \MR{546593}

\bibitem[CH13]{MR3272052}
Ga\"{e}tan Chenevier and Michael Harris, \emph{Construction of automorphic
  {G}alois representations, {II}}, Camb. J. Math. \textbf{1} (2013), no.~1,
  53--73. \MR{3272052}

\bibitem[Che18]{Chenevier}
Gaetan Chenevier, \emph{On restrictions and extensions of cusp forms},
  preprint, 2018.

\bibitem[Clo90]{MR1044819}
Laurent Clozel, \emph{Motifs et formes automorphes: applications du principe de
  fonctorialit\'e}, Automorphic forms, Shimura varieties, and $L$-functions,
  Vol.\ I (Ann Arbor, MI, 1988), Perspect. Math., vol.~10, Academic Press,
  Boston, MA, 1990, pp.~77--159. \MR{MR1044819 (91k:11042)}

\bibitem[CN92]{MR1144439}
Ching-Li Chai and Peter Norman, \emph{Singularities of the
  {$\Gamma_0(p)$}-level structure}, J. Algebraic Geom. \textbf{1} (1992),
  no.~2, 251--278. \MR{1144439}

\bibitem[Con00]{conrad-bc}
Brian Conrad, \emph{Grothendieck duality and base change}, Lecture Notes in
  Mathematics, vol. 1750, Springer-Verlag, Berlin, 2000. \MR{1804902}

\bibitem[Con07]{MR2311664}
\bysame, \emph{Arithmetic moduli of generalized elliptic curves}, J. Inst.
  Math. Jussieu \textbf{6} (2007), no.~2, 209--278. \MR{2311664}

\bibitem[Del79]{MR546620}
Pierre Deligne, \emph{Vari\'et\'es de {S}himura: interpr\'etation modulaire, et
  techniques de construction de mod\`eles canoniques}, Automorphic forms,
  representations and $L$-functions (Proc. Sympos. Pure Math., Oregon State
  Univ., Corvallis, Ore., 1977), Part 2, Proc. Sympos. Pure Math., XXXIII,
  Amer. Math. Soc., Providence, R.I., 1979, pp.~247--289. \MR{MR546620
  (81i:10032)}

\bibitem[dJ93]{dejong-ppav}
Aise~J. de~Jong, \emph{The moduli spaces of principally polarized abelian
  varieties with {$\Gamma_0(p)$}-level structure}, J. Algebraic Geom.
  \textbf{2} (1993), no.~4, 667--688. \MR{1227472}

\bibitem[DP94]{MR1266495}
Pierre Deligne and Georgios Pappas, \emph{Singularit\'{e}s des espaces de
  modules de {H}ilbert, en les caract\'{e}ristiques divisant le discriminant},
  Compositio Math. \textbf{90} (1994), no.~1, 59--79. \MR{1266495}

\bibitem[ERX17]{MR3725733}
Matthew Emerton, Davide Reduzzi, and Liang Xiao, \emph{Unramifiedness of
  {G}alois representations arising from {H}ilbert modular surfaces}, Forum
  Math. Sigma \textbf{5} (2017), e29, 70. \MR{3725733}

\bibitem[FC90]{MR1083353}
Gerd Faltings and Ching-Li Chai, \emph{Degeneration of abelian varieties},
  Ergebnisse der Mathematik und ihrer Grenzgebiete (3) [Results in Mathematics
  and Related Areas (3)], vol.~22, Springer-Verlag, Berlin, 1990, With an
  appendix by David Mumford. \MR{1083353}

\bibitem[G\"01]{gortz-unitary}
Ulrich G\"{o}rtz, \emph{On the flatness of models of certain {S}himura
  varieties of {PEL}-type}, Math. Ann. \textbf{321} (2001), no.~3, 689--727.
  \MR{1871975}

\bibitem[G\"03]{gortz}
\bysame, \emph{On the flatness of local models for the symplectic group}, Adv.
  Math. \textbf{176} (2003), no.~1, 89--115. \MR{1978342}

\bibitem[GK19]{MR3989256}
Wushi Goldring and Jean-Stefan Koskivirta, \emph{Strata {H}asse invariants,
  {H}ecke algebras and {G}alois representations}, Invent. Math. \textbf{217}
  (2019), no.~3, 887--984. \MR{3989256}

\bibitem[Gol14]{MR3177267}
Wushi Goldring, \emph{Galois representations associated to holomorphic limits
  of discrete series}, Compos. Math. \textbf{150} (2014), no.~2, 191--228.
  \MR{3177267}

\bibitem[Gro65]{ega-4-ii}
A.~Grothendieck, \emph{\'{E}l\'ements de g\'eom\'etrie alg\'ebrique. {IV}.
  \'{E}tude locale des sch\'emas et des morphismes de sch\'emas. {II}}, Inst.
  Hautes \'Etudes Sci. Publ. Math. (1965), no.~24, 231. \MR{0199181 (33
  \#7330)}

\bibitem[Gro66]{MR0217086}
Alexander Grothendieck, \emph{\'{E}l\'ements de g\'eom\'etrie alg\'ebrique.
  {IV}. \'{E}tude locale des sch\'emas et des morphismes de sch\'emas,
  troisi\`eme partie.}, Inst. Hautes \'Etudes Sci. Publ. Math. (1966), no.~28,
  255. \MR{0217086 (36 \#178)}

\bibitem[Gro98]{gross-satake}
Benedict~H. Gross, \emph{On the {S}atake isomorphism}, Galois representations
  in arithmetic algebraic geometry ({D}urham, 1996), London Math. Soc. Lecture
  Note Ser., vol. 254, Cambridge Univ. Press, Cambridge, 1998, pp.~223--237.
  \MR{1696481}

\bibitem[Har66]{Hartshorne}
Robin Hartshorne, \emph{Residues and duality}, Lecture notes of a seminar on
  the work of A. Grothendieck, given at Harvard 1963/64. With an appendix by P.
  Deligne. Lecture Notes in Mathematics, No. 20, Springer-Verlag, Berlin-New
  York, 1966. \MR{0222093}

\bibitem[Har89]{MR997249}
Michael Harris, \emph{Functorial properties of toroidal compactifications of
  locally symmetric varieties}, Proc. London Math. Soc. (3) \textbf{59} (1989),
  no.~1, 1--22. \MR{997249}

\bibitem[Har90a]{harris-ann-arb}
\bysame, \emph{Automorphic forms and the cohomology of vector bundles on
  {S}himura varieties}, Automorphic forms, {S}himura varieties, and
  {$L$}-functions, {V}ol.\ {II} ({A}nn {A}rbor, {MI}, 1988), Perspect. Math.,
  vol.~11, Academic Press, Boston, MA, 1990, pp.~41--91. \MR{1044828
  (91g:11063)}

\bibitem[Har90b]{MR1064864}
\bysame, \emph{Automorphic forms of {$\overline\partial$}-cohomology type as
  coherent cohomology classes}, J. Differential Geom. \textbf{32} (1990),
  no.~1, 1--63. \MR{1064864}

\bibitem[He13]{he-normality}
Xuhua He, \emph{Normality and {C}ohen-{M}acaulayness of local models of
  {S}himura varieties}, Duke Math. J. \textbf{162} (2013), no.~13, 2509--2523.
  \MR{3127807}

\bibitem[Hen00]{MR1738446}
Guy Henniart, \emph{Une preuve simple des conjectures de {L}anglands pour
  {${\rm GL}(n)$} sur un corps {$p$}-adique}, Invent. Math. \textbf{139}
  (2000), no.~2, 439--455. \MR{1738446 (2001e:11052)}

\bibitem[HS12]{MR2918491}
Kaoru Hiraga and Hiroshi Saito, \emph{On {$L$}-packets for inner forms of
  {$SL_n$}}, Mem. Amer. Math. Soc. \textbf{215} (2012), no.~1013, vi+97.
  \MR{2918491}

\bibitem[HT01]{ht}
Michael Harris and Richard Taylor, \emph{The geometry and cohomology of some
  simple {S}himura varieties}, Annals of Mathematics Studies, vol. 151,
  Princeton University Press, Princeton, NJ, 2001, With an appendix by Vladimir
  G. Berkovich. \MR{MR1876802 (2002m:11050)}

\bibitem[Ill71]{MR0491680}
Luc Illusie, \emph{Complexe cotangent et d\'eformations. {I}}, Lecture Notes in
  Mathematics, Vol. 239, Springer-Verlag, Berlin-New York, 1971. \MR{0491680}

\bibitem[Jun18]{JunSu}
Su~Jun, \emph{Coherent cohomology of {S}himura varieties and automorphic
  forms}, preprint, 2018.

\bibitem[Kat73]{MR0447119}
Nicholas~M. Katz, \emph{{$p$}-adic properties of modular schemes and modular
  forms}, Modular functions of one variable, {III} ({P}roc. {I}nternat.
  {S}ummer {S}chool, {U}niv. {A}ntwerp, {A}ntwerp, 1972), Springer, Berlin,
  1973, pp.~69--190. Lecture Notes in Mathematics, Vol. 350. \MR{0447119}

\bibitem[Kis10]{MR2669706}
Mark Kisin, \emph{Integral models for {S}himura varieties of abelian type}, J.
  Amer. Math. Soc. \textbf{23} (2010), no.~4, 967--1012. \MR{2669706}

\bibitem[KM76]{MR0437541}
Finn~Faye Knudsen and David Mumford, \emph{The projectivity of the moduli space
  of stable curves. {I}. {P}reliminaries on ``det'' and ``{D}iv''}, Math.
  Scand. \textbf{39} (1976), no.~1, 19--55. \MR{0437541}

\bibitem[KMP16]{MR3569319}
Wansu Kim and Keerthi Madapusi~Pera, \emph{2-adic integral canonical models},
  Forum Math. Sigma \textbf{4} (2016), e28, 34. \MR{3569319}

\bibitem[Kot92]{MR1124982}
Robert~E. Kottwitz, \emph{Points on some {S}himura varieties over finite
  fields}, J. Amer. Math. Soc. \textbf{5} (1992), no.~2, 373--444. \MR{1124982
  (93a:11053)}

\bibitem[KR99]{KR}
Stephen~S. Kudla and Michael Rapoport, \emph{Arithmetic {H}irzebruch-{Z}agier
  cycles}, J. Reine Angew. Math. \textbf{515} (1999), 155--244. \MR{1717613}

\bibitem[Laf11]{MR2869300}
Vincent Lafforgue, \emph{Estim\'{e}es pour les valuations {$p$}-adiques des
  valeurs propres des op\'{e}rateurs de {H}ecke}, Bull. Soc. Math. France
  \textbf{139} (2011), no.~4, 455--477. \MR{2869300}

\bibitem[Lan12]{MR2968629}
Kai-Wen Lan, \emph{Toroidal compactifications of {PEL}-type {K}uga families},
  Algebra Number Theory \textbf{6} (2012), no.~5, 885--966. \MR{2968629}

\bibitem[Lan13]{MR3186092}
\bysame, \emph{Arithmetic compactifications of {PEL}-type {S}himura varieties},
  London Mathematical Society Monographs Series, vol.~36, Princeton University
  Press, Princeton, NJ, 2013. \MR{3186092}

\bibitem[Lan17]{Lan2016}
\bysame, \emph{Integral models of toroidal compactifications with projective
  cone decompositions}, International Mathematics Research Notices \textbf{11}
  (2017), 3237--3280.

\bibitem[Lip09]{lipman}
Joseph Lipman, \emph{Notes on derived functors and {G}rothendieck duality},
  Foundations of {G}rothendieck duality for diagrams of schemes, Lecture Notes
  in Math., vol. 1960, Springer, Berlin, 2009, pp.~1--259. \MR{2490557}

\bibitem[MFK94]{MR1304906}
D.~Mumford, J.~Fogarty, and F.~Kirwan, \emph{Geometric invariant theory}, third
  ed., Ergebnisse der Mathematik und ihrer Grenzgebiete (2) [Results in
  Mathematics and Related Areas (2)], vol.~34, Springer-Verlag, Berlin, 1994.
  \MR{1304906}

\bibitem[Mil83]{MR717596}
J.~S. Milne, \emph{The action of an automorphism of {${\bf C}$} on a {S}himura
  variety and its special points}, Arithmetic and geometry, {V}ol. {I}, Progr.
  Math., vol.~35, Birkh\"{a}user Boston, Boston, MA, 1983, pp.~239--265.
  \MR{717596}

\bibitem[Mil90]{MR1044823}
\bysame, \emph{Canonical models of (mixed) {S}himura varieties and automorphic
  vector bundles}, Automorphic forms, {S}himura varieties, and {$L$}-functions,
  {V}ol. {I} ({A}nn {A}rbor, {MI}, 1988), Perspect. Math., vol.~10, Academic
  Press, Boston, MA, 1990, pp.~283--414. \MR{1044823}

\bibitem[Mok15]{MR3338302}
Chung~Pang Mok, \emph{Endoscopic classification of representations of
  quasi-split unitary groups}, Mem. Amer. Math. Soc. \textbf{235} (2015),
  no.~1108, vi+248. \MR{3338302}

\bibitem[MP19]{MR3948111}
Keerthi Madapusi~Pera, \emph{Toroidal compactifications of integral models of
  {S}himura varieties of {H}odge type}, Ann. Sci. \'{E}c. Norm. Sup\'{e}r. (4)
  \textbf{52} (2019), no.~2, 393--514. \MR{3948111}

\bibitem[Nee96]{neeman-duality}
Amnon Neeman, \emph{The {G}rothendieck duality theorem via {B}ousfield's
  techniques and {B}rown representability}, J. Amer. Math. Soc. \textbf{9}
  (1996), no.~1, 205--236. \MR{1308405}

\bibitem[Pin90]{PINK}
Richard Pink, \emph{Arithmetical compactification of mixed {S}himura
  varieties}, Bonner Mathematische Schriften [Bonn Mathematical Publications],
  vol. 209, Universit\"at Bonn, Mathematisches Institut, Bonn, 1990,
  Dissertation, Rheinische Friedrich-Wilhelms-Universit\"at Bonn, Bonn, 1989.
  \MR{1128753}

\bibitem[PRS13]{PRS}
Georgios Pappas, Michael Rapoport, and Brian Smithling, \emph{Local models of
  {S}himura varieties, {I}. {G}eometry and combinatorics}, Handbook of moduli.
  {V}ol. {III}, Adv. Lect. Math. (ALM), vol.~26, Int. Press, Somerville, MA,
  2013, pp.~135--217. \MR{3135437}

\bibitem[PS16]{MR3512528}
Vincent Pilloni and Beno\^{\i}t Stroh, \emph{Cohomologie coh\'erente et
  repr\'esentations {G}aloisiennes}, Ann. Math. Qu\'e. \textbf{40} (2016),
  no.~1, 167--202. \MR{3512528}

\bibitem[Sat63]{satake}
Ichir\^{o} Satake, \emph{Theory of spherical functions on reductive algebraic
  groups over {${\mathfrak p}$}-adic fields}, Inst. Hautes \'{E}tudes Sci.
  Publ. Math. (1963), no.~18, 5--69. \MR{0195863}

\bibitem[Sha97]{MR1610812}
Freydoon Shahidi, \emph{On non-vanishing of twisted symmetric and exterior
  square {$L$}-functions for {${\rm GL}(n)$}}, Pacific J. Math. (1997),
  no.~Special Issue, 311--322, Olga Taussky-Todd: in memoriam. \MR{1610812}

\bibitem[Sor]{Sorensen}
Claus Sorensen, \emph{A patching lemma}, Stabilization of the trace formula,
  Shimura varieties, and arithmetic applications, Paris Book Project.

\bibitem[Ver69]{verdier-base-change}
Jean-Louis Verdier, \emph{Base change for twisted inverse image of coherent
  sheaves}, Algebraic {G}eometry ({I}nternat. {C}olloq., {T}ata {I}nst. {F}und.
  {R}es., {B}ombay, 1968), Oxford Univ. Press, London, 1969, pp.~393--408.
  \MR{0274464}

\end{thebibliography}

\end{document}